\documentclass[a4paper,11pt]{amsart}
\usepackage[utf8]{inputenc}

\usepackage[style=alphabetic, backend=biber ]{biblatex}
\usepackage{amsmath}
\usepackage{amssymb}
\usepackage{bbold}
\usepackage{mathtools}
\usepackage{stmaryrd}
\SetSymbolFont{stmry}{bold}{U}{stmry}{m}{n}
\usepackage{graphicx}
\usepackage{ esint }
\usepackage{bm}
\usepackage{relsize}
\usepackage{amsthm}
\usepackage[all]{xy}
\usepackage{color}
\usepackage{bm}
\usepackage{comment}
\usepackage{tikz}

\usepackage[framemethod=tikz]{mdframed}
\usepackage{lipsum}
\usetikzlibrary{patterns}
\usetikzlibrary{arrows,decorations.markings}
\usepackage{geometry}
\usepackage{comment}
\usepackage{float}
\usepackage{nicematrix}
\usepackage{standalone}
\usepackage{tikz}
\usepackage{tikz-3dplot}
\usetikzlibrary{arrows.meta}
\usetikzlibrary{patterns}
\usepackage{geometry}
\usepackage{xcolor}
\usepackage{scalerel}
\usepackage{hyperref}

\newtheorem{Example}{Example}

\newtheorem{Remark}{Remark}
\newtheorem{definition}{Definition}
\newtheorem{Theorem}{Theorem}
\newtheorem{Lemma}{Lemma}
\newtheorem{Proposition}{Proposition}
\newtheorem{Corollary}{Corollary}

\DeclarePairedDelimiterX\braket[2]{\langle}{\rangle}{#1 \delimsize\vert #2}
\DeclareMathAlphabet{\mathpzc}{OT1}{pzc}{m}{n}

\newcommand{\R}{\mathbb{R}}
\newcommand{\C}{\mathbb{C}}
\newcommand{\N}{\mathbb{N}}
\renewcommand{\ell}{\text{ell}_h}
\newcommand{\Op}{\text{Op}_h^w}
\newcommand{\supp}{\text{supp }}
\renewcommand{\d}{\mathrm{d}}
\newcommand{\jp}[1]{\left\langle #1 \right\rangle }
\renewcommand{\Re}{\operatorname{Re}}
\renewcommand{\Im}{\operatorname{Im}}

\renewcommand{\epsilon}{\varepsilon}
\author{Roméo Taboada }
\address{Laboratoire de Mathématiques d'Orsay, Universit{\'e} Paris-Saclay, 91405 Orsay cedex, France}
\email{romeo.taboada@universite-paris-saclay.fr}
\title{Long-time propagation of coherent states in a normally hyperbolic setting  }

\begin{document}
	
	\maketitle

	\begin{abstract}
		We present a method to find asymptotics for the evolution of coherent states (or Gaussian wavepackets with standard deviation $\sqrt{h}$) under semiclassical Schrödinger's equation for a given Hamiltonian. These results extend the work of Combescure and Robert, in which the evolution of coherent states can be approximated in the limit $h\to 0$ with deformed Gaussian wavepackets called squeezed coherent states. The description with squeezed states holds for times $t$ that can go to infinity as $h\to 0$, under the constraint $|t|\leq  |\log h|/(6\lambda_0)$ where $\lambda_0$ is the maximal Lyapunov exponent of the classical dynamics. The breakdown of this approximation at time $|\log h|/(6\lambda_0)$ is related to the bending of evolved wavepackets: once propagated states spread at a scale $h^{1/3}$, squeezed states no longer provide an appropriate description.
		\medbreak
		To obtain a representation of propagated states valid up to times $|t|\leq C|\log h|$ with a larger $C$ (for instance, up to Ehrenfest's time $|\log h|/(2\lambda_0)$ where spreading on macroscopic scales is allowed), we make additional assumptions on the flow $\Phi_t$ associated to the classical dynamics, imposing constraints on directions of elongation. Namely, we work in a neighborhood of a normally hyperbolic $\Phi_t$-invariant submanifold $K$, on which the dynamics is considered as slow in comparison with its transverse directions, along which $\Phi_t$ is assumed to be hyperbolic. In this context, we describe the propagated state as a WKB state in transverse directions and a squeezed state along $K$. This description emphasizes the fact that propagated states should no longer be thought of as microlocalized on a point, but rather on an isotropic submanifold (corresponding to transverse unstable directions). Guillemin, Uribe, and Wang presented a similar class of wavefunctions microlocalized on an isotropic submanifold.
	\end{abstract}
	
	\section{Introduction}
	\subsection{A few properties of coherent states}
	\subsubsection{Classical and quantum dynamics}
	The history of coherent states dates back to the early works of Schrödinger \cite{Schro} in 1926, where they were introduced to investigate links between quantum and classical mechanics. They minimize the uncertainty principle in the sense that they are as concentrated in position and momentum as a wave function can be. As a consequence, they were often used to describe the equivalent of a point in phase space.
	\medbreak
	On the one hand, classical dynamics describes the evolution of the position and momentum of a point $(x,\xi)\in \R^{2d}$ according to Hamilton's equations:
	\begin{equation*} 
		\begin{pmatrix} \dot{x}\\ \dot{\xi} \end{pmatrix}= \begin{pmatrix} \frac{\partial \mathpzc{p}}{\partial  \xi} \\ -\frac{\partial \mathpzc{p}}{\partial x} \end{pmatrix}\eqqcolon H_{\mathpzc{p}}(x,\xi),
	\end{equation*}
	for a (classical) Hamiltonian $\mathpzc{p}\colon\R^{2d}\to \R$.
	\newline 
	We will denote the time-t flow associated with the Hamiltonian vector field $H_{\mathpzc{p}}$ by $\Phi^t_{\mathpzc{p}}$. We will always assume that it is complete, in the sense that the flow is defined for all $t\in\R$.
	\medbreak
	One can check that we can obtain the usual Newton's equations for a point subject to forces derived from a potential $V$ by considering the following:
	\begin{equation*}
		\mathpzc{p}(x,\xi)=|\xi|^2+V(x).
	\end{equation*}
	\medbreak 
	On the other hand, quantum mechanics describes states with $L^2\left(\R^d\right)$ normalized functions $\psi$ with the interpretation of a density of probability for the square of its modulus. (The global phase of $\psi$ does not carry a meaningful interpretation, and as a consequence we will not try to compute them, adding a $e^{i\theta}$ where they might appear.)
	\medbreak
	Their evolution is described by introducing an operator called the (quantum) Hamiltonian $\mathpzc{p}^w(x,hD)$ obtained by quantizing functions $\mathpzc{p}$ (verifying symbol estimates) on $\R^{2d}$ thanks to Weyl's quantization (a process described, for instance, in \cite[Chapter 4]{Zwbook}). In the simpler setting of $\mathpzc{p}$ given by $|\xi|^2+V(x)$, we obtain the following differential operator:
	\begin{equation*} 
		-h^2 \Delta+V(x).
	\end{equation*}
	For more general $\mathpzc{p}$'s, the result of this process is a pseudodifferential operator, which is a (potentially) unbounded self-adjoint operator with domain on $L^2\left(\R^d\right)$.
	\medbreak
	The evolution of wavefunctions $\psi$ is then governed by Schrödinger's equation:
	\begin{equation*} 
		ih\frac{\partial \psi}{\partial t}=\mathpzc{p}^w \psi, \text{ for a given } \psi_0=\psi(t=0),
	\end{equation*}
	for which solutions are denoted by $\psi_t\eqcolon e^{-it\mathpzc{p}/h}\psi_0$. We will refer to the operator $e^{-it\mathpzc{p}/h}$ as the propagator of the equation.
	\subsubsection{Construction of coherent states}
	The designation of ``coherent states'' was popularized a few decades after Schrödinger's paper in Glauber's work on optical coherence \cite{Glaub}. A way to construct them is to study the Hamiltonian called the ``quantum harmonic oscillator''
	\begin{equation*} 
		\mathpzc{p}^w(x,hD)=-h^2 \Delta+|x|^2, \quad \text{coming from } \mathpzc{p}(x,\xi)=|\xi|^2+|x|^2,
	\end{equation*} 
	consider its ground state (eigenfunction of lowest eigenvalue)
	\begin{equation*} 
		\varphi_0(x)\coloneq \frac{1}{(\pi h)^{d/4}}\exp(-|x|^2/(2h)),
	\end{equation*} 
	and displace it in the phase space $T^*\R^d\simeq \R^{2d}$ using the Weyl-Heisenberg operators $\hat{T}_h(q,p)$:
	\begin{equation}\label{T rho}
		\hat{T}_h(q,p)u(x)\coloneq\exp\left(-\frac{i}{2h}q.p\right) \exp\left(\frac{i}{h}p.x\right) u(x-q), \text{ for } u\in L^2\left(\R^d\right)
	\end{equation} 
	to obtain $\varphi_{q,p}\coloneq\hat{T}_h(q,p)\varphi_0$.
	\newline
	We can easily see that $\varphi_{q,p}$ is concentrated around $q$ with width $\sqrt{h}$ and that its ($h-$rescaled) Fourier transform is concentrated around $p$ with width $\sqrt{h}$, hence minimizing the uncertainty principle.
	\medbreak
	Another important fact is the resolution of the identity obtained with coherent states; we can decompose any $u\in L^2\left(\R^d\right)$ the following way:	
	\begin{equation}\label{resol identity}
		u(x)=\frac{1}{(2\pi h)^{d}} \int\int \langle u, \varphi_{q,p} \rangle_{L^2} \varphi_{q,p}(x)\ \d q \d p, \quad x\in \R^d.
	\end{equation} 
	As our operators are linear, knowing how to propagate coherent states would be a first step towards understanding the evolution of general $L^2\left(\R^d\right)$ functions.
	\medbreak
	Coherent states were widely used in physics and generalized in many ways. For more information on the role of coherent states in quantum optics, one may consult, for instance, \cite{Klauder_Sudarshan} and for generalizations in mathematical physics \cite{Klauder_Skagerstam}.
	\newline
	A generalization that will be useful to us is the notion of ``squeezed state'', which allows for different uncertainties in position and momentum. Mathematically, one can obtain such states by considering Gaussians with arbitrary covariance as starting points (rather than the ground state $\varphi_0$):
	\begin{equation*} 
		\varphi_0^{(\Gamma)}(x)\coloneq(\pi h)^{-d/4} |\det(\Im \Gamma)|^{1/4} \exp\left(\frac{i}{2h} x.\Gamma x\right),
	\end{equation*} 
	where $\Gamma$ is a complex symmetric matrix verifying $\Im(\Gamma)>0$ (i.e. $\Gamma\in S\mathbb{H}_d$ the Siegel upper plane), and then consider all the states obtained by transporting them with $\hat{T}_h(q,p)$, $\varphi_{q,p}^{(\Gamma)}\coloneq\hat{T}_h(q,p)\varphi_0^{(\Gamma)}$.
	\newline
	The spread of these squeezed states is controlled by $\Gamma$, more precisely the uncertainty in position is now given by $\|\Im(\Gamma)^{-1/2}\| \sqrt{h}$ and in momentum by $\|\Im(\Gamma^{-1})^{-1/2}\| \sqrt{h}$.
	\medbreak
	These states are interesting to understand propagation of coherent states as they correspond to their evolution in some simple model cases, see below.
	\subsubsection{Exact propagation under elementary Hamiltonians}\label{elementary}
	The simplest examples of classical Hamiltonian $\mathpzc{p}_1$ one can take are linear ones 
	\begin{equation*} 
		\mathpzc{p}_1(x,\xi)= a\cdot x+b\cdot \xi, \text{ with } a,b\in \R^{d},
	\end{equation*} 
	which corresponds to constant vector fields. In this case, it is easy to see that the propagated coherent state can be obtained simply by using the Weyl-Heisenberg operator to move it along the classical dynamics ($\Phi^t_{\mathpzc{p}_1}(q,p)=(q+tb, p-ta)$):
	\begin{equation*} 
		e^{-it\mathpzc{p}_1^w/h} \varphi_{q,p}=e^{i\theta_t/h}\varphi_{\Phi^t_{\mathpzc{p}_1}(q,p)}.
	\end{equation*} 
	\medbreak
	Next, we can look at classical Hamiltonians $\mathpzc{p}_2$ that are quadratic in $(x,\xi)$, i.e. flows solutions of a linear ODE: \begin{equation*} 
		\begin{pmatrix}
			\dot{x}
			\\ \dot{\xi}
		\end{pmatrix}=M\begin{pmatrix}
			x \\ \xi
		\end{pmatrix}, \text{ with } M\in \mathfrak{sp}_{2d}(\R),
	\end{equation*} 
	whose solutions we can write as 
	\begin{equation*} \begin{pmatrix}
			x(t) \\ \xi(t)
		\end{pmatrix}=\underbrace{e^{tM}}_{\eqcolon\kappa_t}\begin{pmatrix}
			x_0 \\ \xi_0
		\end{pmatrix}\eqcolon \begin{pmatrix}
			A_t & B_t \\ C_t & D_t \end{pmatrix} \begin{pmatrix}
			x_0 \\ \xi_0
		\end{pmatrix},
	\end{equation*} 
	where $\kappa_t\in \text{Sp}_{2d}(\R)$ is a symplectic matrix.
	Quantum propagators associated with such Hamiltonians are metaplectic operators, denoted as $\mathcal{M}(\kappa_t)\coloneq e^{-it\mathpzc{p}_2^w/h}$.
	\medbreak
	This is where the introduction of squeezed states becomes handy: they form a class that is preserved under the quantum evolution $e^{-it\mathpzc{p}_2^w/h}$. Moreover, we can show that the center of the state still follows the classical trajectory but rotations and deformations in phase space might happen around the center:
	\begin{equation}\label{GammaKappa}
		e^{-it\mathpzc{p}_2^w/h} \varphi_{q,p}^{(\Gamma_0)}=e^{i\theta_t/h}\varphi_{\Phi^1_{\mathpzc{p}_2}(q,p)}^{(\Gamma_1)}
	\end{equation}
	with $\Gamma_1=(C_t+D_t\Gamma_0)(A_t+B_t\Gamma_0)^{-1}\in S\mathbb{H}_d$ ($\Gamma_0=iI_d$ corresponding to the standard coherent states). This last equality is the result of the action of symplectic matrices over the Siegel space, which is well-known generalization of the $SL_2(\R)$-action over Poincaré half-plane to symplectic geometry, see, for instance, \cite{Sieg},\cite{SiegThese}. Thanks to our computation on spreading of squeezed states, we claim that we can control the evolution of its width with the norm of the Jacobian matrix of the flow ($\kappa_t=e^{tM}$ in this setting).
	\medbreak
	This description can be extended to time-dependent quadratic Hamiltonians $\mathpzc{p}_2(t,x,\xi)$, the important property being that the solutions of the classical dynamics are given by a linear equation:
	\begin{equation}\label{assos_symp}
		\begin{pmatrix} x(t)\\ \xi(t) \end{pmatrix}= \kappa_t \begin{pmatrix} x_0 \\ \xi_0 \end{pmatrix},
	\end{equation}
	for $\kappa_t \in \text{Sp}_{2d}(\R)$ (we say that $\kappa_t$ is the symplectic matrix associated with $(\mathpzc{p}_2(s,\boldsymbol{\cdot}))_{s\in [0,t]}$). However, in this context, the notation $e^{-it\mathpzc{p}_2^w/h}$ no longer makes sense, so we will stick to a notation that refers to the propagator as a metaplectic operator $\mathcal{M}(\kappa_t)$.
	\medbreak
	Notice that in this section, we have described the propagated coherent state with only a few parameters and that their transformation is expressed only in terms of classical dynamics.
	\subsection{Propagation of coherent states in more general cases}
	\subsubsection{Approximation in the semiclassical limit}\label{CR method}
	For more general Hamiltonians $\mathpzc{p}$, we cannot hope to obtain an exact description as simple for propagated coherent states. However, we can try to describe the state by some simple approximation in the semiclassical limit $h\to 0$. In this particular limit, we can make proper sense of the connection between the quantum picture and the classical one, translating hypotheses in the classical world into information on the quantum wave function provided that $h$ is small enough.
	\medbreak
	In this article, we will work in this semiclassical setting, making assumptions on the classical dynamics, and we will look for a convenient asymptotic for the propagated coherent state. Although one can investigate this problem for times independent of $h$, we also want to consider times that can go to infinity as $h$ goes to $0$ (hence larger times in some sense), typically in a $t\leq C|\log h|$ fashion, more details on this later.
	\medbreak
	An idea that was successfully carried out in previous works was to approximate a general Hamiltonian $\mathpzc{p}$ with an order 2 function $\mathpzc{p}_{\text{approx}}$ along the classical trajectory $(q_t,p_t)=\Phi^t_{\mathpzc{p}}(q,p)$ (with starting point $(q,p)$ center of the coherent state):
	\begin{align}\label{papprox}
		\begin{split}
			\mathpzc{p}_{\text{approx}}(t,x,\xi)=\mathpzc{p}(q_t,p_t)+\left\langle\frac{\partial \mathpzc{p}}{\partial x}(q_t,p_t),(x-q_t)\right\rangle+&\left\langle \frac{\partial \mathpzc{p}}{\partial \xi}(q_t,p_t),(\xi-p_t)\right\rangle
			\\&+\frac{1}{2}\left\langle \begin{pmatrix}
				x-q_t \\ \xi-p_t
			\end{pmatrix}, \text{Hess } \mathpzc{p}(q_t,p_t) \begin{pmatrix}
				x-q_t \\ \xi-p_t
			\end{pmatrix}\right\rangle.
		\end{split}
	\end{align}
	We have seen in the previous section how to compute (exactly) the evolution of $\varphi_{q,p}$ under Schrödinger's equation associated with $\mathpzc{p}_{\text{approx}}^w$. The only thing left to do would be to control the error done in this approximation.
	\medbreak
	This method was present in the works of Hepp \cite{Hepp} and Heller \cite{Heller} and Hagedorn obtained a control of the remainder in \cite{Hag1} giving an equivalent in the limit $h\to 0$:  for any $|t|<T$ independent of $h$,
	\begin{equation*} 
		\left\|e^{-it\mathpzc{p}^w/h} \varphi_{q,p}- e^{i\theta_t/h}\hat{T}_h(\Phi^t_{\mathpzc{p}}(q,p)) \mathcal{M}(\kappa_t) \varphi_{0}\right\|_{L^2}\leq C_Th^{1/2}, 
	\end{equation*} 
	where $\kappa_t$ is the symplectic matrix given by the Jacobian matrix of the flow at the point $(q,p)$ see (\ref{assos_symp}).
	\medbreak
	In order to reach a higher precision, one would need to take into account higher order phenomena induced by the full expansion of the Hamiltonian $\mathpzc{p}$  along $(q_t,p_t)$ (instead of only order $2$ as in (\ref{papprox})). A method to do so is to add corrective terms to the approximation, using a larger class of states than squeeze states that would allow polynomials to appear next to the Gaussian factors:
	for a $P\in \C[X_1,\dots,X_d]$, we will write
	\begin{equation}\label{excited squeezed}
		\varphi^{(\Gamma,P)}_{0}(x)= (\pi h)^{-d/4} |\det(\Im \Gamma)|^{1/4} P\left(\frac{\Im(\Gamma)^{1/2}x}{\sqrt{h}}\right)\exp\left(\frac{ix\cdot\Gamma x}{h}\right).
	\end{equation}
	Such functions correspond to ``excited states'', they were studied for instance by Hagedorn in \cite{HagBonus}. The reason for this particular scaling is that we can roughly estimate their $L^2$ norm by 
	\begin{equation}\label{L2 norm gauss}
		\|\varphi^{(\Gamma,P)}_{q,p}\|_{L^2}\leq C N_\infty(P),
	\end{equation}
	where $N_\infty(P)$ is the sup norm of coefficients of $P$ and $C>0$ only depending on $\deg P$.
	\medbreak
	Using this new class of states, Hagedorn derived the full asymptotics in \cite{Hag3}, \cite{Hag4}: for any $N\in \N$ and $|t|<T$ independent of $h$,  we obtain
	\begin{equation}\label{CR}
		\left\|e^{-it\mathpzc{p}^w/h}\varphi_{q,p}-e^{i\theta_t/h}\sum_{n=0}^{N-1} h^{n/2} \varphi^{(\Gamma_t,P_t^n)}_{q_t,p_t}\right\|_{L^2}\leq C_T h^{N/2},
	\end{equation}
	where $\Gamma_t$ is associated with the Hamiltonian $(\mathpzc{p}_{\text{approx}}(s,\boldsymbol{\cdot}))_{s\in [0,t]}$ and $P_t^n$ is a polynomial of degree $3n$.
	\medbreak
	These articles give insights on how to treat finite time propagation, but longer times (e.g. times depending logarithmically on $h$) are still a challenge. In this direction, Combescure and Robert in \cite{Comb_Rob_art} and later Hagedorn-Joye \cite{Hag_Joye} explained how to precisely control the dependence on $t$ in this asymptotic, showing that we can bound the norms of polynomials $P_t^n$ the following way:
	\begin{equation}\label{control CR}
		N_\infty(P_t^n)\leq C \jp{t}^n \|\kappa_t\|^{3n}
	\end{equation}
	(recall that $\kappa_t$ is the Jacobian matrix of the flow and we want to consider $t$ depending logarithmically on $h$). A similar control for the remainder can be obtained (it is roughly given by the size of the next corrective term's $L^2$ norm, which we can infer from (\ref{L2 norm gauss}) and (\ref{control CR})).	 
	\medbreak
	As a consequence, they deduced that the asymptotics can still be used (needing more corrective terms to reach the same precision) for times that might depend on $h$ but smaller than a threshold $T_{\text{CR}}$ designed to guaranty $\|\kappa_t\|\leq C h^{-1/6+\epsilon}$ (i.e. each term of the expansion is effectively smaller than the previous one in $L^2$ norm).
	\medbreak
	The question now becomes: Is there some phenomenon happening at this time that would prevent us from obtaining a simple description of the propagated coherent state or is it simply a limitation induced by our choice of representation ?
	\subsubsection{Identification of the threshold, the role of Ehrenfest's time}\label{ellipse}
	An example of a well-known threshold in semiclassical analysis at which we have to take into account additional phenomena is Ehrenfest time.
	Generally speaking, Ehrenfest time tells us for how long we have to propagate (under classical dynamics) a cloud of points initially localized in phase space at scale $\sqrt{h}$ so that it reaches the macroscopic scale. As a consequence, for semiclassical results past this time, we should expect macroscopic effects to appear, even starting from microscopic data. This time can be expressed with the greatest Lyapunov exponent $\lambda_{\text{max}}$ of the system (provided that it is strictly positive), which controls the exponential expansion in phase space, and can be defined as:
	\begin{equation*} 
		T_{\text{Ehrenfest}}\coloneq\frac{1}{2\lambda_{\text{max}}} |\log h|.
	\end{equation*} 
	In this setting, the threshold obtained by Combescure-Robert would be $T_{\text{CR}}=T_{\text{Ehrenfest}}/3$: a shorter time than the Ehrenfest one, at which the maximum spread of the cloud of point is $h^{1/3}$; which is still microscopic. A priori, no new phenomena are perceived by the classical dynamics at $T_{CR}$, and it seems that this maximal time only comes from our choice of representation.
	\medbreak
	An explanation for this threshold is that squeezed states are associated with clouds of points that form ellipsoids: a description in terms of squeezed states as in the works of Combescure-Robert is bound to fail if the cloud of point is far from being an ellipse. Consider a case where the cloud of points spreads along a submanifold. Describing it as an ellipsoid is equivalent to approaching this submanifold by its tangent space at the center of the cloud of points. This approximation works as long as the cloud of point is not too stretched or too concentrated around the manifold.
	\begin{figure}[h]
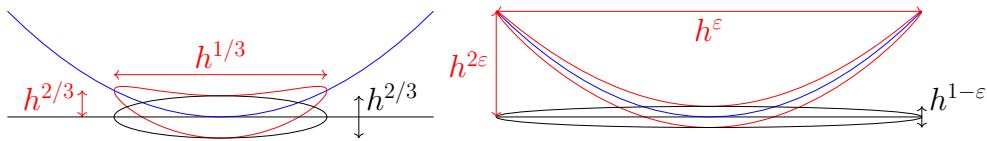

		\begin{center}
			\includegraphics[scale=0.7]{Cloud.tex}
			\includegraphics[scale=0.7]{Cloud2.tex}
		\end{center}
		\caption{Clouds of point spreading along a manifold and ellipsoids approximating them.}
	\end{figure}
	\medbreak
	In the above picture, the region in red delimits a cloud of points spread along a manifold pictured in blue. We represent in black its tangent space at the center of the cloud and draw an ellipsoid approximating the cloud of points. The approximation can only be valid as long as the vertical distance in red is smaller than the vertical distance in black. As the cloud of point does not change its volume, spreading over distances greater than $h^{1/3}$ along the manifold leads to the right-hand picture, for which the curvature must be taken into account to get a good description.
	\medbreak
	
	The goal of this article is to present a different description of propagated coherent states that would work until this Ehrenfest time, making use of additional hypotheses on classical dynamics. 
	\medbreak
	Before turning to the presentation of these assumptions, let us introduce a different class of wave functions for which time evolution is understood on these time scales: WKB states.
	\subsubsection{An alternative description with WKB states, properties, and limitations}\label{WKB description}
	WKB states (or Lagrangian states) are functions that can be written as $a(x)e^{i\frac{\phi_0(x)}{h}}$ where $a$ is usually a smooth compactly supported function and $\phi_0$ is a real-valued smooth function; they are a generalization of plane waves in the sense that they are located in phase space around the manifold
	\begin{equation}\label{Lambda_proj}
		\Lambda_0=\{(x,\nabla \phi_0 (x)), x\in \text{supp }a\}
	\end{equation}
	(while modulated plane waves like $a(x)e^{i\xi_0.x/h}$ are localized in some $\{(x,\xi_0), x\in\text{supp } a\}$). They were introduced to find asymptotics of Schrödinger's equation thanks to the WKB method, see \cite{Maslov}, \cite{Lax}, \cite{KellerBS}. An important property of these Lagrangian states is that we know how they behave under quantum evolution. 
	\medbreak
	To properly state this fact, one can assume that the image under the classical dynamics of the manifold $\Lambda_0$ (denoted $\Lambda_1$) is projectable in the $x$ variable (in the sense that it can be written as in the form (\ref{Lambda_proj}) or express the result after a change of representation (quantizing a symplectic change of coordinates making the projectability condition true). Then, the propagated state can be approximated by another Lagrangian state with manifold $\Lambda_1$ and whose amplitude is computed with transport equations \cite{VanVleck}, \cite{NZ_pression_topo}:
	\begin{equation}\label{propag_etat_lag}
		\left\|e^{-it\mathpzc{p}^w/h} a(x) e^{i\phi_0(x)/h}-\sum_{k=0}^{N-1} h^{k}b_k(x) e^{i\phi_1(x)/h} \right\|_{L^2}\leq C h^{N} \|a\|_{C^{2N}}.
	\end{equation} 
	\medbreak
	What we mean by WKB states is actually a generalization of this concept where we allow amplitudes $a$ to depend on $h$ but assume that their derivatives verifies so called ``$S_\delta$ symbol estimates'' for $\delta<1/2$ (see \cite[Section 4.4]{Zwbook}):
	\begin{equation}\label{Sdelta}
		|\partial^\alpha a(x)|\leq C_\alpha h^{-\delta |\alpha|}, \forall \alpha \in \N^{d}, \forall x\in \R^d.
	\end{equation}
	Intuitively, this assumption means that we can consider functions $a$ that might have a (essential) support of size $h^{\delta}$ and hence fluctuate rather quickly. However, we cannot deal with (essential) supports of size $\sqrt{h}$ or smaller.
	For $S_\delta$ amplitudes with $\delta<1/2$, we can still describe the quantum evolution with (\ref{propag_etat_lag}): the precision of the asymptotic, measured by the right hand term, can still be made as high as needed by taking larger values of $N$,
	\begin{equation*} 
		C h^{N} \|a\|_{C^{2N}}\leq C h^{N(1-2\delta)}, \text{ with } 1-2\delta>0.
	\end{equation*} 
	\medbreak
	At this point, we would like to see if we can write a coherent state as a WKB state and make use of this simple description. A first idea would be to isolate the imaginary part inside the exponential and consider it as the phase:
	\begin{equation*} 
		\varphi_0^{(\Gamma)}(x)= \underbrace{(\pi h)^{-d/4} |\det(\Im \Gamma)|^{1/4} e^{-x\cdot \Im \Gamma x/h}}_{a(x)} \underbrace{e^{ix\cdot \Re \Gamma x/(2h)}}_{e^{i\phi(x)/h}}.
	\end{equation*} 
	However, if we do this, we get 
	\begin{equation}\label{Reste Gamma qcq}
		\|\partial^{\alpha} a\|\leq C h^{-d/4} |\det(\Im \Gamma)|^{1/4} h^{-|\alpha|/2}\|(\Im \Gamma)^{1/2}\|^{|\alpha|}, 
	\end{equation} 
	which, for $\Im \Gamma$ independent of $h$, is insufficient to make the approximation (\ref{propag_etat_lag}) as precise as we want since the right hand term of (\ref{propag_etat_lag}) then is 
	\begin{equation*} 
		C h^{N}\|a\|_{C^{2N}}\leq Ch^{-d/4}.
	\end{equation*} 
	Nevertheless, if we were rather considering states $\varphi_0^{(\Gamma)}$ with, for instance, $\|(\Im \Gamma)^{1/2}\|<h^{\epsilon}$ with $\epsilon>0$, this argument would work: such squeezed states can be considered as WKB states and propagated as such since the right hand term becomes 
	\begin{equation*} 
		C h^{N}\|a\|_{C^{2N}}\leq C h^{-d(1-\epsilon)/4}  h^{2N\epsilon},
	\end{equation*} 
	which can be made arbitrarily small as $N$ goes to infinity.
	\medbreak
	Our idea is that under hypotheses on the classical dynamics, such states can be reached by propagating for a time $t_s \simeq \epsilon |\log h| \leq T_{\text{CR}}$ with the method of section \ref{CR method}. Having $\|\Im \Gamma^{1/2}\|<h^{\epsilon}$ requires all its eigenvalues to be of size $h^\epsilon$ or smaller, which in turn means that the state is expanded in $d$ directions and contracted in $d$ others.
	\medbreak
	This happens, for instance, for a hyperbolic fixed point, a toy-model corresponding to a quadratic hamiltonian given by $\mathpzc{p}_2(x,\xi)=x\cdot\xi$. In this case, the computations can be carried out explicitly, as explained in (\ref{GammaKappa}), resulting in the following:
	\begin{equation*} 
		e^{-it\mathpzc{p}_2^w/h} \varphi_{q,p}^{(\Gamma_0)}=(\pi h)^{-d/4} e^{-td/2} \exp(-e^{-2t}\|x\|^2/h),
	\end{equation*} 
	hence $\Gamma=ie^{-2t}I_d$, which, for $t=\epsilon |\log h|$ gives $\|\Im \Gamma\|=h^{2\epsilon}$. 
	\medbreak
	However, such a behavior is quite rare in cases coming from physics. A more reasonable model is having a hyperbolic closed trajectory.
	Say that there exists a hyperbolic periodic orbit $\gamma_0$ at energy $E_0$, then, by a stability argument, we can show that there exists a $\delta>0$ such that for any energy $E\in [E_0-\delta,E_0+\delta]$, there exists a close-by hyperbolic periodic orbit with energy $E$. Let us study the propagation of a coherent state centered on a point of the periodic orbit $\gamma_0$. 
	\medbreak
	In this new model, we have two dimensions that behave differently than in the previous example, called central directions. They correspond to movement along the periodic orbit (called time or flow direction) and change of considered periodic orbit (i.e. moving to another periodic one with a different energy, this direction is called the energy one). After a potential symplectic change of coordinate to better fit this geometry, one is led to a matrix $\Im\Gamma$ with $d-1$ eigenvalues reaching the $h^\epsilon$ scale by propagating for a time $t_s$ (directions behaving similarly to the previous example), but the remaining one would grow at most logarithmically with $h$. This logarithmic growth is the leading term when computing $\|\Im \Gamma\|$, which prevents us from making the remainder in (\ref{propag_etat_lag}) arbitrarily small with (\ref{Reste Gamma qcq}).
	\medbreak
	\medbreak
	In this article, we introduce a larger class of wavefunctions which allow for a $S_\delta$-WKB behavior in unstable directions, squeezed state behavior in the central ones. Indeed, such functions seems adequate to describe the situation as, in the central directions, we have very little expansion: we can still use a squeezed state description like Combescure-Robert's. This new class (see definition \ref{def S_delta,nu}) is in some sense a generalization of the ``isotropic semiclassical functions'' introduced by Guillemin, Uribe and Wang in \cite{GUW1}, \cite{GUW2} that allows for a squeezed state with $h$ dependent $\Gamma$'s in the central directions rather than only usual coherent state.
	\medbreak
	In the next section, we present assumptions that let us treat more complicated cases where unstable/stable manifolds are of lower dimension than $d-1$.
	\medbreak
	The fact that, in a chaotic setting, a propagated coherent state should look like a WKB state was already present in the physics literature with the work of Maia and al. \cite{Maia_phy}, identifying the spread along unstable manifolds as evidence to support this idea.
	From this observation, a different scheme for propagation of coherent states up to Ehrenfest time was investigated by Schubert, Vallejos and Toscano in \cite{Schub_phy}. They obtained a uniform description of the propagated coherent state showing the transition from a strongly localized state to a localization along a manifold. Their idea was to write coherent states as WKB ones from the start and modify the time-dependent WKB method. Indeed, they introduced a metaplectic correction to take into account the quantum dispersion, which is neglected when using the usual WKB method.
	\subsection{Assumptions and Statement of results}\label{hyp}
	\subsubsection{Assumptions on classical dynamics}
	As explained above, our main hypotheses concern the classical dynamics, i.e. the flow of $\mathpzc{p}$'s hamiltonian vector field denoted by $\Phi^t$. The first hypothesis is 
	\medbreak
	\textbf{Hypothesis 0}
	\begin{equation*}\tag{H0}\label{H0}
		\text{ The classical flow } \Phi^t \text{ is complete and  } \mathpzc{p}\in S^0(m) \text{ for } m \text{ an order function, }
	\end{equation*}
	which is mandatory to make sense of $\Phi^t(\rho)$ for any $\rho,t$ and use semiclassical analysis to study propagators. In the Newtonian setting, one can make assumptions on $V$ so that this is verified.
	\medbreak 
	Secondly, we will assume that
	\medbreak
	\textbf{Hypothesis 1}
	\begin{equation*}\tag{H1}\label{H1}
		\begin{split}
			\text{There is a compact }&\text{symplectic manifold } K^{\delta} 
			\\ \text{ transversally to which the cl}&\text{assical flow is normally hyperbolic, }
		\end{split}
	\end{equation*}
	which is a more precise assumption on the local behavior of classical trajectories. Let us detail this assumption.
	\medbreak
	We fix a small window of energy $[E_0-2\delta,E_0+2\delta]$ in which we will be working and center it at $0$ (considering $\mathpzc{p}-E_0$ instead of $\mathpzc{p}$). Then, for each energy $E$ in this window, we shall assume that we dispose of a flow-invariant compact $K_{E}\subset p^{-1}(E)$ in the neighborhood of which we wish to study the evolution of coherent states. We assume that the union of these compacts $K^{2\delta}= \bigcup_{E\in [-2\delta,2\delta]} K_E $ forms a symplectic submanifold of dimension $2d_{\scalerel*{\parallel}{\perp}}$ (called central manifold).
	\newline
	Transversally to this manifold, the dynamic is normally hyperbolic, meaning that there exist subbundles $\mathcal{V}^{u/s}$ (referred to as transverse directions) of $T_{K^{2\delta}}(\R^{2d})$ of dimension $d_\perp$ such that:
	\begin{equation}
		T_{\rho}(\R^{2d})=T_{\rho}K^{2\delta}\oplus \mathcal{V}^{u}_{\rho}  \oplus \mathcal{V}^{s}_{\rho}, \text{ for all } \rho\in K^{2\delta},
	\end{equation}
	where $\d\Phi^t$ preserves the subbundles and there exist $C>0$ and $\nu>0$ such that for all $\rho\in K^{2\delta},$
	\begin{equation*} 
		\|\d_\rho\Phi^t(v)\|\leq C e^{-\nu |t|}\|v\|, \quad \forall v\in \mathcal{V}^{u/s}_{\rho}, \quad \mp t\geq 0.
	\end{equation*} 
	\medbreak
	Such hypotheses appear in various settings and were, for instance, used in quantum chemistry see \cite{chimie} or in general relativity \cite{WZ}. Normal hyperbolicity can be understood as a generalization of hyperbolic closed trajectories, giving an interesting setting to study quantum resonances, see \cite{GeSj2}, \cite{Steph}, \cite{Dya}.
	\medbreak
	It turns out that we will need furthermore assumptions on the flow restricted to the set $K^\delta$: more precisely, we may assume 
	\medbreak
	\textbf{Hypothesis 1'}
	\begin{equation*}\tag{H1'}\label{H1'}
		\text{ The classical flow is } r\text{-normally hyperbolic, with } r\geq 3.
	\end{equation*}
	This means that expansion along the central manifold will be a lot slower than the expansion in the normal direction in the following sense:
	If $\nu_\text{min}^\perp$ is the supremum of $\nu$'s such that there exists $C>0$ verifying 
	\begin{equation}\label{nu_min}
		\sup_{\substack{\rho\in K^{2\delta} \\ p(\rho)\in[-\delta,\delta]}} \| \d_\rho \Phi^{\mp t} \vert_{\mathcal{V}^{u/s}_{\rho}}\|\leq C e^{-\nu t},\quad t>0,
	\end{equation}
	and $\lambda^{\text{c}}$ verifying
	\begin{equation}\label{lambda_c}
		\lambda^{\text{c}}=\limsup_{t\to +\infty} \sup_{\rho\in K^{2\delta}} \frac{1}{t} \log(\| d\Phi^t(\rho)\vert_{T_{\rho}K^\delta}\|),
	\end{equation}
	then we should have
	\begin{equation*} 
		\nu_\text{min}^\perp>3\lambda^{\text{c}}.
	\end{equation*} 
	A case of particular interest would be the closed hyperbolic trajectory setting mentioned above, where $K^{2\delta}$ is $2$-dimensional (each $K_E$ being a closed hyperbolic trajectory) and if $\rho\in K_E$, the central direction is given by the Hamiltonian vector field $H_\mathpzc{p}(\rho)$ and an energy direction,
	which verifies the $r$-normal hyperbolicity for any $r>0$ (in this case, we have $\lambda^{\text{c}}=0$).
	\medbreak
	An important consequence of these hypotheses is the existence of transverse stable/unstable manifolds to each point $\rho\in K^{2\delta}$ (denoted by $(W^{u/s}_\rho)_{\rho \in K^{2\delta}}\}$). They are of regularity $C^\infty$, depend in an $\alpha-$Hölder way on the base point $\rho$, the subbundles $\mathcal{V}^{u/s}$ are tangent to them, and they can be characterized by points sharply asymptotic to $\rho$ (i.e. points converging to $\rho$ at a speed at least $Ce^{-t\nu}$ for $\nu\in (\lambda^c,\nu^\perp_{\text{min}})$) \cite[Theorem 4.1]{HPS}, \cite[Theorem 5.6.1]{Wiggins}.
	\medbreak
	For more information on this refinement, one might consult \cite{Fenichel}, \cite{HPS}, \cite{Wiggins},\cite{GeSj2}, \cite{Dyalong}.
	\medbreak
	With the above hypotheses, we wish to understand the propagation of coherent states that are initially centered on $K^{\delta}$. However, we also want to describe how coherent states centered not exactly on $K^{\delta}$ but close enough will propagate. 
	\medbreak
	We are looking for such a description for times that can go beyond the Ehrenfest one: the evolved state can reach macroscopic scale. However, we are only interested in describing it close to $K^\delta$; once a portion of the state leaves this neighborhood, we will not be able to say much about it. As a consequence, we need to ensure that such portions cannot return to the vicinity of $K^\delta$, because we will not know how to describe them. From a classical point of view, this means that every point that goes out of the vicinity of $K^\delta$ cannot return there. To ensure this, we need to add extra assumptions on the classical dynamics on a global level:
	\medbreak
	\textbf{Hypothesis 2a}
	\begin{equation*}\tag{H2a}\label{H2a}
		K^{2\delta} \text{ is the (compact) trapped set of the classical flow for energy levels in }  [-2\delta, 2\delta],
	\end{equation*}
	i.e. 
	\begin{equation}\label{trapped_set}
		K^{2\delta}=\{\rho\in \R^{2d}, \mathpzc{p}(\rho)\in (-2\delta,2\delta),\ (\Phi^t(\rho))_{t\in \R} \text{ is bounded}\},
	\end{equation}
	and
	\medbreak
	\textbf{Hypothesis 2b}
	\begin{equation*}\tag{H2b}\label{H2b}
		\text{Non-trapping near infinity for energy levels in }  [-2\delta ,2\delta],
	\end{equation*}
	i.e. there exist a real-valued function $G\in C^\infty(\R^{2d})$, called escape function, $C>0$ and $\mathbf{U_0}$ an open neighborhood of $K^{2\delta}$ in $\R^{2d}$ such that 
	\begin{equation}
		H_pG\geq C \text{ on } p^{-1}(-2\delta,2\delta)\cap \mathbf{U_0}^c.
	\end{equation}
	This hypothesis (\ref{H2b}) may seem redundant with (\ref{H2a}) as assumption (\ref{H1}) also requires that $K^{2\delta}$ is compact. However, having a compact trapped set for energies $E\in(-2\delta,2\delta)$ does not seem to be enough to prevent the complication described above. Hypothesis (\ref{H2b}) can be used to obtain a neighborhood of infinity from which a trajectory cannot escape once it is reached.
	\medbreak
	These hypotheses are traditionally used to ensure the existence of quantum resonances, see \cite[Chapter 8]{HS_non_trapping}, giving an elliptic behavior near infinity to the operator. We mainly use them in section \ref{fuite_def_esc} to show that no portion of the state comes back close to $K^\delta$ once it has left its neighborhood.
	\subsubsection{Propagation of coherent states centered on $K^\delta$}
	Our description will rely on a choice of coordinates in which we can simply describe the interactions between the central and transverse directions.
	\medbreak
	Let us introduce different types of charts:
	\begin{definition}\label{K adapted}
		We say that a symplectomorphism $\kappa_\alpha:U_\alpha\subset\R^{2d} \to V_\alpha\subset\R^{2d}$ (with $U_\alpha$ neighborhood of a point of $K^\delta$ and $V_\alpha$ neighborhood of $0$) is \textbf{adapted to} $\boldsymbol{K^\delta}$ if its coordinates  $(x,y,\xi,\eta)$ verify 
		\begin{enumerate}
			\item $\kappa_\alpha(K^\delta\cap U_\alpha)=\{(x,0,\xi,0), x,\xi\in \R^{d_{\scalerel*{\parallel}{\perp}}}\}\cap V_\alpha$.
			\item For each $\rho\in K^\delta\cap U_\alpha$ identified with some $(\underline{x},0,\underline{\xi},0)$, the subspace $d\kappa_\rho(E^u_\rho)$ is $\epsilon$-close to $\{(\underline{x},y,\underline{\xi},0), y\in \R^{d_\perp}\}$.
			\item For each $\rho\in K^\delta\cap U_\alpha$ identified with some $(\underline{x},0,\underline{\xi},0)$, the subspace $d\kappa_\rho(E^s_\rho)$ is $\epsilon$-close to $\{(\underline{x},0,\underline{\xi},\eta), \eta\in \R^{d_\perp}\}$.
		\end{enumerate}
	\end{definition}
	\begin{definition}\label{coord}
		For a  point $\rho\in K^\delta$, we say that a symplectomorphism
		$\kappa_\rho:U_\rho\subset \R^{2d}\to V_\rho \subset \R^{2d}$ with $U_\rho$ neighborhood of $\rho$ and $V_\rho$ neighborhood of $0$ defines \textbf{adapted coordinates to the point} $\rho$ if its coordinates $(x,y,\xi,\eta)$, verify the following facts:
		\begin{enumerate}
			\item $\kappa_\rho(\rho)=(0,0,0,0)$.
			\item The set $K^\delta$ verifies $\kappa_\rho(K^\delta\cap U_\rho)=\{(x,0,\xi,0),\quad (x,\xi)\in \R^{2d_{\scalerel*{\parallel}{\perp}}}\}\cap V_\rho$.
			\item The unstable manifold verifies $\kappa_\rho(W^u_\rho\cap U_\rho)=\{(0,y,0,0),\quad y\in \R^{d_\perp} \}\cap V_\rho$.
			\item The stable manifold verifies $d\kappa_\rho(E^s_\rho)=\{(0,0,0,\eta),\quad \eta\in \R^{d_\perp} \}$.
		\end{enumerate}
	\end{definition}

	The existence of such coordinates is explained in the Appendix \ref{Appendix coord}. 
	\medbreak
	To translate these changes of coordinates occurring in $\R^{2d}$ (hence on the classical side of the picture) into a change of representation of our wavefunctions $u\in L^2\left(\R^d\right)$, we use h-Fourier Integral Operators: operators quantizing canonical relations (of which symplectomorphisms are a special case). We will introduce the following
	\begin{equation*} 
		\mathcal{U}_\rho \text{ a h-Fourier integral operator associated with } \kappa_\rho,
	\end{equation*} 
	(respectively $\mathcal{U}_\alpha$ for a chart only adapted to $K^\delta$). For more information on these operators, one can consult \cite{Zwbook}, \cite{HormFIO},\cite{FIO}.
	
	\tdplotsetmaincoords{75}{25}
	\begin{figure}[H]
		\begin{center}
			\begin{tikzpicture}
				\fill[orange!20] (-3,2.7,0) rectangle (3,-2.7,0);
				\draw[orange!60!black] (-3,-2.6,0) node[below] {Chart domain for $\kappa_\rho$};
				\draw [domain=-3:3, samples=80, smooth,purple, very thick] 
				plot (1.5 ,\x,{sin(1.5*40)+sin(30*\x)});
				\draw [domain=-2.7:2.7, {Latex[scale=1.5]}-{Latex[scale=1.5]}, samples=80, smooth,purple, very thick] 
				plot (1.5 ,\x,{sin(1.5*40)+sin(30*\x)});
				\draw [domain=-2:2, samples=80, smooth,black, very thick] 
				plot (1.5 ,{\x},{sin(1.5*40)-2*sin(30*\x)});
				\draw [domain=-1:1, {Latex[scale=1.5, reversed]}-{Latex[scale=1.5,reversed]}, samples=80, smooth,black,very thick] 
				plot (1.5 ,{\x},{sin(1.5*40)-2*sin(30*\x)});
				\draw [domain=-3:3, samples=80, smooth,purple, very thick] 
				plot (-1.5 ,\x,{sin(-1.5*40)+3*sin(30*\x)});
				\draw [domain=-2:2, {Latex[scale=1.5]}-{Latex[scale=1.5]}, samples=80, smooth,purple, very thick] 
				plot (-1.5 ,\x,{sin(-1.5*40)+3*sin(30*\x)});
				\draw [domain=-1.5:1.5, {Latex[scale=1.5, reversed]}-{Latex[scale=1.5,reversed]}, samples=80, smooth,black,  very thick] 
				plot (-1.5 ,{\x},{sin(-1.5*40)-sin(30*\x)});
				\draw [domain=-2:2, samples=80, smooth,black,  very thick] 
				plot (-1.5 ,{\x},{sin(-1.5*40)-sin(30*\x)});
				\draw [domain=-4:4, samples=80, smooth,blue, very thick] 
				plot (\x ,0,{sin(\x*40)});
				\draw[blue] (3.5,-0.3,{sin(4*40)}) node[below ] {$K=\{(x,0,\xi,0)\}$};
				\draw[blue, thick] (-1.5,0,{sin(-1.5*40}) node[below left] {$\rho=(0,0,0,0)$};
				\draw[blue, thick] (-1.5,0,{sin(-1.5*40}) node {$\times$};
				\draw[black, thick] (-1.5 ,-2,{sin(-1.5*40)-sin(30*-2)}) node[left] {$W^s(\rho)=\{(0,0,0,\eta)\}$};
				\draw[purple, thick] (-1.5 ,3,{sin(-1.5*40)+3*sin(30*3)}) node[left] {$W^u(\rho)=\{(0,y,0,0)\}$};
				\draw[blue, thick]  (1.5 ,0,{sin(1.5*40)}) node[above right]
				{$\tilde{\rho} =(x_0,0,\xi_0,0)$};
				\draw[blue, thick]  (1.5 ,0,{sin(1.5*40)}) node
				{$\times$};
				\draw[purple, thick]  (1.5 ,-2.7,{sin(1.5*40)-sin(30*2.7)}) node[right]
				{$W^u(\tilde{\rho})\neq \{(x_0,y,\xi_0,0)\}$};
			\end{tikzpicture}
		\end{center}
		\caption{Coordinates associated with the point $\rho\in K$}
	\end{figure}
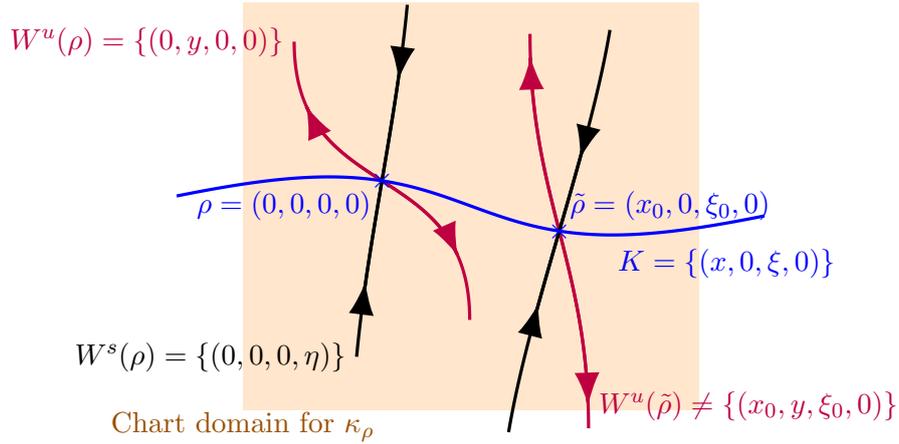
	\medbreak
	Finally, we introduce the largest Lyapunov exponent $\lambda_{\text{max}}$ to define the time scales on which our expansion will be valid:
	\begin{equation}\label{lambda_max}
		\lambda_{\text{max}}\coloneq\limsup_{t\to +\infty} \sup_{\rho\in K^\delta} \frac{1}{t} \log(\| d\Phi^t(\rho)\|).
	\end{equation}
	\medbreak
	In this article, we prove that we can write a propagated coherent state initially localized on $K^\delta$ in these particular charts as follows:
	
	\begin{Theorem}\label{th_K}
		Let $\epsilon\in (0, \frac{1}{6\lambda_{\text{max}}})$.
		\newline
		Under the hypotheses (\ref{H0}) and (\ref{H1'}), there exists $h_0>0$ such that for any $0<h<h_0$, any $\rho\in K^\delta$ and $\kappa_\alpha$ chart adapted to $K^\delta$ containing $\rho$ and for any time $t(h)\in [\epsilon|\log h|,(\frac{1}{2 \lambda_{\text{max}}}-\epsilon)|\log h|]$,
		\newline there exists a family of amplitudes $u_\gamma^{(t)}\in C^\infty(\R^{d_\perp},\C)$ for any $\gamma\in \N^{d-d_\perp}$ and a function
		$\Gamma_{\scalerel*{\parallel}{\perp}}^{(t)}\in C^\infty(\R^{d_\perp}, S\mathbb{H}_{d_{\scalerel*{\parallel}{\perp}}})$ such that for all $N\in \N$ and for all $(x,y)\in V_{\Phi^t(\rho)}$:
		\begin{equation*}
			[\mathcal{U}_{\Phi^t(\rho)} e^{-it\mathpzc{p}/h} (\mathcal{U}_\alpha)^* \varphi_{\kappa_{\alpha}^{-1}(\rho)}](x,y)=\sum_{\substack{\gamma\in \N^{d_{\scalerel*{\parallel}{\perp}}} \\|\gamma|\leq N-1}}h^{|\gamma |/2}  v_{\gamma}^{(t)}(x,y)
			+O_{L^2}(h^{N/2}(\sigma_c^{(t)})^{3N}),
		\end{equation*}
		where
		\begin{equation*}
			v_{\gamma}^{(t)}(x,y)\coloneq u_\gamma^{(t)}(y)|\det \Im\Gamma^{(t)}_{\scalerel*{\parallel}{\perp}}|^{1/4}\left(\frac{\Im(\Gamma_{\scalerel*{\parallel}{\perp}}^{(t)}(y))^{1/2}x}{\sqrt{h}}\right)^\gamma\exp\left(i\frac{x\cdot \Gamma_{\scalerel*{\parallel}{\perp}}^{(t)}(y)x}{2h}\right),
		\end{equation*}
		with $\sigma_c^{(t)}=\sup_{\rho\in K^\delta} \|d\Phi^t(\rho)\vert_{T_{\rho}K^\delta}\|$ if $\lambda^{\text{c}}>0$ (defined in (\ref{lambda_c})) and $e^{\alpha t}$ for some (small) $\alpha>0$ otherwise.
		\newline
		Moreover, we have the following estimates: there exists a $C>0$ independent of $\rho$ such that for any $\alpha\in \N^{d_\perp}$, for any $|\gamma|\leq N-1$, 
		\begin{enumerate}
			\item\label{pt 1} $\|\partial^{\alpha}\Gamma_{\scalerel*{\parallel}{\perp}}^{(t)}\|_\infty,\|\partial^{\alpha}(\Gamma_{\scalerel*{\parallel}{\perp}}^{(t)})^{-1}\|_\infty\leq C |\log h|^{|\alpha|} (\sigma_c^{(t)})^2$ and
			\newline
			$\|(\Im\ \Gamma_{\scalerel*{\parallel}{\perp}}^{(t)})^{-1/2}\|_\infty, \|(\Im\ (\Gamma_{\scalerel*{\parallel}{\perp}}^{(t)})^{-1})^{-1/2}\|_\infty\leq C(\sigma_c^{(t)})$.
			\item\label{pt 2} The region on which $u_\gamma^{(t)}$ is not $O(h^\infty)$ is of a volume smaller than $ch^{d_\perp/2-\cdot \ } J_u^{(t)}(\rho),$ 
			\newline with $J_u^{(t)}(\rho)
			=|\det (d\Phi^{t}(\rho)\vert_{\mathcal{V}_u(\rho)})|$. Moreover, the diameter of this region is smaller than $h^{1/2-\cdot\ } e^{t\lambda_{\text{max}}}$.
			\item\label{pt 3} $\|\partial_y^\alpha u_\gamma^{(t)}\|_{L^\infty(\d y)}\leq Ch^{-\delta_t|\alpha|} (J_u^{(t)}(\rho))^{-1/2} h^{-d/4},$
			with $\delta_t=\frac{1}{2}-\frac{t}{|\log h|}\nu_{\text{min}}^\perp+\cdot\ \in(0,\frac{1}{2})$. 
			\item\label{pt 4} As a consequence, we can show that 
			\begin{equation*} 
				\|\partial_y^\alpha v_\gamma^{(t)}\|_{L^2(\d x)L^\infty(\d y)} \leq C (\sigma_c^{(t)})^{3|\gamma|}  h^{-\tilde{\delta_t}|\alpha|} (J_u^{(t)}(\rho))^{-1/2} h^{-d/4},
			\end{equation*} 
			with $\tilde{\delta_t}=\max\left(\frac{1}{2}-\frac{t}{|\log h|}\nu_{\text{min}}^\perp+\cdot\ , 2\left(\frac{|\log \sigma_c^{(t)}|}{|\log h|}\right)+\cdot\ \right)\in(0,\frac{1}{2})$. 
		\end{enumerate} 
	\end{Theorem}
	The notation $\alpha+\cdot$ (resp. $\alpha-\cdot$) means that the estimate is true for any $\alpha+\tilde{\epsilon}$ (resp $\alpha-\tilde{\epsilon}$) with $\tilde{\epsilon}>0$.
	\begin{Remark} 
		\begin{itemize}
			\item The fact that the initial state is given by
			$(\mathcal{U}_\alpha)^* \varphi_{\kappa_{\alpha}^{-1}(\rho)}$ instead of simply $\varphi_\rho$ means that we still have a wavepacket centered at $\rho$ but the Gaussian is deformed so that the $x$ coordinates (and their Fourier dual) describe the function in the central direction, whereas the $y$ coordinates (and their Fourier dual) describe the transverse direction (unstable, resp. stable). This is not really restriction: for instance, when we want to study general $L^2$ functions $u$ on a manifold (in our case on $\R^{2d}$ with its submanifold $K^\delta$), we have to take charts before using the resolution of the identity (\ref{resol identity}). In the context of the theorem, we take one of these charts to be $\kappa_{\alpha}$. 
			\item For the final state, we also have the interpretation for the $x$ coordinates (and their Fourier dual) describe the function in the central direction, and the $y$ coordinates (and their Fourier dual) describe the transverse direction.
			\item The theorem only gives information for times larger than $\epsilon |\log h|$, for a description involving smaller time scales, one should use Combescure-Robert's theorem; see (\ref{CR}) and (\ref{control CR}) or proposition \ref{succ_1st} for a version in charts. Recall that this method consists in approximating the propagated state $[\mathcal{U}_{\Phi^t(\rho)} e^{-it\mathpzc{p}/h} (\mathcal{U}_\alpha)^* \varphi_{\kappa_{\alpha}^{-1}(\rho)}]$ by a sum of squeezed excited coherent states of the form
			\begin{equation*} 
				\varphi^{(\Gamma,P)}_{0}(x,y)= (\pi h)^{-d/4} |\det(\Im \Gamma)|^{1/4} P\left(\frac{\Im(\Gamma)^{1/2}(x,y)}{\sqrt{h}}\right)\exp\left(\frac{i(x,y)\cdot\Gamma (x,y)}{h}\right).
			\end{equation*} 
			On time scales for which the two descriptions coexist, they are compatible. From our choice of coordinates $\kappa_{\alpha}$ and $\kappa_{\Phi^t(\rho)}$, one can actually show that the matrix $\Gamma$ obtained is block-diagonal (with a block corresponding to the $x$ coordinates and one associated with the $y$'s). It turns out that the squeezed excited coherent states $\varphi^{(\Gamma,P)}_{0}$ appearing in the approximation all have a tensorial-product structure:
			\begin{equation}\label{tensor}
				u(y) \left|\det \Im(\Gamma_{\scalerel*{\parallel}{\perp}}) \right|^{1/4} h_{\gamma}^{{\scalerel*{\parallel}{\perp}}}\left(\frac{\Im(\Gamma_{\scalerel*{\parallel}{\perp}})^{1/2} x}{\sqrt{h}}\right) \exp\left(\frac{ix\cdot\Gamma_{\scalerel*{\parallel}{\perp}} x}{h}\right),
			\end{equation}
			with \begin{equation*} 
				u(y)=h^{-d/4}\left|\det \Im(\Gamma_\text{hyp}) \right|^{1/4} h_{\gamma}^{\text{hyp}}\left(\Im(\Gamma_\text{hyp})^{1/2}h^{-1/2} y\right) e^{\frac{i}{h} y\cdot \Gamma_\text{hyp} y },
			\end{equation*} 
			which fits the framework of theorem \ref{th_K} (with $\Gamma_{\scalerel*{\parallel}{\perp}}$ independent of $y$), see the proof of lemma \ref{1st_methodK} for more details.
			\item In this theorem, the WKB part of our state corresponds to $u_\gamma$, the propagated state is microlocalized on the unstable manifold to $\Phi^t(\rho)$, which is described in $\Phi^t(\rho)$'s coordinates by 
			\begin{equation}\label{iso_model}
				\mathcal{I}_m=\kappa_{\Phi^t(\rho)}(W^u(\Phi^t(\rho)))=\{(0,y,0,0), y\in \R^{d_\perp}\}\cap V_{\Phi^t(\rho)}.
			\end{equation}
			This manifold is isotropic, it corresponds to the model isotropic manifold introduced in \cite{GUW2}.
			\medbreak
			We prove estimates on the $L^\infty$ norm of $u_\gamma$, its derivatives and the size of their essential support (points \ref{pt 2} and \ref{pt 3} of this theorem), as they are the relevant quantities to control propagation of usual WKB states. 
			\item In central directions, we stick to $L^2$ norms (point \ref{pt 4}), as it is the space in which we control the remainder term. Recall that $L^2$ norms of excited squeezed states can be estimated with the norm of the associated polynomial, see (\ref{L2 norm gauss}). From the estimates on $\Im \Gamma_{\scalerel*{\parallel}{\perp}}^{(t)},\Im(\Gamma_{\scalerel*{\parallel}{\perp}}^{(t)})^{-1}$ in point \ref{pt 1}, we can see that the spread in these directions is governed by $\sigma_c^{(t)}$ which is the analogous of $\|\kappa_t\|$ that was involved in the Combescure-Robert threshold $T_{CR}$ see (\ref{control CR}) and below. Thanks to the $r$-normal hyperbolicity with $r\geq 3$, we can check that for times considered here, we always have 
			\begin{equation}\label{borne sigma c}
				\sigma_c^{(t)}\leq C h^{-1/6+\epsilon},
			\end{equation} as required for a squeezed state description, see the end of section \ref{CR method}.
			\item Notice that $\delta_t$ gets closer to $0$ as $t$ increases in point \ref{pt 3}: there is a smoothing effect of the amplitude in the unstable direction $u_\gamma^{(t)}$ due to expansion. However, when it comes to the smoothness in $y$ of the full state $v_\gamma^{(t)}$ in point \ref{pt 4}, we have to compare $\delta_t$ with a growing quantity: when computing derivatives in $y$ some of them can hit $\Gamma_{\scalerel*{\parallel}{\perp}}^{(t)}$ and have to be estimated in a different way, involving $\sigma^{(t)}_c$.
		\end{itemize}
	\end{Remark}
	\medbreak
	\subsubsection{Propagation of coherent states centered in a neighborhood of $K^\delta$}
	Let us fix some $\tau>0$ and introduce an arbitrary Riemannian metric in a neighborhood of $K^\delta$ in $\R^{2d}$. We consider $\rho$ a point at distance less than $h^{\tau}$ of $K^\delta$. We want to investigate the propagation of the coherent state $\varphi_\rho$.
	\medbreak
	We will only describe the portion of the propagated state that lands in the $\epsilon_1$ neighborhood of $K^{3\delta/2}$ (denoted by $K^{3\delta/2}(\epsilon_1)$), for some fixed $\epsilon_1>0$. We choose $\epsilon_1$ small enough so that $K^{3\delta/2}(\epsilon_1)$ is covered by charts $\kappa_\alpha$ that are adapted to $K^{3\delta/2}$ in the sense of definition \ref{K adapted}.
	\newline 
	More precisely, in order to only look at the portion of the propagated state located in $K^{3\delta/2}(\epsilon_1)$, we introduce a cutoff $\chi\in S_{0}$ supported in the set $K^{3\delta/2}(\epsilon_1)\cap p^{-1}(-3\delta/2,3\delta/2)$ such that $\chi=1$ on $K^{\delta}(\epsilon_1/2)\cap p^{-1}(-\delta,\delta)$.
	\medbreak
	Adding the trapped set hypotheses (\ref{H2a}) and (\ref{H2b}), we obtain a description valid up to a maximal time that only involves Lyapunov exponents of the central dynamics (\ref{lambda_c}); hence for times that might be much larger than in theorem \ref{th_K}:
	\begin{Theorem}\label{th_trapped_set}
		Consider the case $\lambda_c>0$ and fix $\epsilon>0$.
		\newline
		Under the assumption (\ref{H0}), (\ref{H1'}) and (\ref{H2a}), (\ref{H2b}), there exists a $h_0>0$ such that for any $0<h<h_0$, any $\rho$ in a $h^{\tau}$ neighborhood of $K^\delta$, $\tilde{\rho}$ reaching this distance, $\kappa_\alpha$ chart adapted to $K^\delta$ containing $\rho$, any time $t(h)\in [\epsilon|\log h|,(\min(\frac{\tau}{\lambda_c}, \frac{1}{6\lambda_c})-\epsilon)|\log h|]$  (or $[\epsilon |\log h|,C|\log h|]$ with any $C>0$ if $\lambda^{\text{c}}=0$) and $\kappa_\beta$ chart adapted to $K^\delta$ containing $\Phi^t(\tilde{\rho})$,
		\newline There exists a family of amplitudes $u_\gamma^{(t)}\in C^\infty(\R^{d_\perp},\C)$ for any $\gamma\in \N^{d-d_\perp}$, a function $\Gamma_{\scalerel*{\parallel}{\perp}}^{(t)}\in C^\infty(\R^{d_\perp}, S\mathbb{H}_{d_{\scalerel*{\parallel}{\perp}}})$ and $\overline{x}\phantom{}^{(t)},\overline{\xi}\phantom{}^{(t)}\in C^\infty(\R^{d_\perp},\R^{d_{\scalerel*{\parallel}{\perp}}})$, $\phi^{(t)}\in C^\infty(\R^{d_\perp},\R)$ such that for all $N\in \N$ and for all $(x,y)\in V_{\Phi^t(\tilde{\rho})}$:
		\begin{align} \label{DL th 2}
			\begin{split}
				[\mathcal{U}_{\beta}& \chi^w e^{-it\mathpzc{p}/h}  (\mathcal{U}_\alpha)^* \varphi_{\kappa_{\alpha}^{-1}(\rho)}](x,y)
				\\&=\sum_{\substack{\gamma\in \N^{d_{\scalerel*{\parallel}{\perp}}} \\|\gamma|\leq M(N)}}h^{|\gamma|/2} \exp\left(\frac{i}{h} (\phi^{(t)}(y)+\overline{\xi}\phantom{}^{(t)}(y) (x-\overline{x}\phantom{}^{(t)}(y)/2))\right)v_{\gamma}^{(t)}(x,y)+O_{L^2}(h^{N/2}(\sigma_c^{(t)})^{3N}),
			\end{split}
		\end{align}
		where 
		\begin{equation*} 
			v_{\gamma}^{(t)}(x,y)= u_\gamma^{(t)}(y)\left(\frac{\Im(\Gamma_{\scalerel*{\parallel}{\perp}}^{(t)}(y))^{1/2}(x-\overline{x}\phantom{}^{(t)}(y))}{\sqrt{h}}\right)^\gamma\exp\left(i\frac{(x-\overline{x}\phantom{}^{(t)}(y))\cdot \Gamma_{\scalerel*{\parallel}{\perp}}^{(t)}(y)(x-\overline{x}\phantom{}^{(t)}(y))}{h}\right),
		\end{equation*} 
		and $\sigma_c^{(t)}=\sup_{\rho\in K^\delta} \|d\Phi^t(\rho)\vert_{T_{\rho}K^\delta}\|$.
		\newline
		Moreover, $\overline{x}\phantom{}^{(t)},\overline{\xi}\phantom{}^{(t)},\phi^{(t)}$ describe an isotropic manifold 
		\begin{equation*} 
			\mathcal{I}^{(t)}=\biggl\{ \biggl(\overline{x}\phantom{}^{(t)}(y),y,\overline{\xi}\phantom{}^{(t)}(y),\underbrace{\nabla \phi^{(t)}(y)+\frac{(\nabla \overline{\xi}\phantom{}^{(t)}(y)).\overline{x}\phantom{}^{(t)}(y)-\overline{\xi}\phantom{}^{(t)}(y).(\nabla \overline{x}\phantom{}^{(t)}(y))}{2}}_{\coloneq\overline{\eta}\phantom{}^{(t)}(y)}\biggr), y\in D_{\epsilon_1}\biggr\},
		\end{equation*} 
		such that there exists $C>0$ independent of $t$ such that
		\begin{equation*} 
			\|\partial^\alpha \overline{x}\phantom{}^{(t)}\|_\infty,\|\partial^\alpha\overline{\xi}\phantom{}^{(t)}\|_\infty,\|\partial^\alpha \overline{\eta}\phantom{}^{(t)}\|_\infty\leq C, \text{ for any } \alpha\in \N^{d_\perp},
		\end{equation*} 
		and verifying \begin{equation*} 
			\mathcal{I}^{(t)}\subset \left[\kappa^\beta \Phi^{(t-t_\epsilon)} (\kappa^{\beta_{\epsilon}})^{-1}\right] \left(\mathcal{I}^{(t_\epsilon)}\right),
		\end{equation*} 
		where $t_\epsilon\coloneq\epsilon |\log h|$, $\kappa^{\beta_{\epsilon}}$ some chart adapted to $K^\delta$ containing $\Phi^{t_\epsilon}(\tilde{\rho})$ and $\mathcal{I}^{(t_\epsilon)}$ the isotropic manifold given by the theorem in that case.
		\newline
		Moreover, we have the estimates of points \ref{pt 1},\ref{pt 2},\ref{pt 3} and \ref{pt 4} as listed in theorem \ref{th_K} (with $J_u^{(t)}(\tilde{\rho})$ instead of $J_u^{(t)}(\rho)$).
	\end{Theorem}
	\begin{Remark} \label{rem th 2}
		This second theorem is quite similar to the first one (case $\rho\in K$), but there are a few differences:
		\begin{itemize}
			\item Just as in theorem \ref{th_K}, the description given by theorem \ref{th_trapped_set} is compatible with Combescure-Robert's. Despite no longer getting a clear tensor product structure as in equation (\ref{tensor}), we can get a approximate one, see lemma \ref{1st_methodK} and its proof for a more precise statement.
			\item In this new setting, the propagated coherent state is also microlocalized along an isotropic manifold $\mathcal{I}^{(t)}$ but now its expression is more complicated than (\ref{iso_model}). Directions of expansion are linked to the behavior of classical dynamics close to $\Phi^t(\rho)\notin K^\delta$ but we can only adjust our charts as in definition \ref{coord} for points on $K^\delta$. The result is that we need to use a more general description of the isotropic manifold $\mathcal{I}^{(t)}$ (which is still projectable on the $y$ variable).
			\item To elaborate further on the need for this new description, one could look at the expression of a propagated coherent state located at $\Phi^t(\rho)\in K^\delta$ (setting of theorem \ref{th_K}) in a different (yet close) adapted chart $\mathcal{U}_{\Phi^t(\tilde{\rho})}$. In that context, the isotropic manifold $\mathcal{I}^{(t)}$ is $W^u(\Phi^t(\rho))$ rather than $W^u(\Phi^t(\tilde{\rho}))$, hence the different description. In the general case (i.e. $\rho\notin K^\delta$), $W^u(\Phi^t(\rho))$ does not exist.
			\item In general, isotropic manifolds projectable in $y$ can be described using smooth functions $\tilde{x}, \tilde{\xi}\colon\R^{d_\perp} \to \R^{d_{\scalerel*{\parallel}{\perp}}}$ and $\tilde{\phi}:\R^d\to \R$ as follows
			\begin{equation*} 
				\left\{ \left(\tilde{x}(y),y,\tilde{\xi}(y),\nabla \tilde{\phi}(y)+\frac{(\nabla \tilde{\xi}(y)).\tilde{x}(y)-\tilde{\xi}(y).(\nabla \tilde{x}(y))}{2}\right), y\in D_{\epsilon_1}\right\}.
			\end{equation*} 
			Indeed, to obtain such a description, one should start from the parametrization 
			\begin{equation*} 
				\{(\tilde{x}(y),y,\tilde{\xi}(y),\tilde{\eta}(y))\},
			\end{equation*} 
			apply the symplectic transformation $f(x,y,\xi,\eta)= (x-\tilde{x}(y),y,\xi,\eta+\nabla \tilde{x}(y)\cdot \xi )$ to cancel the $x$ component
			and notice that any isotropic manifold with null $x$ component (projectable in $y$) can be written as $\{(0,y,\xi(y),\nabla\phi_0(y)), y\in D_{\epsilon_1}\}$ (with a smooth $\xi$). Coming back to $\mathcal{I}$ by inverting the symplectic transformation yields 
			\begin{equation*} 
				\mathcal{I}=\{\tilde{x}(y),y,\tilde{\xi}(y), \nabla \phi_0(y)-\nabla \tilde{x}(y) \cdot \tilde{\xi}(y), y\in D_{\epsilon_1}\},
			\end{equation*} 
			which gives the result by taking $\tilde{\phi}(y)=\phi_0(y)-\tilde{x}(y)\cdot \tilde{\xi}(y)/2$.
			\item Recall that in the first theorem, we had a description for a state microlocalized along the manifold $\mathcal{I}_m=\{(0,y,0,0), y\in D_{\epsilon_1}\}$ as some 
			\begin{equation*} 
				v^{(t)}(x,y)=u^{(t)}(y)P^{(t)}\left(\frac{\Im(\Gamma_{\scalerel*{\parallel}{\perp}}^{(t)}(y))^{1/2}x}{\sqrt{h}}\right)\exp\left(i\frac{x\cdot \Gamma_{\scalerel*{\parallel}{\perp}}^{(t)}(y)x}{h}\right).
			\end{equation*} 
			In order to know what a state microlocalized (at scale $h^{1/2+\cdot}\|\Im (\Gamma_{\scalerel*{\parallel}{\perp}}^{(t)})^{1/2}\|$) along $\mathcal{I}^{(t)}$ does look like, one should find a symplectomorphism mapping $\mathcal{I}_m$ to $\mathcal{I}^{(t)}$, quantize it into a Fourier Integral Operator $\hat{T}_{\mathcal{I}^{(t)}}$ and apply it to any $v^{(t)}$ given as above.
			\newline
			The symplectomorphism we are looking for was actually obtained above in the parametrization of $\mathcal{I}_m$: it is the composition $F_2\circ F_1$
			\begin{equation*} 
				F_1:(x,y,\xi,\eta)\mapsto (x,y,\xi+\overline{\xi}(y),\eta+\nabla \phi(y)+\nabla(\overline{x}(y)\cdot\overline{\xi}(y))/2),
			\end{equation*} 
			and 
			\begin{equation*} 
				F_2:(x,y,\xi,\eta)\mapsto (x+\overline{x}(y),y,\xi,\eta-\nabla\overline{x}(y)\cdot\xi).
			\end{equation*} 
			Note that if $\overline{x},\overline{\xi},\phi$ have small $C^1$ norms, $F_2\circ F_1$ is close to identity.
			\medbreak
			These can be quantized as the following h-Fourier Integral Operators:
			\begin{equation*} 
				\mathcal{T}_1u(x,y)=\exp\left( \frac{i}{h}\left(\overline{\xi}(y)x+\phi(y)+\frac{\overline{x}(y)\cdot\overline{\xi}(y)}{2}\right)\right)u(x,y),
			\end{equation*} 
			and 
			\begin{equation*} 
				\mathcal{T}_2 u(x,y)=u(x-\overline{x}(y),y),
			\end{equation*} 
			which yields the following h-Fourier Integral Operator:
			\begin{equation*}
				\hat{T}_{\mathcal{I}}u(x,y)=\mathcal{T}_2\mathcal{T}_1u(x,y)= \exp\left( \frac{i}{h} [(\overline{\xi}(y))\cdot(x-\overline{x}(y)/2)+\phi(y)]\right) u(x-\overline{x}(y),y).
			\end{equation*} 
			Then, applying this operator to a $v^{(t)}$ given as in theorem \ref{th_K} gives a state as in theorem \ref{th_trapped_set}.
			\begin{figure}[h]
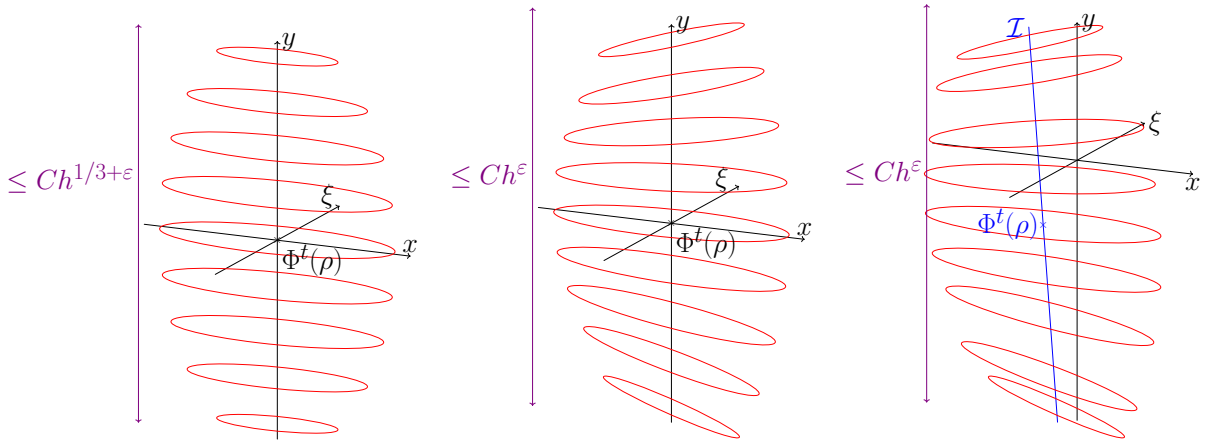

				\begin{center}
					\includegraphics[scale=0.42]{plat.tex}
					\includegraphics[scale=0.42]{CR.tex}
					\includegraphics[scale=0.42]{general.tex}
				\end{center}
				\caption{Schematic representation of the propagated state described by the different theorems. When we fix the value of $y$ (and $\eta$) and draw the region in the central direction ($x$ and $\xi$) in which the propagated state is not $O(h^\infty)$, we obtain the red ellipsoids drawn above.}
			\end{figure}
			\medbreak
			Left picture: Case described in section \ref{CR method} (previous method). In this case, the evolved state is described as a squeezed coherent state with a tensorial structure between the central and hyperbolic direction. As a consequence, all the red ellipsoids can be seen as slices of an ellipsoid in $(x,\xi,y,\eta)$ whose shape is determined by a matrix $\Gamma$ (acting on $x$ and $y$ coordinates) of the evolved state. The red ellipsoids all have the same shape, their surfaces shrinks as $y$ gets farther from $0$ due to Gaussian dependence in $y$.
			\medbreak
			Center picture: Case of theorem \ref{th_K}. With this new description designed for propagation with longer times, we lose the tensorial structure and the coherent state behavior in $y,\eta$. As a consequence, all the red ellipsoids can have different shapes linked to the different matrices $\Gamma_{\scalerel*{\parallel}{\perp}}(y)$ (acting on $x$ only). Nonetheless, all these ellipsoids are still centered at $x=0,\xi=0$, which is consistent with the microlocalization along $\mathcal{I}_m$ obtained in the theorem.
			\medbreak
			Right picture: Case of theorem \ref{th_trapped_set}. This description can be used to evolve coherent state that are not centered on
			$K^\delta$: in this picture $\rho\notin K^\delta$. The main difference with the center picture is that the red ellipsoids are centered
			on $\mathcal{I}$ rather than $\mathcal{I}_m$.
			\medbreak
			The diameter of every ellipsoid drawn in the previous pictures are bounded by $Ch^{1/3}$ (due to the bounds on $\sigma_c^{(t)}$ and hence on the $\Gamma_{\scalerel*{\parallel}{\perp}}$'s, see point \ref{pt 1}).
			\item Theorem \ref{th_trapped_set} sometimes tell us that the remainder is the predominant term. This might happen, for instance, if $\rho$ is too far away from the weak-stable manifold: assume that the ball of size $\sqrt{h}$ around $\rho$ does not intersect the weak-stable manifold to $K^\delta$, then all the points of this ball go to infinity as time grows (see figure \ref{Wsc}, case 2). As a consequence, we expect the entire propagated coherent to have left the $\epsilon_1$ neighborhood of $K^\delta$ for times larger than Ehrenfest's one (in the sense that theorem  \ref{th_trapped_set} would give us $\mathcal{U}_{\beta} \chi^w e^{-it\mathpzc{p}/h}  (\mathcal{U}_\alpha)^* \varphi_{\kappa_{\alpha}^{-1}(\rho)}=O(h^\infty)$).
			\medbreak
			However, whenever $\Phi^t(\rho)$ is in the domain of $\kappa_\beta$, we know that a portion of the propagated coherent state is still in $K^\delta(\epsilon_1)$ (see figure \ref{Wsc}, case 1) so that the estimate from point \ref{pt 3} of theorem \ref{th_trapped_set} is sharp for $\alpha=0$ and $\gamma=0$. Beware that even if $\Phi^t(\rho)$ is not in the domain of $\kappa_\beta$, there might be a portion of the propagated coherent state in $K^\delta(\epsilon_1)$, see figure \ref{Wsc} case 3 and the remark associated with (\ref{tau/lambda c}).
			\begin{figure}[h]\label{Wsc}
				\begin{center}
					\includegraphics[scale=0.8]{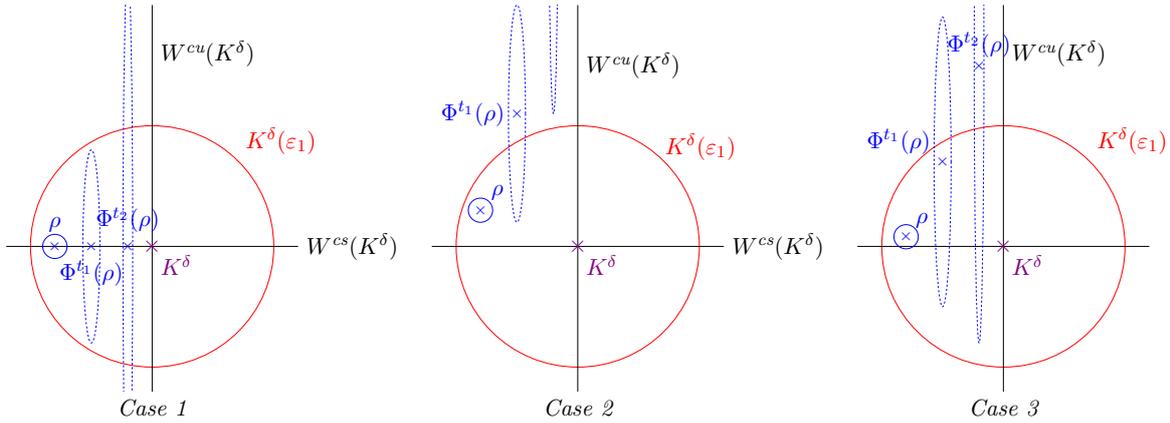}
				\end{center}
				\caption{Schematic representation of the evolution of a ball of radius $\sqrt{h}$ around $\rho$ (transversally to $K^\delta$).}
			\end{figure}
		\end{itemize}
	\end{Remark}

	\begin{Remark}
		Let us elaborate on the time scales considered.
		\begin{itemize}
			\item The first restriction is that we consider times 
			\begin{equation}\label{Ehren central/6}
				t\leq \left(\frac{1}{6\lambda_c}-\epsilon\right)|\log h|.
			\end{equation} 
			As can be observed in (\ref{DL th 2}), the limitation (\ref{Ehren central/6}) ensures the convergence of our expansion since it gives $\sigma_c^{(t)}\leq C h^{-1/6+\epsilon}$. Recall that this limitation also appeared in theorem \ref{th_K}, see (\ref{borne sigma c}) and the associated remark. 
			\item The second restriction is given by 
			\begin{equation}\label{tau/lambda c}
				t\leq \left(\frac{\tau}{\lambda_c}-\epsilon\right) |\log h|.
			\end{equation}
			To make sense of it, let us give some information on the proof of theorem \ref{th_trapped_set}. The theorem is actually obtained by expressing the propagated coherent state in a chart adapted to $\Phi^{t}(\tilde{\rho})\in K^\delta$, with $\tilde{\rho}$ being the point that reaches the distance $d(\rho,K^\delta)$, for example. The limitation (\ref{tau/lambda c}) guaranties that this chart is still relevant to describe the portion of the propagated state in $K^\delta(\epsilon_1)$ even if $\Phi^t(\rho)$ and $\Phi^{t}(\tilde{\rho})$ might get far apart. 
			\medbreak 
			Indeed, it is not obvious that the chart associated with $\Phi^{t}(\tilde{\rho})$ is relevant for our study, as initially $\rho$ is at a distance at most $h^\tau$ from $K^\delta$ and $t$ might be greater than $\frac{\tau}{\lambda_{\text{max}}} |\log h|$: $\Phi^t(\rho)$ might not even be in $K^\delta(\epsilon_1)$. It turns out that the relevant quantity to decide whether the chart adapted to $\Phi^{t}(\tilde{\rho})$ is appropriate to describe the propagated state is the distance between $\mathcal{I}_t$ and $\Phi^{t}(\tilde{\rho})$ rather than the one between $\Phi^{t}(\rho)$ and $\Phi^{t}(\tilde{\rho})$.
			The growth of this distance can be related to the drift in the central direction, it is bounded by $h^{\tau}e^{t\lambda_c}$ so that our assumption (\ref{tau/lambda c}) tells us that it remains smaller than $h^\epsilon$.
			\medbreak
			\begin{figure}[h]\label{dist It rhot}
				\begin{center}
					\includegraphics[scale=1]{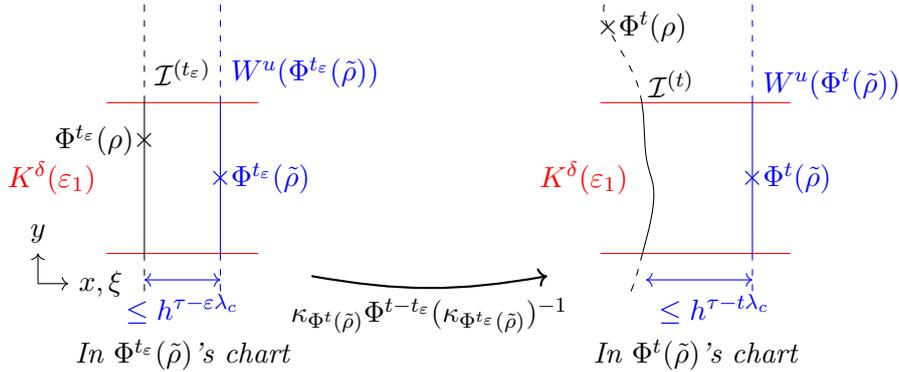}
				\end{center}
				\caption{Schematic representation of the evolution of the distance between the isotropic manifolds $\mathcal{I}^{(t)}$ and $\Phi^t(\tilde{\rho})$.}
			\end{figure}
			\medbreak
			Let us briefly explain why we need the distance $d(\mathcal{I}_t,\Phi^{t}(\tilde{\rho}))$ to be small. We have seen that our propagated coherent state is microlocalized along the isotropic manifold $\mathcal{I}^{(t)}$. The estimates on the functions $u_\gamma, \Gamma_{\scalerel*{\parallel}{\perp}}, v_\gamma$ that we are looking for should involve quantities for the classical dynamics just as in theorem \ref{th_K}, see points \ref{pt 1},\ref{pt 2},\ref{pt 3} and \ref{pt 4} of said theorem. But these dynamical quantities only make sense on $K^\delta$, not on $\mathcal{I}^{(t)}$. This is the reason why we need a close-by point $\Phi^t(\tilde{\rho})$, for which those quantities make sense, to compare with.
			\item If we stick to times smaller than $\min\bigl(\frac{1}{2}-\cdot \,\tau-\cdot\ \bigr)\frac{|\log h|}{ \lambda_{\text{max}}}$, we can remove the cutoff $\chi^w$ and still obtain the same description (in this case, the state stays on microscopic scale, adding $\chi^w$ or not only changes the result by a $O(h^\infty)$).
			\item In the case $\lambda_c=0$, the theorem is valid for times $t(h)\in[\epsilon |\log h|,C|\log h|]$ for any $C>0$, setting 
			$\sigma_c^{(t)}=e^{\alpha t}$ for $\alpha>0$ as in theorem \ref{th_K}.
		\end{itemize}
		
	\end{Remark}

	\subsection{Organization of the paper}
	After some reminders on coherent states and semiclassical analysis, in section \ref{section1st} we adapt the method of propagation introduced by \cite{Lucas} that worked up to $T_\text{CR}=T_{\text{Ehrenfest}}/3$ time: it consists in writing the propagator as a product of $n=C|\log h|$ Fourier Integral Operators, each of them propagating for a finite time, and compute their effect on our state while controlling the norms of the polynomials appearing.
	\medbreak
	While we will use this technique to propagate up to a time $\epsilon |\log h|$, as explained above, we will need to introduce an alternate description of our state, translating a mix between a squeezed state and a WKB state (see definition \ref{def S_delta,nu}), to propagate further. This new space can be used to describe the propagated coherent state in specific coordinates that describe exactly the transverse stable and unstable manifolds passing the center of our coherent state. Then, we explain a method for finite-time propagation of those states in section \ref{section2nd}, by doing a mix of stationary phase and expansion of the integral defining the Fourier Integral Operators.
	\medbreak
	Once methods for finite propagation are obtained, we control the repeated action of those on coherent states up to Ehrenfest time in section \ref{iter_method} and show that they can still be described as a mix of squeezed and WKB state in adapted coordinates.
	\section{Preliminaries}
	
	\textbf{Notations: }
	We will refer to the central direction composed of the $x$ and $\xi$ variables with a ``$\text{c}$'' index, while we will mention the hyperbolic coordinates $y$ and $\eta$ with a ``$\text{hyp}$'' index. For the points on which our coherent states will be centered, we will rather use the notation $\rho$ (and $\tilde{\rho}$ for the center of the charts) when referring to them in the ambient space $\R^{2d}$ and we will use bold symbols to refer to them in the charts of $\tilde{\rho}$:
	\begin{equation*} 
		\boldsymbol{\rho}=\kappa_{\tilde{\rho}}(\rho)=(q_x,q_y,p_x,p_y)\in \R^{2d}.
	\end{equation*} 
	
	\subsection{Classical dynamics (trapped set setting)}\label{preli_dyn_clas}
	In this section, we want to treat the case where $K^\delta$ is a trapped set, see (\ref{H2a}) and (\ref{H2b}). We aim at a control of our state for times at which it can be macroscopic in the transverse direction.
	\medbreak
	
	\subsubsection{Final exit lemma}\label{fuite_def_esc}
	The goal of the next two lemmas is to show that every departure from $K^\delta$ is definitive: no point starting near $K^\delta$ that moves away from $K^\delta$ can return in a neighborhood of $K^\delta$.
	\begin{Lemma}\label{def exit micro}
		Let $V_1$ be an open neighborhood of $K^\delta$ and $\rho\in p^{-1}([-\delta,\delta])\setminus (V_1\cup \mathbf{U_0}^c)$ where $\mathbf{U_0}$ is an open neighborhood of $K^{2\delta}$ . Then, there exist times $T_{\pm}=T_{\pm}(\rho),\ T_-<0<T_+$ such that 
		\begin{align*}
			\Phi^{T_+} (\rho)\in \mathbf{U_0}^c &\text{ or } \Phi^{T_-}(\rho) \in \mathbf{U_0}^c,
			\\\Phi^{T_\pm}(\rho) &\in V_1 \cup \mathbf{U_0}^c.
		\end{align*}
	\end{Lemma}
	The proof of this lemma is done in the proof of \cite[Lemma 4.6]{Steph}.
	\medbreak
	In addition to this lemma, we use classical constructions with escape functions to obtain a neighborhood of $K^\delta$ from which the departures are definitive. Let us first recall a result obtained by Gerard and Sjöstrand in the Appendix of \cite{GeSj}:
	\begin{Proposition}\cite[Proposition A.6.]{GeSj}\label{tilde G}
		Given our escape function $G$, we can construct a new smooth escape function $G_1$, equal to $G$ outside of a compact set, such that on $p^{-1}(-\delta,\delta)$, we have $G_1=0$ in a neighborhood of $K^\delta$ and such that locally uniformly on $p^{-1}(-\delta,\delta)$:
		\begin{equation*} 
			H_p(G_1)\geq C_0^{-1} |\nabla G_1|. 
		\end{equation*} 
	\end{Proposition}
	Hence the following 
	\begin{Corollary}\label{def exit macro}
		There exists $V_2$ a compact neighborhood of $K^\delta$ inside $p^{-1}(-\delta,\delta)$ such that if $t_1<t_2\in \R$ and $\rho\in p^{-1}(-\delta,\delta)$ are such that 
		\begin{equation*} 
			\Phi^{t_1}(\rho)\in V_2 \text{ and }  \Phi^{t_2}(\rho) \notin V_2,
		\end{equation*} 
		then 
		\begin{equation*} 
			\forall s\geq t_2, \Phi^{s} (\rho) \notin V_2.
		\end{equation*} 
	\end{Corollary}
	\begin{proof}
		Consider $V_2$ to be $\{G_1=0\}\cap (\mathbf{U_0}\cup W)$ with $\mathbf{U_0}$ the neighborhood of $K^{2\delta}$ defined in the assumption \ref{H2b} and $W$ the compact outside of which $\tilde{G}=G$ in proposition \ref{tilde G}.
	\end{proof}
	\medbreak
	We can now obtain the following
	\begin{Lemma}\label{def_exit}
		Let $\rho$ be a point in $K^\delta(\epsilon_1/2)$ and assume that $\Phi^t(\rho)$ also is in $K^\delta(\epsilon_1/2)$ for a $t>2t_0$, then $\forall s\in [t_0,t-t_0]$, $\Phi^{s}(\rho)$ is in $K^\delta(\epsilon_1/2)$.
	\end{Lemma}
	\begin{proof}
		Consider $U_0$ to be the complementary of $V_2$ mentioned in corollary \ref{def exit macro}, $V_1$ the $\epsilon_1/2$ neighborhood of $K^\delta$ and $t_0$ the supremum of all $|T_\pm(\rho)|$ for $\rho\in p^{-1}(-\delta,\delta)\setminus (V_1\cup \mathbf{U_0}^c)$ (which is finite as this region is compact).
		\medbreak
		By contradiction, assume that there exists some $s\in [t_0,t-t_0]$ such that $\Phi^{s}(\rho)\notin K^\delta(\epsilon_1/2)$. Let us find some $s_0\in[0,t]$ such that $\Phi^{s_0}(\rho )\in V_2$ then corollary \ref{def exit macro} will give the contradiction. 
		\medbreak
		If $\Phi^{s}(\rho )\in V_2$, we can conclude. If not, we can apply the lemma \ref{def exit micro} to $\Phi^{s}(\rho)$ and conclude that $\Phi^{s+T_+}(\rho)$ or $\Phi^{s+T_-}(\rho)$ is in $V_2$, which gives a $s_0\in[0,t]$ due to the definition of $t_0$.
	\end{proof}
	We write $t=n_{\text{max}}t_0+\delta$ with $n_{\text{max}}\in \N$ and $0\leq \delta<t_0$. As mentioned in the introduction, we will decompose the propagation for a time $t$ into $n_{\text{max}}$ propagations of duration $t_0$, or more precisely $n_{\text{max}}-1$ of duration $t_0$ and one of duration $t_0+\delta$. The fact that the last propagation is longer in duration does not affect the computations done below (as we can increase $t_0$ if needed). For the sake of simplicity, from now on, we will assume  that $t=n_{\text{max}}t_0$.
	\medbreak
	Finally, let us mention the construction of a refinement of the escape function in proposition \ref{tilde G} for which we can choose the region on which $H_pG>0$.
	\begin{Lemma}\cite[Lemma 4.6.]{Steph}\label{G2}
		Let $V$ be an open neighborhood of $K^{2\delta}$, then there exists $G_2\in C^\infty(\R^{2d})$ such that 
		\begin{itemize}
			\item $G_2=0$ on a neighborhood of $K^{2\delta}$,
			\item $H_pG_2 \geq C$ on $p^{-1}(-\delta,\delta)\setminus V$,
			\item $H_pG_2 \geq0$ on $p^{-1}(-\delta,\delta)$.
		\end{itemize}
	\end{Lemma}
	This lemma will become useful in section \ref{section reduction}, when we show the quantum analogues of lemma \ref{def_exit}
	\subsubsection{Study of the classical dynamics close to $K^\delta$}\label{section dyn class}
	In this section, we pick a point $\rho$ at a distance lower than $h^{\tau}$ of $K^\delta$. Our first objective is to choose the charts that will be used to express the propagated coherent state at an intermediate time $0\leq s\leq t$. Using charts that are adapted to some points on $K^\delta$ in the sense of definition \ref{coord}, rather than charts that are just adapted to $K^\delta$ in the sense of definition \ref{K adapted}, would make it easier to use the normal hyperbolicity assumption.
	\medbreak
	However, at a time $s$, we cannot use adapted coordinates to $\Phi^s(\rho)$ to express the state, as those only exist for points exactly on $K^\delta$. As a consequence, we will instead pick a point $\tilde{\rho}\in K^{\delta}$ that is initially close to $\rho$ and use the adapted coordinates to $\Phi^s(\tilde{\rho})$, base point which we expect to be close enough to $\Phi^s(\rho)$ so that these adapted coordinates are somewhat still relevant for the description.
	\medbreak
	Let us make this more precise, recall that we have assumed that $t=n_{\text{max}}t_0$ for a $n_{\text{max}}>0$ and $t_0$ defined in section \ref{fuite_def_esc}. We will discretize the time interval considered $[0,t_{\text{max}}]$ with multiples of $t_0$ and consider the following charts:
	\begin{definition}\label{nth chart}
		Consider a point $\rho$ at a distance lower than $h^{\tau}$ of $K^\delta$ and let $\tilde{\rho}$ be a point of $K^\delta$ minimizing this distance. We say that the \textbf{chart associated with} $\boldsymbol{\rho}$ \textbf{at time}  $nt_0$ is given by 
		\begin{equation*} 
			\kappa^{(n)}\coloneq\kappa_{\tilde{\rho}^n}:U_{\tilde{\rho}^n} \to V_{\tilde{\rho}^n}
		\end{equation*} 
		which a chart adapted to the point is associated with the point $\tilde{\rho}^n\coloneq\Phi^{nt_0}(\tilde{\rho})$ in the sense of definition \ref{coord}.
	\end{definition}
	We also assume, for now, that at a time $t_{\text{max}}=n_{\text{max}}t_0$, $\Phi^{t_{\text{max}}}(\rho)$ is in $K^{\delta}(\epsilon_1)$ so that we can apply the lemma \ref{def_exit}.
	\medbreak
	The next step is to study how far apart the points $\Phi^{nt_0}(\rho)$ and $\Phi^{nt_0}(\tilde{\rho})$ can be. As we will see in the beginning of section \ref{section1st}, we will need these estimates so that we can control error terms.
	\medbreak
	We recall that $\rho^n=\Phi^{nt_0}(\rho)$ and $\tilde{\rho}^n$ have the following coordinates in the charts:
	\begin{equation}\label{notation coord}
		\boldsymbol{\rho^n}\coloneq\kappa^{(n)}(\rho^n)\coloneq(q^n_x,q^n_y,p^n_x,p^n_y), \quad \boldsymbol{\tilde{\rho}^n}\coloneq\kappa^{(n)}(\tilde{\rho}^n)=(0,0,0,0).
	\end{equation}
	We are looking for bounds involving $\lambda^c$ in the central direction (\ref{lambda_c}) and $\nu_{\text{min}}$ (\ref{nu_min}) in the transverse ones. To do so, we may choose $t_0$ larger if necessary so that for $t\geq t_0,$
	\begin{equation} \label{lambda max et t0}
		\left|\sup_{\rho\in K^\delta} \frac{1}{t} \log(\| d\Phi^t(\rho)\|)-\lambda_{\text{max}}\right|\leq \epsilon_2/3, \text{ for a fixed } \epsilon_2>0,
	\end{equation}
	and similarly for $\lambda^c$ and $\nu_{\text{min}}$.
	\medbreak
	We want to prove by induction the following estimates:
	\begin{Lemma}\label{proof_dyn_class}
		Let $\rho$ be a point in a $h^{\tau}$ neighborhood of $K^\delta$ such that $\Phi^{n_\text{max}t_0}(\rho)\in K^\delta(\epsilon_1/2)$. Consider its iterates by the flow $\rho^n=\Phi^{nt_0}(\rho)$, with $n\leq n_\text{max}=\frac{\tau \log(1/h)}{t_0\lambda^c}$
		\medbreak
		Then $\boldsymbol{\rho^n}$'s coordinates defined in (\ref{notation coord}) verify:
		\begin{equation*} 
			\|q^n_x\|,\|p^n_x\|,\|p^n_y\|\leq (1+\epsilon_1) e^{(\lambda^{\text{c}}+2\epsilon_2/3)nt_0} h^{\tau}, 
		\end{equation*} 
		\begin{equation*} 
			\|q^n_y\|\leq (1+\epsilon_1)\max\bigl( e^{-(n_{\text{max}}-n)t_0(\nu_{\text{min}}^\perp-2\epsilon_2/3)},e^{(\lambda^{\text{c}}+2\epsilon_2/3)nt_0} h^{\tau}\bigr) .
		\end{equation*} 
	\end{Lemma}
	
	\begin{Remark}
		The important aspect of those estimates is to show that the central Lyapunov exponent is sufficient to control the dynamics along the central direction, $\lambda_\text{max}$ (defined in (\ref{lambda_max})) is never involved in the estimates. 
	\end{Remark}
	\begin{proof}
		\begin{itemize}
			\item We start by estimating roughly $\|q^n_y\|$ by $\epsilon_1/2$ with lemma \ref{def_exit} to deduce estimates on $\|q^n_x\|$, $\|p^n_x\|$, $\|p^n_y\|$, then we will study more closely $\|q^n_y\|$.
			\item 
			\textbf{Expansion in the transverse direction}
			\newline
			By definition, $\boldsymbol{\rho^{n+1}}=\kappa^{(n+1)}\Phi^{t_0}(\kappa^{(n)})^{-1}(\boldsymbol{\rho^n})=\boldsymbol{\Phi^{(n)}}(q^n_x,q^n_y,p^n_x,p^n_y)$ where we have introduced $\boldsymbol{\Phi^{(n)}}=\kappa^{(n+1)}\Phi^{t_0}(\kappa^{(n)})^{-1}$.
			\medbreak
			By Taylor expansion (coming back to scalar-valued functions), for all $i\in \llbracket 1, 2d \rrbracket$, there exists a $\theta_i\in ]0,1[$ such that 
			\begin{align*}
				\boldsymbol{\Phi^{(n)}}(q^n_x,q^n_y,p^n_x,p^n_y)_i&=\boldsymbol{\Phi^{(n)}}(q^n_x,0,p^n_x,0)_i+d_{\text{hyp}\to\text{i}} \boldsymbol{\Phi^{(n)}}(q^n_x,0,p^n_x,0).\Bigl[q^n_y,p^n_y\Bigr]
				\\&+\frac{1}{2}d^2_{\text{hyp,hyp}\to\text{i}} \boldsymbol{\Phi^{(n)}}(q^n_x,\theta_i q^n_y, p^n_x, \theta_i p^n_y). \Bigl[q^n_y,p^n_y\Bigr]^{\otimes 2}.        
			\end{align*}
			Here, the subscript $i$ means the $i$-th coordinate, $d_{\text{hyp}\to\text{i}} \boldsymbol{\Phi^{(n)}}(q^n_x,0,p^n_x, 0)$ is the part of a row of the Jacobian containing the derivatives of $i$-th components of the flow with respect to $y$ and $\eta$. Similarly,
			$d^2_{\text{hyp,hyp}\to\text{i}}\boldsymbol{\Phi^{(n)}}$ can be seen as a quadratic form in the transverse coordinates, see an example below in (\ref{d2nd_stab}).
			\medbreak
			\textbf{Controlling the quadratic part}
			\newline
			So far, we don't know precisely the form of the matrix $d^2_{\text{hyp,hyp}\to\text{i}} \boldsymbol{\Phi^{(n)}}(q^n_x,\theta_i q^n_y, p^n_x, \theta_i p^n_y)$. However, we have some information for the situation when $q^n_x=p^n_x=0$ and we are on the stable/unstable manifold.
			\newline
			Indeed, as the dynamics preserves the stable distributions, we know that, at these points, the matrix $d^2_{\text{hyp,hyp}\to\text{i}}\boldsymbol{\Phi^{(n)}}$ must have a special form: every time we derivate with respect to $\eta$, non zero coefficients can only appear for $\eta$ coordinates, which gives this form:
			\begin{multline}\label{d2nd_stab}
				d^2_{\text{hyp,hyp}\to i}\boldsymbol{\Phi^{(n)}}(0,0,0,\theta_i p^n_y).\Bigl[q^n_y,p^n_y\Bigr]^{\otimes 2}=
				\\\left\{
				\begin{aligned}
					&a_{1,i}(\theta_i).\Bigl(q^n_y,q^n_y\Bigr) +   a_{2,i}(\theta_i).\Bigl(q^n_y,p^n_y\Bigr) 
					\\&\quad \text{ if $i$ is a } x,y \text{ or } \xi \text{ coordinate }
					\\&a_{1,i}(\theta_i).\Bigl(q^n_y,q^n_y\Bigr) +   a_{2,i}(\theta_i ).\Bigl(q^n_y,p^n_y\Bigr)+ a_{3,i}(\theta_i ).\Bigl(p^n_y,p^n_y\Bigr)
					\\&\quad \text{ if $i$ is a } \eta \text{ coordinate }
				\end{aligned}\right.
			\end{multline}
			with $a_{k,i}(\theta_i)$ bilinear forms.
			\medbreak
			On the other hand, looking at the unstable manifold gives a slightly different form:
			
			\begin{multline}\label{d2nd_instab}
				d^2_{\text{hyp,hyp}\to i}\boldsymbol{\Phi^{(n)}}(0,\theta q^n_y,0,0).\Bigl[q^n_y,p^n_y\Bigr]^{\otimes 2}=
				\\\left\{
				\begin{aligned}
					&    a_{2,i}(\theta_i). \Bigl(q^n_y,p^n_y\Bigr) + a_{3,i}(\theta_i).\Bigl(p^n_y,p^n_y\Bigr)
					\\& \quad\text{ if $i$ is a } x,\xi \text{ or } \eta \text{ coordinate }
					\\&a_{1,i}(\theta_i) .\Bigl(q^n_y,q^n_y\Bigr) +   a_{2,i}(\theta_i ).\Bigl(q^n_y,p^n_y\Bigr)+ a_{3,i}(\theta_i ).\Bigl(p^n_y,p^n_y\Bigr)
					\\&\quad \text{ if $i$ is a } y \text{ coordinate }
				\end{aligned}\right.
			\end{multline}
			Now to understand $d^2_{\text{hyp,hyp}\to i} \boldsymbol{\Phi^{(n)}}(q^n_x, \theta_i q^n_y,p^n_x, \theta_i p^n_y)$ we will simply write it as a quadratic form:
			\begin{align*}
				d^2_{\text{hyp,hyp}\to i} \boldsymbol{\Phi^{(n)}}(q^n_x,\theta_i q^n_y&,p^n_x, \theta_i p^n_y).\Bigl[q^n_y,p^n_y\Bigr]^{\otimes 2}=a_{1,i}(q^n_x,\theta_i q^n_y,p^n_x, \theta_i p^n_y) .\Bigl(q^n_y,q^n_y\Bigr) 
				\\&+ 2  a_{2,i}(q^n_x,\theta_i q^n_y,p^n_x, \theta_i p^n_y). \Bigl(q^n_y,p^n_y\Bigr)+ a_{3,i}(q^n_x,\theta_i q^n_y,p^n_x, \theta_i p^n_y).\Bigl(p^n_y,p^n_y\Bigr)
			\end{align*}
			and estimate $a_{1,i},a_{2,i},a_{3,i}$ by comparing to their values along the unstable/stable manifolds.
			\medbreak
			We have the following estimates (where we have used that iterates stay in $K^\delta(\epsilon_1/2)$) for the central direction (ie $i\in \llbracket 1,d_{\scalerel*{\parallel}{\perp}}\rrbracket$ or $\llbracket d+1,d+d_{\scalerel*{\parallel}{\perp}}\rrbracket$ ), using (\ref{d2nd_stab}) for $a_{1,i}$ and (\ref{d2nd_instab}) for $a_{3,i}$:
			\begin{align*}
				\|d^2_{\text{hyp,hyp}\to \text{i}}\boldsymbol{\Phi^{(n)}}(q^n_x,\theta_i q_y^n,p^n_x,\theta_i p_y^n).\Bigl[q^n_y,p^n_y\Bigr]^{\otimes 2}\|&\leq  \|d^3 \boldsymbol{\Phi^{(n)}}\|_{\infty,K^\delta(\epsilon_1/2)}  \|(q^n_x,0,p^n_x,p^n_y)\|\|q^n_y\|^2
				\\+2\|d^3 \boldsymbol{\Phi^{(n)}}\|_{\infty,K^\delta(\epsilon_1/2)}\|(q^n_x,0,p^n_x,p^n_y)\|\|q^n_y\|\|p^n_y\|&+\|d^3 \boldsymbol{\Phi^{(n)}}\|_{\infty,K^\delta(\epsilon_1/2)}  \|(q^n_x,q^n_y,0,p^n_x)\|\|p^n_y\|^2 
			\end{align*} 
			for the unstable and stable ones, we get similarly (with $i\in \llbracket d_{\scalerel*{\parallel}{\perp}}+1, d\rrbracket$):
			\begin{multline*}
				\|d^2_{\text{hyp,hyp}\to i}\boldsymbol{\Phi^{(n)}}(q^n_x,\theta_i q_y^n,p^n_x,\theta_i p_y^n).\Bigl[q^n_y,p^n_y\Bigr]^{\otimes 2}\|\leq \|d^3 \boldsymbol{\Phi^{(n)}}\|_{\infty,K^\delta(\epsilon_1/2)} \| \|(q^n_x,q^n_y,p^n_x,0)\|\|p^n_y\|^2  
				\\+2\| d^2 \boldsymbol{\Phi^{(n)}}\|_{\infty,K^\delta(\epsilon_1/2)}\|q^n_y\|\|p^n_y\|+\|d^2 \boldsymbol{\Phi^{(n)}}\|_{\infty,K^\delta(\epsilon_1/2)} \|q^n_y\|^2
				\quad  
				\\\|d^2_{\text{hyp,hyp}\to i+d}\boldsymbol{\Phi^{(n)}}(q^n_x,\theta_{i+d} q_y^n,p^n_x,\theta_{i+d} p_y^n).\Bigl[q^n_y,p^n_y\Bigr]^{\otimes 2}\|\leq  \|d^3 \boldsymbol{\Phi^{(n)}}\|_{\infty,K^\delta(\epsilon_1/2)}  \|(q^n_x,0,p^n_x,p^n_y)\|\|q^n_y\|^2
				\\+2\| d^2 \boldsymbol{\Phi^{(n)}}\|_{\infty,K^\delta(\epsilon_1/2)}\|q^n_y\|\|p^n_y\|+\|d^2 \boldsymbol{\Phi^{(n)}}\|_{\infty,K^\delta(\epsilon_1/2)} \|p^n_y\|^2  .
			\end{multline*}
			Consequently, we obtain that:
			\begin{multline}\label{center}(q^{n+1}_x,p^{n+1}_x)=\boldsymbol{\Phi^{(n)}}(q^n_x,0,p^n_x,0)_c +d_{\text{hyp}\to\text{c}}\boldsymbol{\Phi^{(n)}}(q^n_x,0,p^n_x,0).\Bigl[q^n_y,p^n_y\Bigr]
					\\\quad+O(\|(q^n_x,p^n_x,0,p^n_y)\|\|q^n_y\|^2+\|q^n_y\|\|p^n_y\|+\|(q^n_x,p^n_x,q^n_y,0)\|\|p^n_y\|^2)
			\end{multline}
			\begin{equation}\label{unstab} q^{n+1}_y=d_{\text{hyp}\to y} \boldsymbol{\Phi^{(n)}}(q^n_x,0,p^n_x,0).\Bigl[q^n_y,p^n_y\Bigr]+O(\|q^n_y\|^2+\|q^n_y\|\|p^n_y\|+\|(q^n_x,q^n_y,p^n_x,0)\|\|p^n_y\|^2) 
			\end{equation}
			\begin{equation}\label{stab}
				 p^{n+1}_y =d_{\text{hyp}\to \eta} \boldsymbol{\Phi^{(n)}}(q^n_x,0,p^n_x,0).\Bigl[q^n_y,p^n_y\Bigr]+O(\|(q^n_x,0,p^n_x,p^n_y)\|\|q^n_y\|^2+\|q^n_y\|\|p^n_y\|+\|p^n_y\|^2).
			\end{equation}
			Here, it's important to notice that the constants linked to the $O$ can be taken independent of $n$ (hence of $h$). As mentionned above, we will only use the estimates in the center (\ref{center}) and the stable direction (\ref{stab}) as we will estimate the unstable one in a different way.
			\medbreak
			\textbf{Comparing with values at the base point of the chart}
			\newline
			Let us control the two first terms of the right hand side of (\ref{center}): with a Taylor expansion between $(q^n_x,0,p^n_x,0)$ and the origin, we obtain for the central component (ie $i\in \llbracket 1,d_{\scalerel*{\parallel}{\perp}}\rrbracket$ or $\llbracket d+1,d+d_{\scalerel*{\parallel}{\perp}}\rrbracket$ ):
			\begin{equation*} \boldsymbol{\Phi^{(n)}}(q^n_x,0,p^n_x,0)_i=\underbrace{\boldsymbol{\Phi^{(n)}}(0,0,0,0)_i}_{=0}+d_{\text{c}\to\text{c}}\boldsymbol{\Phi^{(n)}} (0,0,0,0).\Bigl[q^n_x,0,p^n_x,0\Bigr]+O(\|q^n_x,p^n_x\|^2).
			\end{equation*} 
			\medbreak
			For the linear parts $d_{\text{hyp}\to\text{c}}\boldsymbol{\Phi^{(n)}}(q^n_x,0,p^n_x,0).\Bigl[q^n_y,p^n_y\Bigr]$ that appears in (\ref{center}), we get:
			\begin{align*} 
				d_{\text{hyp}\to\text{c}}\boldsymbol{\Phi^{(n)}}(q^n_x,0,p^n_x,0).\Bigl[q^n_y,p^n_y\Bigr]&=d_{\text{hyp}\to\text{c}}\boldsymbol{\Phi^{(n)}}(0,0,0,0).\Bigl[q^n_y,p^n_y\Bigr]+O(\|q^n_x,p^n_x\|\|q^n_y,p^n_y\|)
				\\&=O(\|q^n_x,p^n_x\|\|q^n_y,p^n_y\|)
			\end{align*}
			where we have used the fact $d_{\text{hyp}\to\text{c}}\boldsymbol{\Phi^{(n)}}(0,0,0,0)=0$.
			\medbreak
			Combining those facts, we get 
			\begin{equation}\label{r^n_{center}}
				(q^{n+1}_x,p^{n+1}_x)=d_{\text{c}\to\text{c}}\boldsymbol{\Phi^{(n)}}(0,0,0,0).\Bigl[q^n_x,0,p^n_x,0\Bigr]+r^n_{\text{center}},
			\end{equation}
			with
			\begin{multline*} r^n_{\text{center}}=O(\|(q^n_x,p^n_x,0,p^n_y)\|\|q^n_y\|^2+\|q^n_y\|\|p^n_y\|+\|(q^n_x,p^n_x,q^n_y,0)\|\|p^n_y\|^2)
				\\+O(\|q^n_x,p^n_x\|\|q^n_y,p^n_y\|+\|q^n_x,p^n_x\|^2).
			\end{multline*} 
			\medbreak
			Similar computations in the stable direction gives:
			\begin{equation}\label{r^n_{stab}}
				p^{n+1}_y=d_{\eta\to\eta}\boldsymbol{\Phi^{(n)}}(0,0,0,0).\Bigl[0,0,0,p^n_y\Bigr]+r^n_{\text{stab}},
			\end{equation}
			with $r^n_{\text{stab}}=O(\|q^n_x,p^n_x\|\|q^n_y,p^n_y\|+\|(q^n_x,0,p^n_x,p^n_y)\|\|q^n_y\|^2+\|q^n_y\|\|p^n_y\|+\|p^n_y\|^2)$.
			\medbreak
			\textbf{Adding errors by induction}
			\newline
			Our goal is now to prove by induction that:
			\begin{align}
				\|(q^n_x,p^n_x)-d_{\text{c}\to \text{c}} [\boldsymbol{\Phi^{(n)}}\text{\dots}\boldsymbol{\Phi^{(0)}}](0,0,0,0).\Bigl[q^0_x,p^0_x\Bigr]\|&\leq \epsilon_1 e^{n(\lambda_c+2\epsilon_2/3)} h^{\tau}\label{induc_center}
				\\\|p^n_y-d_{\eta \to \eta} [\boldsymbol{\Phi^{(n)}}\text{\dots}\boldsymbol{\Phi^{(0)}}](0,0,0,0). \Bigl[p^0_y\Bigr]\|&\leq \epsilon_1 e^{n(\lambda_c+2\epsilon_2/3)} h^{\tau}\label{induc_stab}.
			\end{align}
			Assume that this holds for some $n\leq n_{\text{max}}$, then we know that (with rough estimates on the linear part, especially in the stable direction)
			\begin{equation*} 
				\|(q^n_x,p^n_x)\|\leq (1+\epsilon_1)e^{n(\lambda_c+2\epsilon_2/3)} h^{\tau},
			\end{equation*} 
			\begin{equation*} 
				\|p^n_y\|\leq (1+\epsilon_1)  e^{n(\lambda_c+2\epsilon_2/3)} h^{\tau},
			\end{equation*} 
			then (\ref{r^n_{center}}) gives 
			\begin{align*}
				\|(q^{n+1}_x,p^{n+1}_x)-d_{\text{c}\to \text{c}} [\boldsymbol{\Phi^{(n)}}\dots&\boldsymbol{\Phi^{(0)}}](0,0,0,0).\Bigl[q^0_x,p^0_x\Bigr]\|\leq \epsilon_1 e^{-\epsilon_2/3} e^{(n+1)(\lambda_c+2\epsilon_2/3)} h^{\tau}
				\\&+\tilde{C}(1+\epsilon_1)(3\epsilon_1+2\epsilon_1^2) e^{n(\lambda_c+2\epsilon_2/3)} h^{\tau}, 
			\end{align*}
			Where $\tilde{C}$ is the implicit constant of the $O$ controlling $r^n_{\text{center}}$, and we have used the a priori estimate $\|q^n_x,q^n_y,p^n_x,p^n_y\|\leq \epsilon_1$. Then, by choosing $t_0$ larger if necessary so that $e^{-\epsilon_2t_0/3}+3\tilde{C}e^{-(\lambda_c+2\epsilon_2/3)t_0}<1$, we can take $\epsilon_1>0$ small enough so that $e^{-\epsilon_2t_0/3}\epsilon_1+\tilde{C}e^{-(\lambda_c+2\epsilon_2/3)t_0}(1+\epsilon_1)(3\epsilon_1+2\epsilon_1^2)\leq \epsilon_1$ which gives (\ref{induc_center}) for $n+1$.
			\medbreak
			For (\ref{induc_stab}), we use (\ref{r^n_{stab}}) to obtain 
			\begin{align*}
				\|p^{n+1}_y-d_{\eta \to \eta} [\boldsymbol{\Phi^{(n)}}\text{\dots}\boldsymbol{\Phi^{(0)}}](0,0,0,0) \Bigl[p^0_y\Bigr]\|&\leq e^{-\epsilon_2t_0/3} \epsilon_1 e^{(n+1)(\lambda_c+2\epsilon_2/3)} h^{\tau}
				\\&+\tilde{C}(1+\epsilon_1)(3\epsilon_1+\epsilon_1^2)e^{n(\lambda_c+2\epsilon_2/3)} h^{\tau},
			\end{align*}
			which gives also the result for $\epsilon_1$ small enough. 
			\item Finally, to obtain the estimate on $\|q^n_y\|$, we have to do the same argument for $\Phi^{-t_0}$ (i.e starting from $n$ and going to $0$) and compare similarly the $y$ component to its linearized model. 
		\end{itemize} 
	\end{proof}
	
	\begin{Remark}
		Notice that in the stable direction, we cannot see the effect of the contraction due to the choice of coordinates: the error term coming from \ref{induc_stab} is the leading order term in the stable direction.
	\end{Remark}
	\begin{Remark}\label{dyn class pas rho}
		As we have seen in the introduction, our assumption that $\Phi^{nt_0}(\rho)\in K^\delta(\epsilon_1)$ is not necessarily verified, see figure \ref{dist It rhot} and the associated remark. However, in this same remark, we have explained that the points that should be relevant from a classical point of view are the points of $\mathcal{I}^{(t)}$, and all of these should stay in a $\epsilon_1$ vicinity of $K^\delta$. Hence we will need to apply lemma \ref{proof_dyn_class} to points $\underline{\rho} \in \mathcal{I}^{(t)}$ instead of $\rho$, but they verify the  assumptions.
	\end{Remark}
	Finally, the estimates of lemma \ref{proof_dyn_class} let us compare the linearized dynamics:
	\begin{Lemma}\label{compare_linear}
		For all $n\in \llbracket 0,n_\text{max}\rrbracket,$
		\begin{equation*} 
			\left\|d\boldsymbol{\Phi^{n}}(\boldsymbol{\tilde{\rho}})-d\boldsymbol{\Phi^{n}}(\boldsymbol{\rho})\right\|\leq C e^{nt_0(\lambda_{\text{max}}+2\epsilon_2/3) } e^{nt_0 (\lambda^{\text{c}}+\epsilon_2/3)} h^{\tau},
		\end{equation*} 
		where we have set $\boldsymbol{\Phi^{n}}=\boldsymbol{\Phi^{(n)}}\text{\dots}\boldsymbol{\Phi^{(0)}}$.
	\end{Lemma}
	\begin{proof}
		Using the chain rule, we check that
		\begin{equation*} 
			d_{\boldsymbol{\tilde{\rho}}}\boldsymbol{\Phi^{n}}-d_{\boldsymbol{\rho}}\boldsymbol{\Phi^{n}}=\sum_{k=0}^{n-1} d_{\boldsymbol{\Phi^{k+1}}(\tilde{\rho})} [\boldsymbol{\Phi^{(n)}\text{\dots}\Phi^{(k+1)}}]\circ [d_{\boldsymbol{\Phi^{k}}(\boldsymbol{\rho})}\boldsymbol{\Phi^{(k)}}-d_{\boldsymbol{\Phi^{k}}(\boldsymbol{\tilde{\rho}})}\boldsymbol{\Phi^{(k)}}]\circ d_{\boldsymbol{\rho}}[\boldsymbol{\Phi^{(k-1)}\text{\dots}\Phi^{(0)}}].
		\end{equation*} 
		By Taylor expansion, we obtain the following
		\begin{align*}
			\|d_{\boldsymbol{\Phi^{(k)}}(\boldsymbol{\rho})}\boldsymbol{\Phi^{k}}-d_{\boldsymbol{\Phi^{k}}(\boldsymbol{\tilde{\rho}})}\boldsymbol{\Phi^{(k)}}\|&\leq C \|\boldsymbol{\Phi^{k}}(\boldsymbol{\rho})-\boldsymbol{\Phi^{k}}(\boldsymbol{\tilde{\rho}})\|
			\\&\leq C (1+\epsilon_1)\max( e^{-(n_{\text{max}}-k)t_0(\nu_{\text{min}}^\perp-2\epsilon_2/3)},e^{(\lambda^{\text{c}}+2\epsilon_2/3)kt_0} h^{\tau}),
		\end{align*}
		with the estimate of lemma \ref{proof_dyn_class}.
		\newline
		Similarly, we obtain
		\begin{align*} 
			\|d_{\boldsymbol{\Phi^{k+1}}(\tilde{\rho})} [\boldsymbol{\Phi^{(n)}\text{\dots}\Phi^{(k+1)}}]\|&\leq C \|d_{\boldsymbol{\Phi^{k+1}}(\tilde{\rho})} [\boldsymbol{\Phi^{(n)}\text{\dots}\Phi^{(k+1)}}]\| \sum_{j=k}^n d(\boldsymbol{\Phi^{j+1}}(\tilde{\rho}),\boldsymbol{\Phi^{j+1}}(\rho))
			\\&\leq C\|d_{\boldsymbol{\Phi^{k+1}}(\tilde{\rho})} [\boldsymbol{\Phi^{(n)}\text{\dots}\Phi^{(k+1)}}]\|
			\\&\leq C e^{(n-k)t_0(\lambda_{\text{max}}+2\epsilon_2/3) }.
		\end{align*}
		Finally, coming back to the chain rule, we obtain:
		\begin{align*}
			\|d_{\boldsymbol{\tilde{\rho}}}\boldsymbol{\Phi^{n}}-d_{\boldsymbol{\rho}}\boldsymbol{\Phi^{n}}\|&\leq Ce^{(n)t_0(\lambda_{\text{max}}+2\epsilon_2/3) }  \sum_{j=0}^n d(\boldsymbol{\Phi^{j+1}}(\tilde{\rho}),\boldsymbol{\Phi^{j+1}}(\rho))
			\\&\leq Ce^{(n)t_0(\lambda_{\text{max}}+2\epsilon_2/3) }\max( e^{-(n_{\text{max}}-n)t_0(\nu_{\text{min}}^\perp-2\epsilon_2/3)},e^{(\lambda^{\text{c}}+2\epsilon_2/3)nt_0} h^{\tau}).
		\end{align*}
	\end{proof}
	\begin{Remark}
		This corollary is only pertinent for small values of $n$ (i.e. below Ehrenfest time) as the right hand side might be a negative power of $h$.
		\medbreak
		Unfortunately, it is not easy to obtain an estimate when restricting the differential to the tangent of $K^\delta$ that only involves the central Lyapunov exponent; the fact that we are not exactly on $K^\delta$ prevents us from ruling out contributions of transverse directions.
	\end{Remark}

	\subsection{Semiclassical analysis}
	For some reminders on pseudo-differential operators, symbol classes and microlocalization, one can see \cite{Zwbook}, \cite{DimSjo} for instance.
	\subsubsection{Integral representation formula of Fourier Integral Operators}
	Let us recall some facts about the local theory of Fourier Integral Operators. Consider $\Omega_0$ and $\Omega_1$ two open sets of $T^*\R^d=\R^{2d}$ small enough, a symplectomorphism $F \colon \Omega_0 \to \Omega_1$ and two points $\rho_0\in \Omega_0$, $\rho_1=F(\rho_0)$. We define the twisted graph of $F$
	\begin{equation*} 
		\text{Gr}'(F)\coloneq\Bigl\{(X^1,\Xi^1;X^0,-\Xi^0),\ (X^1,\Xi^1)=F(X^0,\Xi^0),\ (X^0,\Xi^0)\in \Omega_0\Bigr\}\subset \Omega_1\times \Omega_0.
	\end{equation*} 
	We want to represent this graph in $(X^1,\Xi^0)$ coordinates. However, the projection 
	\begin{equation*} 
		(X^1,\Xi^1,X^0,\Xi^0) \in \text{Gr}'(F) \mapsto (X^1,\Xi^0)
	\end{equation*} 
	might not be of full rank $2d$ near $(\rho_1,\rho_0)$ but of rank $2d-\mathpzc{k}$ (which can be assumed constant provided that $\Omega_0$, $\Omega_1$ are small enough), with $0\leq \mathpzc{k} \leq d$. Nonetheless there always exists a subset $J$ of $\llbracket 1,d\rrbracket$ of cardinal $\mathpzc{k}$ such that $\text{Gr}'(F)$ can be represented in the $(\widetilde{X^1},\overline{\Xi^1},\Xi^0)$ coordinates. Without loss of generality, let us assume that $J=\llbracket 1,\mathpzc{k}\rrbracket$ and denote the associated coordinates by 
	\begin{equation*}
		\overline{X}=(X_j)_{j\in J},\ \overline{\Xi}=(\Xi_j)_{j\in J},\ \widetilde{X}=(X_j)_{j\notin J}, \ \widetilde{\Xi}=(\Xi_j)_{j\notin J}.
	\end{equation*}
	The fact that $\text{Gr}'(F)$ can be represented in the $(\widetilde{X^1},\overline{\Xi^1},\Xi^0)$ coordinates means that there exists $\phi\in C^\infty(\R^{d-\mathpzc{k}}\times\R^{\mathpzc{k}}\times \R^{d},\R)$ such that 
	\begin{multline*}
		\text{Gr}'(F)=
		\\\Bigl\{\left(-d_{\overline{\Xi^1}} \phi\left(\widetilde{X^1},\overline{\Xi^1},\Xi^0\right), \widetilde{X^1},\overline{\Xi^1},d_{\widetilde{X^1}} \phi\left(\widetilde{X^1},\overline{\Xi^1},\Xi^0\right);d_{\overline{\Xi^0}}\phi\left(\widetilde{X^1},\overline{\Xi^1},-\Xi^0\right), \overline{\Xi^0}\right)
		,\ (\widetilde{X^1},\overline{\Xi^1},\Xi^0) \Bigr\}.
	\end{multline*}
	and 
	\begin{equation*}
		\text{ the matrix } \begin{pmatrix} \partial^2_{\Xi^0,\widetilde{X^1}} \psi , \partial^2_{\Xi^0,\overline{\Xi^1}} \psi \end{pmatrix} \text{ is invertible}.
	\end{equation*}
	Then, by introducing auxiliary variables $\Theta=(\theta_1,\dots, \theta_k)$ (which will play the role of $\overline{\Xi^1}$), we can define a generating function of the transformation $F$ in the $(X^1,\Xi^0)$ coordinates $\psi\in C^\infty(\R^d\times \R^d \times \R^\mathpzc{k})$  as follows:
	\begin{equation*}
		\psi(X^1,\Xi^0,\Theta)=\Theta\cdot \overline{X^1}+\phi\left(\widetilde{X^1},\Theta,\Xi^0\right),
	\end{equation*}
	which verifies:
	\begin{align}\label{hyp psi}
		\begin{split}
			\text{Gr}'(F)&=\Bigl\{ (X^1,d_{X^1}\psi(X^1,\Xi^0,\Theta); d_{\Xi^0} \psi(X^1,\Xi^0;\Theta), -\Xi^0),\ (X^1,\Xi^0,\Theta)\in C_\psi \Bigr\},
			\\C_\psi&=\Bigl\{(X^1,\Xi^0,\Theta),\ d_\Theta \psi(X^1,\Xi^0,\Theta)=0\Bigr\},
			\\ d^2_{\overline{X^1},\Theta} \psi&=I_{\mathpzc{k}} \text{ and } \begin{pmatrix}  \partial^2_{\Xi^0,\widetilde{X^1}} \psi, \partial^2_{\Xi^0,\Theta} \psi , \end{pmatrix} \text{ is invertible}.
		\end{split}
	\end{align}
	In the case $\mathpzc{k}=0$, we find the usual condition $\partial^2_{\Xi^0,X^1}\psi$ invertible.
	\medbreak
	Such a generating function allows for a local representation of compactly microlocalized Fourier Integral operator with an oscillating integral (see, for instance, \cite[Theorem 10.4]{Zwbook}): if $T$ is a Fourier Integral Operator quantizing $F$ microlocalized in a sufficiently small neighborhood $U_1\times U_0\subset T^*\R^d\times T^* \R^d$ then
	\begin{multline*} 
		Tu(X^1)=
		\\\frac{1}{(2\pi h)^{(\mathpzc{k}+2d)/2}} \int_{\R^\mathpzc{k}} \int_{\R^{2d}} e^{\frac{i}{h} \bigl(\psi(X^1,\Xi^0,\Theta)-X^0.\Xi^0\bigr)} a(X^1,\Xi^0,\Theta) u(X^0) \d X^0 \d \Xi^0 \d \Theta+O(h^\infty)_S \|u\|_{H^{-M}},
	\end{multline*} 
	for any $M$, where the $O$ is with respect to any Schwartz semi-norm. 
	\newline Here, $a\in C^\infty_c(\R^{2d+\mathpzc{k}})$ admits an expansion in powers of $h$. If we also assume that $T$ is microlocally unitary on some compact region $A$ of the phase space, then the leading term $a_0$ verifies  
	\begin{equation}\label{a0}
		\bigl|a_0(X^1,\Xi^0,\Theta)\bigr|^2=\left|\det \begin{pmatrix}  \partial^2_{\Xi^0,\widetilde{X^1}} \psi, \partial^2_{\Xi^0,\Theta} \psi , \end{pmatrix}\right|\coloneq D_\psi(X^1,\Xi^0,\Theta) \text{ on }A.
	\end{equation}
	\subsubsection{Metaplectic operators}
	Metaplectic operators are a special type of Fourier Integral Operators whose phases $\psi$ are quadratic forms. It can be shown that quantizations of linear symplectomorphisms $\kappa$ are an example of such operators. More precisely, the metaplectic group is a double-covering of the symplectic group $\text{Sp}_{2d}(\R)$. As a consequence, a metapletic operator associated with a linear symplectomorphism is defined up to a phase. We will denote such a metaplectic operator by $\mathcal{M}_h(F)$ (or $\mathcal{M}(F)$) and perform computations up to a global phase when it is involved.
	\medbreak
	They are simple Fourier Integral Operators encoding a non-trivial dynamic. 
	The best way to understand them is by looking at their actions on coherent states, see section \ref{elementary} or proposition \ref{Metap_phi_0} for more details.
	\begin{Example}
		\begin{itemize}
			\item The $h-$Fourier transform is an example of metaplectic operator, it quantizes the linear symplectic map \begin{equation*} J=\begin{pmatrix}
					0 & I_d \\ -I_d & 0
				\end{pmatrix}.
			\end{equation*} 
			Partial Fourier transform are another example of such operators
			\item Let $A$ be a linear invertible transformation of $\R^{d}$, then $u\mapsto |\det A^{-1}|^{1/2}u(A^{-1}x)$ is a metaplectic operator associated with the linear symplectomorphism 
			\begin{equation*} 
				\kappa=\begin{pmatrix}
					A & 0
					\\ 0 & (A^T)^{-1}
				\end{pmatrix}.
			\end{equation*} 
			\item Let $A$ be a real symmetric matrix, then $u\mapsto e^{i Ax\cdot x/2} u$ is a metaplectic operator associated with the linear symplectomorphism
			\begin{equation*} \kappa=\begin{pmatrix}
					I_d & 0
					\\ A & I_d
				\end{pmatrix}.
			\end{equation*} 
		\end{itemize}
	\end{Example}
	\begin{Remark}
		In fact, these transformations generate the metaplectic group, see \cite[Section 7]{De_Gosson}.
	\end{Remark}
	\medbreak
	An explicit formula can be used to link the phase $\psi$ and $\kappa$ assuming the invertibility of the coefficient $\partial^2_{X,\Xi}\psi$ (i.e. when no auxiliary variable $\Theta$ is needed):
	\begin{equation}\label{kappa psi}
		\kappa=\begin{pmatrix}
			(\partial^2_{X,\Xi}\psi)^{-1} & -(\partial^2_{X,\Xi} \psi)^{-1} \partial^2_{\Xi,\Xi}\psi
			\\ \partial^2_{X,X}\psi(\partial^2_{X,\Xi}\psi)^{-1} & (\partial^2_{X,\Xi} \psi)^T-\partial^2_{X,X} \psi (\partial^2_{X,\Xi} \psi)^{-1} \partial^2_{\Xi,\Xi} \psi
		\end{pmatrix}.
	\end{equation}
	In the general case with auxiliary variable $\Theta$, the formula is more involved.
	\medbreak
	Another useful fact is the intertwining with Weyl-Heisenberg operators $\hat{T}(\rho)$ defined in (\ref{T rho}):
	\begin{Proposition}\label{intertwining}
		Let $\rho\in \R^{2d}$ and $\kappa$ be a linear symplectomorphism then,
		\begin{equation*} 
			\mathcal{M}(\kappa)\hat{T}(\rho)=\hat{T}(\kappa(\rho))\mathcal{M}(\kappa).
		\end{equation*} 
	\end{Proposition}

	\subsection{Coherent states}
	\subsubsection{Usual coherent states}\label{coherent_state}
	We start by defining the semiclassical coherent state centered at $0$ by:
	\begin{equation*} 
		\varphi_0(x)=\frac{1}{(\pi h)^{d/4}} e^{-\|x\|^2/(2h)}.
	\end{equation*} 
	It is often used to describe a particle located at position $0$ and momentum $0$ as it verifies $\text{WF}_h(\varphi_0)=\{(0,0)\}$.
	\medbreak
	Then we can define the coherent state centered at $\rho=(q,p)\in \R^{2d}$ thanks to Weyl-Heisenberg operators (\ref{T rho}):
	\begin{equation*} 
		\varphi_\rho(x)\coloneq\hat{T}(\rho)\varphi_0(x)=\frac{e^{-ip\cdot q/(2h)}}{(\pi h)^{d/4}}\exp\left(i\frac{p\cdot x}{h}\right) \exp\left(-\frac{\|x-q\|^2}{2h})\right) .
	\end{equation*} 
	It will also be useful to zoom to the microscopic scale for proofs. To do so, we introduce the unitary scaling operator $\Lambda_h$ defined by 
	\begin{equation*} 
		\Lambda_hu(x)=h^{-1/4} u\left(\frac{x}{\sqrt{h}}\right), \text{ for } u\in L^2(\R^d),
	\end{equation*} 
	and we notice that
	\begin{equation*} 
		\varphi_0=\Lambda_h \Psi_0,
	\end{equation*} 
	where $\Psi_0(x)=\frac{1}{\pi^{1/4}} e^{-x^2/2}$ is the usual Gaussian.
	Similarly, we can also change the scaling of metaplectic operators using $\Lambda_h$:
	\begin{equation*} 
		\mathcal{M}_h(\kappa)\Lambda_h=\Lambda_h \mathcal{M}_1(\kappa).
	\end{equation*} 
	\medbreak
	We end this part by explaining how we can turn the properties of operators acting on coherent states into properties of operators acting on general $L^2\left(\R^d\right)$ states.
	\begin{definition}
		Let us define the Fourier-Bargmann transform. For $u\in L^2\left(\R^d\right)$, we set:
		\begin{equation*} 
			u^\#(\rho)=(2\pi h)^{-d/2} \langle u, \varphi_\rho\rangle, \quad \rho\in \R^{2d}.
		\end{equation*} 
	\end{definition}
	
	\begin{Proposition}\cite[Proposition 4]{Comb_Rob}
		The Fourier-Bargmann transform $u\mapsto u^\#$ is an isometry between $L^2\left(\R^d\right)$ and $L^2(\R^{2d})$.
	\end{Proposition}
	
	\begin{Proposition}\cite[Section 1.2.3]{Comb_Rob}
		$(\varphi_{\rho})_{\rho\in \R^{2d}}$ is an over-complete system of $L^2\left(\R^d\right)$ which means that for all $u\in L^2\left(\R^d\right),$
		\begin{equation*}  
			u=(2\pi h)^{-d} \int_{\R^{2d}} \langle u, \varphi_\rho\rangle \varphi_\rho \d \rho.
		\end{equation*} 
	\end{Proposition}
	\begin{Remark}
		This tells us that one can obtain $u$ from knowledge of $u^\#$. The process is quite similar to Fourier's transform, but represents the function with $2d$ variables (i.e. in phase space) rather than $d$ variables (i.e. in frequency domain).
	\end{Remark}
	
	Recall that the function $\varphi_0$ introduced above corresponds to the ground state of the harmonic oscillator $-h^2 \Delta+x^2$. The other eigenfunctions of this operator will also have an important role in our computations, they are called excited states and can be obtained by applying creation operators $a_j^\dagger=\frac{1}{\sqrt{2h}}(-h\partial_{x_j}+x_j)$ a certain number of times:
	\begin{equation*} 
		\varphi_{0}^{\gamma}\coloneq (\textbf{a}^\dagger)^\gamma\varphi_0 = (a_1^{\dagger})^{\gamma_1}\dots (a_d^\dagger)^{\gamma_d} \varphi_0=H_\gamma\left(x/\sqrt{h}\right)\varphi_0, \quad \gamma\in\N^d,
	\end{equation*} 
	where $H_\gamma$ is a product of Hermite polynomials of total degree $|\gamma|$.
	\medbreak
	Considering linear combination of these states, suggest introducing, for $P\in \C[X]$,
	\begin{equation*} 
		\varphi_0^{(P)}(x)\coloneq P\left(x/\sqrt{h}\right)\varphi_0,
	\end{equation*} 
	that we will still refer to as an excited state.
	\newline
	Similarly, we define the excited state centered at point $\rho$ as $\varphi_{\rho}^{(P)}=\hat{T}(\rho)\varphi_{0}^{(P)}$.
	\medbreak
	These excited states will let us express the action of a pseudodifferential operator on a coherent state as a semiclassical expansion in excited states centered at the same point:
	
	\begin{Proposition}\cite[Lemma 14]{Comb_Rob}\label{cas_iso}
		Assume that $a\in S(m)$, then for every $N\geq 1$, we have 
		\begin{equation*} 
			\Op (a) \varphi_{\rho}= \sum_{|\gamma|\leq N} h^{|\gamma|/2} c_{\gamma} \varphi_{\rho}^{(X^\gamma)}+O_{L^2}\left(h^{(N+1)/2}\right),
		\end{equation*} 
		with $c_{\gamma}$ involving derivatives of $a$ at point $\rho$ and the estimate on the remainder being uniform in $L^2\left(\R^d\right)$ norm for $\rho$ in a compact of $\R^{2d}$.
	\end{Proposition}
	\subsubsection{Action of metaplectic operators on coherent states}
	Let us start by propagating our states when the underlying dynamics is linear: compute the action of metaplectic operators. It will turn out that, in this setting, we can compute the result exactly.
	\medbreak
	Using proposition \ref{intertwining} on the intertwining between Weyl-Heisenberg and metaplectic operators, we can reduce our study to coherent states centered at $0$.
	\medbreak
	\paragraph{Propagation of the ground state}
	\begin{Proposition}\cite[Chapter 3]{Comb_Rob}\label{Metap_phi_0}
		Let $\kappa=\begin{pmatrix}
			A & B \\ C & D
		\end{pmatrix}\in Sp(n)$ be a symplectic linear map and $\mathcal{M}_h(\kappa)$ be a metaplectic operator associated with $\kappa$ and define $M\coloneq A+iB, N\coloneq C+iD$.
		\newline 
		Then, $\det M\neq 0$ and 
		\begin{equation*} 
			\mathcal{M}_h(\kappa)\varphi_0(x)=(\pi h)^{-n/4} \left|\det (\Im(\Gamma_\kappa))\right|^{1/4} e^{i \frac{x\cdot\Gamma_\kappa x}{2h}}
		\end{equation*} 
		where $\Gamma_\kappa=NM^{-1}=(C+iD)(A+iB)^{-1}$. Moreover $\Gamma_\kappa$ is a complex symmetric matrix verifying
		\begin{equation*} 
			\Im(\Gamma_\kappa)=M^{-T} \overline{M}\phantom{}^{-1}
		\end{equation*} 
		which is positive definite. We say that $\Gamma_\kappa$ is a matrix in Siegel's space and write $\Gamma_\kappa \in S\mathbb{H}_d$.
	\end{Proposition}
	\begin{Example} 
		If $\kappa$ is diagonal, $\kappa=\begin{pmatrix}
			\lambda I_d &0 \\ 0 & \lambda^{-1}I_d
		\end{pmatrix}$ then 
		\begin{equation*} 
			\mathcal{M}_h(\kappa)\varphi_0(x)=\frac{\lambda^{-1/2}}{(\pi h)^{n/4}} e^{-\lambda^{-2}\|x\|^2/(2h)},
		\end{equation*} 
		which means that $\Gamma_\kappa\in iS_d(\R)$. 
	\end{Example}

	\paragraph{Propagation of excited states}
	We can in fact go even further and compute the propagation of excited states $\varphi_{0}^{(P)}$ rather than only the ground state $\varphi_0$.
	\medbreak
	To do so, we follow \cite{Hag4},\cite{HagBonus} 
	
	\begin{Proposition}\label{propag excited}
		Let $\kappa=\begin{pmatrix}
			A & B \\ C & D
		\end{pmatrix}\in Sp(2d)$ be a symplectic linear map and $\mathcal{M}_h(\kappa)$ be a metaplectic operator associated with $\kappa$. Then,
		\begin{equation*} 
			\mathcal{M}_h(\kappa)\varphi_{0}^{(P)}(x)=(\pi h)^{-n/4} \left|\det \Im(\Gamma_\kappa)\right|^{1/4} Q(P)\left(h^{-1/2}\Im (\Gamma_\kappa)^{1/2} x\right) e^{i\frac{x.\Gamma_\kappa x}{2h}},
		\end{equation*} 
		where $\Gamma_\kappa=(C+iD)(A+iB)^{-1}$ and $Q(P)$ polynomial of same degree as $P$, verifying $N_\infty\bigl(Q(P)\bigr)$ \newline$\leq C_{\deg(P)}N_\infty(P)$, with $N_\infty$ the sup norm of coefficients. 
	\end{Proposition}
	\paragraph{Norm estimation in $H_1\bigl(\jp{\rho}^K\bigr)$}
	Here we would like to introduce norms for squeezed states that can perceive the deformation induced by $\kappa$. The first remark is that the $L^2$ norm cannot be used to do this because our states are always normalized. A natural extension is to look at weighted Sobolev spaces $H_h(\jp{\rho}^K)$, defined as follows:
	\begin{equation*} 
		H_h\left(\jp{\rho}^K\right)\coloneq\Op\left(\jp{\rho}^{-N}\right)\left(L^2\bigl(\R^d\bigr)\right)\subset \mathcal{S}'\left(\R^d\right).
	\end{equation*} 
	For $N\in \N$, we have a characterization of the elements of this space as functions $u\in \mathcal{S}'\left(\R^d\right)$ such that 
	\begin{equation*} 
		x^{\alpha} (h\partial)^\beta u\in L^2\left(\R^d\right),\ \forall |\alpha|+|\beta| \leq N,
	\end{equation*} 
	and the following norm:
	\begin{equation}\label{norm eq}
		\|u\|^2_{H_h\left(\jp{\rho}^K\right)}\coloneq \sup_{|\alpha|+|\beta|\leq N} \left\|x^\alpha (h\partial)^\beta u\right\|^2_{L^2}.
	\end{equation}
	We recall a useful boundedness property of pseudodifferential operators with appropriate symbols on those spaces:
	\begin{Proposition}\cite[Proposition 2.3]{Lucas}
		Let $K\in \mathbb{N}$, there exist $M\in \N$ and $C>0$ such that for any $a\in S_\delta\left(\jp{\rho}^K\right), \Op(a) \colon H_{h}\left(\jp{\rho}^K\right)\to L^2\left(\R^d\right)$ is uniformly bounded and 
		\begin{equation*} 
			\|\Op(a)\|_{H_{h}(\jp{\rho}^K)\to L^2\left(\R^d\right)}\leq C \sup_{|\alpha|\leq M} h^{|\alpha|/2} \left\|\jp{\rho}^{-K} \partial^\alpha a\right\|_{\infty}
		\end{equation*} 
	\end{Proposition}
	We can estimate the norm of squeezed states in these spaces:
	\begin{Lemma}\cite[Lemma 2.3]{Lucas}
		There exists a family of universal constants $(C_{K,k})_{K\in \N,k\in\N^d}$ such that the following holds: let $K\in \N, k\in \N^d$ and $\kappa$ be a symplectic linear map. Then for all $0<h\leq1,$
		\begin{equation*} 
			\left\|\mathcal{M}_h(\kappa)(\varphi_0^{(P)})\right\|_{H_h(\jp{\rho}^K)}\leq C_{K,\text{deg }P} N_\infty(P) \sum_{l=0}^K h^{(l+|k|)/2} \|\kappa\|^l,
		\end{equation*} 
		where $N_\infty(P)$ is the sup norm of coefficients of $P$.
	\end{Lemma}

	Notice that as long as $\|\kappa\|\leq C h^{-1/2+\cdot},$ the right-hand term in the above lemma is dominated by the term $l=0$, i.e. the $L^2$ norm. Hence, for squeezed states such as in section \ref{CR method} (where $\|\kappa\|\leq C h^{-1/6+\cdot}$), these norms do not observe the impact of $\kappa$. 
	\medbreak
	In order to see its effect, one would have to zoom up to a microscopic scaling:
	\begin{Lemma}\cite[Corollary 2.1]{Lucas}\label{control_xy}
		There exists a family of constants $C_{K,k}$ such that for all $P\in \C[X]$, for all linear symplectic maps $\kappa$ and for all $K\in \N,$
		\begin{equation*} 
			\left\|\mathcal{M}_1(\kappa) (P\Psi_0)\right\|_{H_1(\jp{\rho}^K)}\leq C_{K,\text{deg} P} N_\infty(P) \|\kappa\|^K,
		\end{equation*} 
		where $N_\infty(P)$ is the sup norm of coefficients of $P$.
	\end{Lemma}
	\begin{proof}
		We give a different proof of this result than in \cite{Lucas}, based on properties of matrices $\Gamma$. We use the equivalent norm introduced in (\ref{norm eq}) to perform computations.
		\medbreak
		To compute some $\|x^\alpha \partial^\beta \mathcal{M}_1(\kappa) [P\Psi_0]\|_{L^2}$, we write 
		\begin{equation*} 
			\mathcal{M}_1(\kappa) [P\Psi_0]=\left|\det \Im (\Gamma)\right|^{-1/4} \tilde{P}\left( \Im(\Gamma)^{1/2}x\right) e^{ix \Gamma x/2},
		\end{equation*}  which gives, using the notations of proposition \ref{Metap_phi_0}
		\begin{align*}
			\|x^\alpha \partial^\beta \mathcal{M}_1(\kappa) [P&\Psi_0]\|_{L^2}^2\leq \sum_{\beta_1+\beta_2=\beta} \binom{\beta}{\beta_1} \int \frac{|x|^{2\alpha} \left|\Im \Gamma\right|^{\beta_1} \left|\Gamma x\right|^{2\beta_2}}{\sqrt{\left|\det \Im (\Gamma)\right|}} \tilde{P}^{(\beta_1)} \left(\Im(\Gamma)^{1/2}x\right)^2 e^{-x\Im(\Gamma)x} \d x 
			\\& \leq C\sup_{\beta_1+\beta_2=\beta} \left|\Im(\Gamma)^{-1}\right|^{\alpha}\left|\Im(\Gamma)\right|^{\beta_1} \left|\Gamma \Im(\Gamma)^{-1/2}\right|^{2\beta_2}\int |y|^{2(\alpha+\beta_2)} \tilde{P}^{(\beta_2)}(y)^2e^{-y^2}\d y
			\\&\leq C N_\infty(P)\sup_{\beta_1+\beta_2=\beta} |M|^{2\alpha}|M^{-1}|^{2\beta_1} |N|^{2\beta_2}
			\\&\leq C N_\infty(P)\|\kappa\|^{2|\alpha+\beta|}.
		\end{align*}
		Where we have used the fact that:
		\begin{align}\label{GammaImGamma}
			\begin{split}
				\left|\Gamma \Im(\Gamma)^{-1/2}\right|_2^2&=\text{Tr} \left(\Im(\Gamma)^{-1/2}\overline{\Gamma} \Gamma \Im(\Gamma)^{-1/2}\right)
				\\&=\text{Tr} \left(\overline{\Gamma}\Gamma \Im(\Gamma)^{-1}\right)
				\\& =\text{Tr} \left(\overline{M}^{-T}\overline{N}^T N M^{-1} M \overline{M}^{T}\right)
				\\&=\text{Tr} \left(\overline{N}^T N\right)
				\\&= |N|_2^2.
			\end{split}
		\end{align}
	\end{proof}

	\section{\texorpdfstring{Propagating for small $\log(1/h)$ times: expansion of the integrand}{Propagating for small log(1/h) times: expansion of the integrand}}\label{section1st}
	\subsection{Context of the computation} \label{context}
	Here, we will show how to propagate a squeezed state during a time $t_0$ independent of $h$ with the propagator $e^{-it_0\mathpzc{p}^w/h}$. In the application, the initial data will be given by a coherent state propagated over a time $t_s \leq \epsilon_s T_{\text{Ehrenfest}}$. 
	\medbreak 
	We want to show that after a new propagation over a time $t_0$, the state can still be written in a similar fashion with adequate estimates. This method, along with the one developed in the next section, will be the elementary steps that we will repeat some $C|\log h|$ times in section \ref{iter_method}.
	\medbreak
	Results will be expressed in the charts described in definition \ref{coord} so that they fit some base points $\tilde{\rho}^0$ and $\tilde{\rho}^1=\Phi^{t_0}(\tilde{\rho}^0)$.
	\medbreak
	Let us denote by $U^{t_0}$ the propagator associated to $\mathpzc{p}$ seen in these charts:
	\begin{equation} \label{propagator}
		U^{t_0}=\mathcal{U}_{\tilde{\rho}^1}e^{-it_0\mathpzc{p}^w/h} \mathcal{U}_{\tilde{\rho}^0}^{-1},
	\end{equation}
	which is a compactly microlocalized Fourier integral operator. It quantizes the symplectomorphism 
	\begin{equation}\label{def F}
		F=\kappa_{\tilde{\rho}^1} \Phi^{t_0} (\kappa_{\tilde{\rho}^0})^{-1},
	\end{equation} 
	which corresponds to the flow $\Phi^{t_0}$ written between two charts $U_{\tilde{\rho}^0}$, $U_{\tilde{\rho}^1}$. Due to the construction of the charts, we have a special form for the Jacobian matrix of $F$ at $0$:
	\begin{equation}\label{dF0}
		d_0F= \begin{pNiceArray}{cccc}[last-row,last-col]
			A &0 & E & 0 & x
			\\0 & B & 0 & 0 & y
			\\F & 0 & C & 0 & \xi 
			\\0 & 0 & 0 & D & \eta
			\\ x & y & \xi & \eta &
		\end{pNiceArray},
	\end{equation}
	with $\|B^{-1}\|,\|D\|\leq e^{(\nu_{\text{min}}-\epsilon_2/3)t_0}\coloneq\nu<1$ and $\|A\|,\|E\|,\|F\|,\|C\|\leq e^{(\lambda_{\text{c}}-\epsilon_2/3)t_0}\coloneq\mu<\nu^{-1}$ thanks to our choice of $t_0$ in (\ref{lambda max et t0}). 
	\newline	
	In this first method, we will not use any specificity of those coordinates, but they will be useful in the next section.
	\medbreak 
	In these coordinates, we choose the initial data to be some 
	\begin{equation*} 
		u= \hat{T}\left(\boldsymbol{\rho^0}\right)\mathcal{M}_h(\kappa)\Lambda_h(P\Psi_0),
	\end{equation*} 
	with $\boldsymbol{\rho^0}=\kappa_{\tilde{\rho}^0}(\rho^0)$ and a $\kappa$ that might have a norm going to infinity as $h$ goes to $0$ (in a $h^{-\epsilon}$ fashion) but verifies $\|\kappa\|\leq Ch^{-1/6+\cdot}$ as required in section \ref{CR method}.
	\medbreak
	We write the integral representation of $U^{t_0}$ as a compactly microlocalized Fourier Integral Operator:
	\begin{equation}\label{rep int}
		U^{t_0}u(X^1)=(2\pi h)^{-(2d+\mathpzc{k})/2}\int_{\R^{2d}}\int_{\R^{\mathpzc{k}}} e^{\frac{i}{h}\left(\psi(X^1,\Xi^0,\Theta)-X^0\Xi^0\right)}a\left(X^1,\Xi^0,\Theta\right)u(X^0)\d X^0 \d \Xi^0\d \Theta,
	\end{equation}
	with $\psi$ a generating function of $F$ with $\mathpzc{k}$ auxiliary variables, $a$ an amplitude in $C^\infty_c(\R^{2d+\mathpzc{k}})$ and $u\in L^2\left(\R^d\right)$.
	\medbreak
	We aim at proving a generalization of \cite[Proposition 2.10]{Lucas} to a $d$-dimensional (and $\mathpzc{k}\neq 0$) setting: 
	\begin{Proposition}[Isotropic case]\label{iter iso}
		Let $\kappa$ be a symplectic linear map and $P\in \C[X]$, then there exists a family of polynomials $(Q_k(P))_{k\in \N}$ such that:
		\begin{itemize}
			\item $Q_0(P)=P,$ 
			\item $Q_k(P)$ is a polynomial of degree $\deg P+3k$ and the map $P\mapsto Q_k(P)$ is linear, with coefficients depending on $\kappa$ and the derivatives of $\psi$ and $a$ at point $(q^1,p^0,\Theta^0)$ up to the $3k$-th order and we have:
			\begin{equation*} 
				N_\infty\bigl(Q_k(P)\bigr)\leq C\|a\|_{C^k} \|\kappa \|^{3k} N_\infty(P).
			\end{equation*} 
			Moreover, if $(q^1,p^0,\Theta^0)\notin \supp a$, then $Q_k=0,$
			\item for every $N\in \N$,
			\begin{equation*} 
				U^{t_0}\Bigl(\hat{T}\left(\kappa_{\tilde{\rho}^0}(\rho^0)\right)\mathcal{M}_h(\kappa)\Lambda_h[P\Psi_0]\Bigr)= \hat{T}\left(\kappa_{\tilde{\rho}^1}(\rho^1)\right) \mathcal{M}_h\left(d_{\boldsymbol{\rho^0}} F\circ \kappa\right) \Lambda_h \left[\sum_{k=0}^{N/2} h^{k/2} Q_k(P) \Psi_0\right]+R_N,
			\end{equation*} 
			with
			\begin{equation*} 
				\|R_N\|_{L^2}\leq Ch^{N/2}  \|a\|_{C^{N+M}} \|\kappa\|^{3N}  N_\infty(P),
			\end{equation*} 
			where $N_\infty(P)$ is the sup norm on the coefficients of $P$, $M$ universal constant, and the constants $C$ depend on the norms of $\psi$ and its derivatives.
		\end{itemize}
	\end{Proposition}
	\begin{Remark}\label{CV kappa petit}
		Under the assumption $\|\kappa\|\leq Ch^{-1/6+\cdot}$, we obtain 
		\begin{equation*} 
			\|R_N\|_{L^2}\leq Ch^{(0+)N}  \|a\|_{C^{N+M}}   N_\infty(P),
		\end{equation*}
	 	so that the remainder can be made arbitrarily small by taking $N$ large enough.
	\end{Remark}
	\subsection{Proof of proposition \ref{iter iso}}
	Due to our specific choice of initial data, we make the Ansatz that after propagation with $U^{t_0}$ as considered in this section, the state can be written as some $\hat{T}\left(\kappa_{\tilde{\rho}^1}(\rho^1)\right)\mathcal{M}(\tilde{\kappa})\Lambda_h [Q\Psi_0]$ for $\rho^1=\Phi^{t_0}(\rho^0)$ given by the classical propagation ($t_0$ defined above in section \ref{preli_dyn_clas}) and $\tilde{\kappa}, Q$ to be determined.
	\medbreak
	Using this Ansatz, we can go to microscopic scaling to do our computations. Set $u_0=\mathcal{M}_1(\kappa)(P\Psi_0)$, with $\|\kappa\|\leq Ch^{-1/6+\cdot}$. We want to understand
	\begin{equation*} 
		u_1=\left(\Lambda_h\right)^* \hat{T}\left(\kappa_{\tilde{\rho}^1}(\rho^1)\right)^* U^{t_0} \hat{T}\left(\kappa_{\tilde{\rho}^0}(\rho^0)\right) \Lambda_h u_0.
	\end{equation*} 
	
	Let us write $\boldsymbol{\rho^1}=\kappa_{\tilde{\rho}^1}\left(\rho^1\right)=\left(q^1,p^1\right)$ and $\boldsymbol{\rho^0}=\kappa_{\tilde{\rho}^0}\left(\rho^0\right)=\left(q^0,p^0\right)$ and define $\Theta^0$ such that:
	\begin{equation*} 
		\partial_\Theta \psi\left(q^1,p^0,\Theta^0\right)=0,\quad \partial_{X^1} \psi\left(q^1,p^0,\Theta^0\right)=p^1,\quad\partial_{\Xi^0} \psi\left(q^1,p^0,\Theta^0\right)=q^0. 
	\end{equation*} 
	Then, we use (\ref{rep int}) and write
	\begin{equation}\label{use rep int}
	\begin{multlined}[.8\textwidth]
			\Bigl(\Lambda_h^{*} T(\boldsymbol{\rho^1})^*U^{t_0} T(\boldsymbol{\rho^0}) \Lambda_hu_0\Bigr)\left(X^1\right)=h^{d/4}e^{-\frac{i X^1p^1}{\sqrt{h}}}\Bigl(U^{t_0} T(\boldsymbol{\rho^0}) \Lambda_h u_0\Bigr)\left(q^1+\sqrt{h}X^1\right)
			\\\shoveleft{=\frac{e^{-\frac{i X^1p^1}{\sqrt{h}}}}{(2\pi h)^{d+\mathpzc{k}/2}} \int_{\R^{2d+\mathpzc{k}}} e^{\frac{i}{h}\bigl(\psi(q^1+\sqrt{h}X^1,\Xi^0,\Theta)-X^0(\Xi^0-p^0)\bigr)}a\left(q^1+\sqrt{h}X^1,\Xi^0,\Theta\right)     u_0\left(\frac{X^0-q^0}{\sqrt{h}}\right) }
			\\\shoveright{\d X^0 \d \Xi^0 \d \Theta}
			\\\shoveleft{=\frac{1}{(2\pi )^{d+\mathpzc{k}/2}} \int_{\mathbb{R}^{2d+\mathpzc{k}}} e^{i\psi_{\text{tot}}\bigl(X^1,\Xi^0,X^0,\Theta\bigr)} a\left(q^1+\sqrt{h}X^1,p^0+\sqrt{h}\Xi^0,\Theta^0+\sqrt{h}\Theta\right) u_0\left(X^0\right)}
			\\ \d X^0 \d \Xi^0\d\Theta
	\end{multlined}
	\end{equation}
	where we have made the change of variables 
	\begin{equation*} 
		X^0\leftarrow h^{-1/2}\left(X^0-q^0\right), \Xi^0\leftarrow -p^0+ h^{-1/2}\Xi^0, \Theta\leftarrow -\Theta^0+h^{-1/2}\Theta
	\end{equation*}
	to locate at scale $\sqrt{h}$ both our variables around the point where the initial coherent state is located, this gives us a total phase of:
	\begin{equation*} 
		\psi_{\text{tot}}\left(X^1,\Xi^0,X^0,\Theta\right)=\frac{1}{h} \psi\left(q^1+\sqrt{h}X^1,p^0+\sqrt{h}\Xi^0,\Theta^0+\sqrt{h}\Theta\right)-X^0\Xi^0-h^{-1/2}\left(X^1p^1+q^0\Xi^0\right).
	\end{equation*} 
	\begin{Remark}
		After this change of variables, every variable has microscopic scaling. We expect them to have at most size 
		\begin{equation*} 
			X^0,X^1,\Xi^0,\Theta\sim \|\kappa\|\leq C h^{-1/6+\cdot}.
		\end{equation*} 
	\end{Remark}
	Now, we will make a Taylor expansion of the phase and the amplitude in the integral.
	\medbreak
	Let us start by the phase, we write the Taylor expansion of 
	\begin{equation*} 
		\psi\left(q^1+\sqrt{h}X^1,p^0+\sqrt{h}\Xi^0,\Theta^0+\sqrt{h}\Theta\right)
	\end{equation*}
 	at order $N+1\in \N$ (in $\sqrt{h}$) as:
	\begin{align*}
		\psi\bigl(q^1+\sqrt{h}X^1&,p^0+\sqrt{h}\Xi^0,\Theta^0+\sqrt{h}\Theta\bigr)=\psi\left(q^1,p^0,\Theta^0\right)+\sqrt{h}\Bigl(X^1.\nabla_{X^1}\psi\left(q^1,p^0,\Theta^0\right)+
		\\&\Xi^0.\nabla_{\Xi^0}\psi\left(q^1,p^0,\Theta^0\right)+\Theta.\nabla_\Theta \psi(q^1,p^0,\Theta^0)\Bigr)
		+h\frac{\text{Hess }\psi\left(q^1,p^0,\Theta^0\right)}{2}.\bigl[X^1,\Xi^0,\Theta\bigr]^{\otimes 2}
		\\&+\sum_{k=3}^{N+1} h^{k/2} \psi_k\left(X^1,\Xi^0,\Theta\right) +h^{(N+2)/2} r_{N+2}^\psi\left(X^1,\Xi^0,\Theta\right),
	\end{align*}
	with $k!\psi_k$ the $k$-th differential of $\psi$ applied to $\left(X^1, \Xi^0,\Theta\right)^{\otimes k}$, so that it is $k$-linear in these variables and 
	\begin{multline*}
		r_{N+2}^\psi\left(X^1,\Xi^0,\Theta\right)=
		\\\frac{1}{N+1}\int_0^1 (1-s)^{N+1}D^{N+2} \psi\left(q^1+s\sqrt{h}X^1,p^0+s\sqrt{h}\Xi^0,\Theta^0+s\sqrt{h}\Theta\right). \Bigl[X^1,\Xi^0,\Theta\Bigr]^{\otimes(N+2)}\d s
	\end{multline*}
	whose derivatives can be controlled by:
	\begin{equation*} 
		\left|\partial^\alpha r^\psi_{N+2}(X^1,\Xi^0,\Theta)\right|\leq C \|\psi\|_{C^{N+2+|\alpha|}}\Bigl\langle\left(X^1,\Xi^0,\Theta\right)\Bigr\rangle^{N+2}.
	\end{equation*} 
	There are a few important things to note about the development of $\psi_{\text{tot}}$: firstly, the term of order $1/h$ will not be too important for us as it is independent of our variable, and so it only adds a global phase. As we work with metaplectic operators, every calculus is made up to a factor in $\mathbb{U}^1$, and so, we can forget this term.
	\medbreak
	Then, we notice that the terms of order $h^{-1/2}$ cancel out, this is expected in our computations since these terms are responsible for the localization of our function in phase space. As we have located our computations around the starting and end points thanks to $T\left(\boldsymbol{\rho^1}\right)$ and $T\left(\boldsymbol{\rho^0}\right)$, the classical transport is already encoded, so it is not surprising that transport is no longer needed. 
	\medbreak
	As we shall see, the terms that are of order $1$ will be responsible for the quadratic deformation of our coherent state, encoding a metaplectic operator.
	\medbreak
	The higher order terms will act as further deformations of this coherent state, making its description more and more precise, fitting even more to the unstable manifold that will be described later on. The important part is that, for now, we will be able to treat these terms as corrective ones.
	\medbreak
	Let us group up these higher order terms in $c_3^\psi$ (corrections of cubic order and more): 
	\begin{equation*} 
		\psi_{\text{tot}}\left(X^1,\Xi^0,X^0,\Theta\right)=\underbrace{\text{Hess } \psi\left(q^1,p^0,\Theta^0\right).\Bigl[X^1,\Xi^0,\Theta\Bigr]^{\otimes 2}-X^0\Xi^0 }_{\psi_2(X^1,X^0,\Xi^0,\Theta)}+\sqrt{h}c^\psi_3\left(X_1,\Xi^0,\Theta\right).
	\end{equation*} 
	We now inject this expansion into the exponential and we get:
	\begin{align} \label{DL exp}
		\begin{split}
			e^{i\sqrt{h}c_3^\psi\left(X^1,\Xi^0,\Theta\right)}=\sum_{k=0}^{N-1} &\frac{(i\sqrt{h})^k}{k!} c_3^\psi\left(X^1,\Xi^0,\Theta\right)^k
			\\&+\frac{(i\sqrt{h})^N}{(N-1)!} c_3^\psi\left(X^1,\Xi^0,\Theta\right)^N \int_0^1 e^{i\sqrt{h}sc_3^\psi\left(X^1,\Xi^0,\Theta\right)}(1-s)^{N-1}\d s
		\end{split}
	\end{align}
	\begin{Remark}
		Expanding of $e^{i\sqrt{h}c_3^\psi}$ is justified by the fact that we expect $\sqrt{h} c^\psi_3\left(X_1,\Xi^0,\Theta\right)$ to be $o(1)$, it is a consequence of $X^1,\Xi^0 \sim \|\kappa\|^2 \leq C h^{-1/6+\cdot}$.
	\end{Remark}
	As $c_3^\psi$ contains not only cubic corrections, but also corrections of higher order, we should regroup terms having the same order by developing the powers in the expansion of the exponential.
	\medbreak
	We obtain:
	\begin{equation*} 
		e^{i\sqrt{h}c_3^\psi\left(X^1,\Xi^0,\Theta\right)}=\sum_{k=0}^{N-1} h^{k/2}P_k\left(X^1,\Xi^0,\Theta\right)+h^{N/2}r_N^{\text{corr}}+r^{\text{exp}}_N,
	\end{equation*} 
	with
	\begin{itemize}
		\item[$\bullet$] $P_k$ polynomial of degree $3k$ in $\left(X^1, \Xi^0,\Theta\right)$ coefficients involving derivatives of $\psi$ taken at the point $\left(q^1,p^0,\Theta^0\right)$ up to the $3+k$-th order.
		\item[$\bullet$] $r_N^{\text{corr}}$ coming from corrections of order $N$ and more involved in the expansion
		\item[$\bullet$] $r^{\text{exp}}_N$ integral rest form the Taylor expansion of the exponential.
	\end{itemize}
	We now write the Taylor expansion of $a$ and multiply it by the other one to get:
	\begin{align}\label{def remainders}
		&e^{i\sqrt{h}c_3^\psi\left(X^1,\Xi^0,\Theta\right)}a\left(\sqrt{h}X^1,\sqrt{h}\Xi^0,\Theta\right)=\sum_{k=0}^{N-1}h^{k/2} \tilde{P}_k\left(X^1,\Xi^0,\Theta\right)+\underbrace{r^{\text{exp}}_Na\left(\sqrt{h}X^1,\sqrt{h}\Xi^0,\sqrt{h}\Theta\right)}_{r^{\text{type }2}_N}
		\\&+\underbrace{h^{N/2} \sum_{k=0}^{N-1} P_k\left(X^1,\Xi^0,\Theta\right) r^a_{N-k}\left(X^1,\Xi^0,\Theta\right)+h^{N/2} R_{N}^{\text{corr}}a\left(\sqrt{h}X^1,\sqrt{h}\Xi^0,\sqrt{h}\Theta\right)}_{h^{N/2} r^{\text{type }1}_N}
	\end{align}
	Where $\tilde{P}_k$ is a polynomial of degree $3k$ in $(X^1,\Xi^0,\Theta)$ given by
	\begin{equation*} 
		\tilde{P_k}=\sum_{k_1+k_2=k} P_{k_1}\left(X^1,\Xi^0,\Theta\right)\times \left( \frac{1}{k_2 !} D^{k_2}a\left(q^1,p^0,\Theta^0\right). \Bigl[X^1,\Xi^0,\Theta\Bigr]^{\otimes k_2}\right)
	\end{equation*} 
	and $r_p^a$ being the integral of the Taylor expansion of a up to order $p-1$:
	\begin{equation*} 
		r_p^a=\frac{1}{p!}\int_0^1 (1-s)^{p-1}D^{p} a\left(q^1+s\sqrt{h}X^1,p^0+s\sqrt{h}\Xi^0,\Theta^0+s\sqrt{h}\Theta\right).\Bigl[X^1,\Xi^0,\Theta\Bigr]^{\otimes p}\d s.
	\end{equation*} 
	\medbreak
	We are now ready to compute the different integrals involved in the expansion, using the special form of $u_0$.
	\medbreak
	\paragraph{Leading term:} The leading term is 
	\begin{equation*} 
		\frac{1}{(2\pi)^{d+\mathpzc{k}/2}} \int_{\mathbb{R}^{2d+\mathpzc{k}}}e^{i\left(\text{Hess }\psi\left(q^1,p^0,\Theta^0\right).\bigl[X^1,\Xi^0,\Theta\bigr]^{\otimes 2}/2-X^0\Xi^0\right)} a\left(q^1,p^0,\Theta^0\right) u_0\left(X^0\right)\d X^0 \d \Xi^0\d \Theta.
	\end{equation*} 
	By differentiating $F\bigl(\partial_{\Xi^0} \psi\left(X^1,\Xi^0,\Theta\right),\Xi^0\bigr)=\bigl(X^1,\partial_{X^1} \psi\left(X^1,\Xi^0,\Theta\right)\bigr)$,
	we can see that the quadratic form 
	\begin{equation*} 
		\psi_2(X^1,X^0,\Xi^0,\Theta)=\frac{\text{Hess }\psi\left(q^1,p^0,\Theta^0\right)}{2}.\Bigl[X^1,\Xi^0,\Theta\Bigr]^{\otimes 2}-X^0\Xi^0
	\end{equation*} 
	is a generating function for the linear map $d_\rho F$, or equivalently $d_\rho F$ is given by:
	\begin{equation*} 
		d_\rho F= \begin{pmatrix}
			(\partial^2_{X,\Xi} \psi)^{-1} & - (\partial^2_{X,\Xi}\psi)^{-1} \partial^2_{\Xi,\Xi} \psi
			\\ \partial^2_{X,X} \psi (\partial^2_{X,\Xi}\psi)^{-1} & \partial^2_{X,\Xi} \psi -\partial^2_{X,X} \psi (\partial^2_{X,\Xi} \psi)^{-1} \partial^2_{\Xi,\Xi} \psi.
		\end{pmatrix}
	\end{equation*} 
	
	So, the main order term of the expansion is given by the action of a metaplectic operator $\mathcal{M}_1\left(d_{\boldsymbol{\rho^0}}F\right)$ on $u_0$:
	\begin{multline*} 
		U^{t_0} T\left(\boldsymbol{\rho^0}\right)\mathcal{M}_h(\kappa)\Lambda_h[P\Psi_0]=T\left(\boldsymbol{\rho^1}\right)\mathcal{M}_h\left(d_{\boldsymbol{\rho^0}}F\circ \kappa\right) \Lambda_h\left[\frac{a\left(q^1,p^0,\Theta^0\right)}{D_\psi\left(q^1,p^0,\Theta^0\right)^{1/2}}P\Psi_0\right]
		\\\shoveright \phantom{} +\text{smaller terms}.
	\end{multline*} 
	Here, we came back to the original scaling by multiplying with $\Lambda_h$ which changed our metaplectic operators into ``semiclassical'' one going from $h=1$ to $h$. Moreover, if we also consider an expansion of $a$ in powers of $h$, we know that the main order term is such that $\left|a^0(q^1,p^0,\Theta^0)\right|=D_\psi\left(q^1,p^0,\Theta^0\right)^{1/2}$, meaning that at order $0$ there is no change of the amplitude.
	\medbreak

	\paragraph{Terms of higher order:}\label{Theta higher order}
	To understand them, we should study 
	\begin{equation*} 
		\frac{1}{(2\pi)^{d+\mathpzc{k}/2}} \int_{\R^{2d+\mathpzc{k}}} e^{i\psi_2(X^1,X^0,\Xi^0,\Theta)} P_k\left(X^1,\Xi^0,\Theta\right) u_0\left(X^0\right)\d X^0 \d \Xi^0.
	\end{equation*} 
	Let us start with the study of those terms when $P_k$ is independent of $\Theta$. To do so, we focus on monomial terms:
	\begin{multline*} 
		I_{\alpha,\beta}(u_0)=\frac{1}{(2\pi)^{d+\mathpzc{k}/2}} \int_{\R^{2d+\mathpzc{k}}} e^{i\bigl(\text{Hess }\psi\left(q^1,p^0,\Theta^0\right).\bigl[X^1,\Xi^0,\Theta\bigr]^{\otimes 2}/2 -X^0\Xi^0\bigr)} \prod_{i=1}^d (X^1_i)^{\alpha_i}(\Xi^0_i)^{\beta_i} u_0\left(X^0\right)
		\\\shoveright \phantom{} \d X^0 \d \Xi^0,
	\end{multline*} 
	with $|\alpha|+|\beta|\leq 3k$. We can take the $\left(X^1_i\right)^{\alpha_i}$ terms out of the integral and use the properties of the Fourier transform $\mathcal{F}$ to get
	\begin{align*}
		I_{\alpha,\beta}(u_0)&=\frac{\prod_{i=1}^d (X^1_i)^{\alpha_i}}{(2\pi)^{(d+\mathpzc{k})/2}}\int_{\R^{d+\mathpzc{k}}}e^{i\text{Hess }\psi\left(q^1,p^0,\Theta^0\right). \bigl[X^1,\Xi^0,\Theta\bigr]^{\otimes 2}/2}\prod_{i=1}^d \left(\Xi^0_i\right)^{\beta_i} \mathcal{F}u_0\left(\Xi^0\right)\d \Xi^0 \d \Theta
		\\&=\frac{\prod_{i=1}^d \left(X^1_i\right)^{\alpha_i}}{(2\pi)^{(d+\mathpzc{k})/2}}\int_{\R^{d+\mathpzc{k}}}e^{i\text{Hess }\psi\left(q^1,p^0,\Theta^0\right) .\bigl[X^1,\Xi^0,\Theta\bigr]^{\otimes 2}/2} \mathcal{F}\bigl((-i)^{|\beta|} \partial_{\Xi^0}^\beta u_0\bigr)(\Xi^0)\d \Xi^0 \d \Theta
		\\&=\prod_{i=1}^d (X^1_i)^{\alpha_i} \mathcal{M}_1\left(d_{\boldsymbol{\rho^0}} F\right) \bigl((-i)^{|\beta|}\partial_{\Xi^0}^\beta u_0\bigr).
	\end{align*}
	We then use the exact Egorov property of metaplectic operators and recall the special form of $u_0$:
	\begin{align*}
		I_{\alpha,\beta}\Bigl(\mathcal{M}_1(\kappa)[P\Psi_0]\Bigr)&=\text{Op}_1(\prod_{i=1}^d \left(X^1_i\right)^{\alpha_i}) \mathcal{M}_1\left(d_{\boldsymbol{\rho^0}}F\right)\text{Op}_1\bigl((\Xi^0)^\beta\bigr) \mathcal{M}_1(\kappa)[P\Psi_0]
		\\&=\mathcal{M}_1\left(d_{\boldsymbol{\rho^0}}F\circ \kappa\right) \text{Op}_1\Bigl(\bigl[d_{\boldsymbol{\rho^0}}F\circ \kappa(x^0,\xi^0)\bigr]_X^\alpha\Bigr) \text{Op}_1\Bigl(\bigl[\kappa(x^0,\xi^0)\bigr]_\Xi^\beta\Bigr)[P\Psi_0].
	\end{align*}
	The action of the first two operators transforms $P\Psi_0$ in another excited state $Q(P)\Psi_0$ with $Q(P)$ of degree $\deg P+|\alpha|+|\beta|$, by developing the powers involved, we can bound the coefficients of $Q$ by $C\left\|d_{\boldsymbol{\rho^0}} F\right\|^{\|\alpha|} \|\kappa\|^{|\alpha|+|\beta|}$.
	\medbreak
	Now, let us focus on the study of the general case by including monomials in $\Theta$:
	\begin{multline*} 
		I_{\alpha,\beta,\gamma}=\frac{1}{(2\pi)^{d+\mathpzc{k}/2}} \int_{\R^{2d+\mathpzc{k}}} e^{i\left(\text{Hess }\psi\left(q^1,p^0,\Theta^0\right).\bigl[X^1,\Xi^0,\Theta\bigr]^{\otimes 2}/2 -X^0\Xi^0\right)} \prod_{i=1}^d \left(X^1_i\right)^{\alpha_i}\left(\Xi^0_i\right)^{\beta_i} \theta_i^{\gamma_i} u_0(X^0)
		\\\shoveright \phantom{}\d X^0 \d \Xi^0\d \Theta.
	\end{multline*} 
	Our goal is to translate the effect of $\Theta$ into a multiplication by a polynomial of $X^1,\Xi^0,X^0$ using an integration by parts.
	\medbreak
	To do so, we compute (denoting $e^{i\text{Hess }\psi\left(q^1,p^0,\Theta^0\right).\bigl[X^1,\Xi^0,\Theta\bigr]^{\otimes 2}/2}$ by $e_\psi$):
	\begin{equation*} 
		\nabla_X e_\psi=ie_\psi\left[\frac{\partial^2 \psi}{\partial X^2}.X^1 +\frac{\partial \psi}{\partial X \partial \Theta}.\Theta +\frac{\partial \psi}{\partial X \partial `\Xi}.\Xi^0\right].
	\end{equation*} 
	But recall that by assumption,
	\begin{equation*} 
		\frac{\partial^2 \psi }{\partial X'\partial \Theta} \text{ is invertible }
	\end{equation*} 
	($\psi$ can even be chosen so that this matrix is $I_{\mathpzc{k}}$, see \ref{hyp psi}).
	Hence,
	\begin{equation*} 
		e_\psi\Theta=-i\left(\frac{\partial^2 \psi }{\partial X'\partial \Theta}\right)^{-1}\left[\nabla_{X'} e_\psi -ie_\psi\frac{\partial\psi}{\partial X \partial X'}X-ie_\psi\frac{\partial \psi}{\partial X' \partial \Xi}. \Xi^0\right]
	\end{equation*} 
	and 
	\begin{equation*} 
		\Theta^{\gamma_i}=\left(-ie_\psi^{-1} \left(\frac{\partial^2 \psi }{\partial X'\partial \Theta}\right)^{-1}\left[\nabla_{X'} e_\psi -ie_\psi\frac{\partial\psi}{\partial X' \partial X}X-ie_\psi\frac{\partial \psi}{\partial X' \partial \Xi}. \Xi^0\right]_\Theta \right)^{\gamma_i},
	\end{equation*} 
	where $\bigl[\ \cdot\ \bigr]_\Theta$ means that we only keep the last $\mathpzc{k}=\text{dim}(\Theta)$ coordinates.
	\medbreak
	Using integration by part, we conclude that:
	\begin{equation*} 
		I_{\alpha,\beta,\gamma}= \frac{1}{2\pi}\int_{\R^{2d+\mathpzc{k}}} e_\psi P_{\alpha,\beta,\gamma}\left(X^1,\Xi^0,X^0\right) u_0\left(X^0\right) \d X^0 \d \Xi^0\d \Theta,
	\end{equation*} 
	with $P_{\alpha,\beta,\gamma}$ a polynomial of total degree $|\alpha|+|\beta|+|\gamma|$ with bounded coefficients.
	\medbreak
	This new polynomial can be estimated in the same way as before, indeed, the presence of $X^0$ is essentially harmless as it has the same behavior as $\Xi^0$ and just add a polynomial that behaves in the same way.
	\medbreak
	As a consequence, by grouping those terms into one polynomial, we can write the term of $k$-th order in the approximation of 
	\begin{equation*} 
		U^{t_0} T\left(\boldsymbol{\rho^0}\right)\mathcal{M}_h(\kappa)\Lambda_h[P\Psi_0]
	\end{equation*} as 
	\begin{equation*} 
		h^{k/2}T\left(\boldsymbol{\rho^1}\right)\mathcal{M}_h\left(d_{\boldsymbol{\rho^0}}F\circ \kappa\right)\Lambda_h[Q_k(P)\Psi_0],
	\end{equation*}
	 with $Q_k(P)$ polynomial of order $\deg P +3k$, linear in $P$. Putting $\|d_{\boldsymbol{\rho^0}} F\|^{\alpha}$ in the constant (we propagate only during a finite time), we get the following estimate using our control on the coefficients of $P_k$:
	\begin{equation*} 
		N_\infty\bigl(Q_k(P)\bigr)\leq C \|a\|_{C^k}N_\infty(P)\|\kappa\|^{3k} .
	\end{equation*} 
	This gives us control over the $L^2$ norm of this $k$-th term:
	\begin{equation} \label{higher order}
		\left\|h^{k/2}T\left(\boldsymbol{\rho^1}\right)\mathcal{M}_h\left(d_{\boldsymbol{\rho^0}}F\circ \kappa\right)\Lambda_h[Q_k(P)\Psi_0]\right\|_{L^2}\leq C \|a\|_{C^k}N_\infty(P)h^{k/2}\|\kappa\|^{3k},
	\end{equation}
	which let us check that each term is indeed smaller than the previous one, as $\|\kappa\|\leq Ch^{-1/6+\cdot}$.
	\medbreak

	\paragraph{Control of the first remainder term:}
	This first type remainder $r^{\text{type }1}_N$ defined in (\ref{def remainders}) contains all terms of order higher than $N$: they either come from the product of an order $k$ of the exponential and a rest of order greater than $N-k$ for $a$ or from remainders that we obtained by injecting our expansion into the exponential $r_N^{\text{corr}}$.
	\medbreak
	All these terms are sums of quantities that can be controlled with 
	\begin{equation}\label{symbol remainder}
		|\partial^\alpha r^{\text{type }1}_N\left(X^1,\Xi^0, \Theta\right)|\leq C \|a\|_{C^{N+|\alpha|}} \bigl\langle X^1,\Xi^0,\Theta\bigr\rangle^{3N},
	\end{equation}
	as they are products of polynomials and rests for which we have the right complementary ``sub-polynomial behavior'' to obtain a power $3N$.
	\medbreak
	This tells us that $r^{\text{type }1}_N$ is in some sense a symbol in the class $S\left(\bigr\langle X^1,\Xi^0,\Theta\bigr\rangle^N\right)$.
	So, it could be interesting to see our integral as the composition of a metaplectic operator and a pseudodifferential operator with this symbol: we will control its norm thanks to the unitarity of the metaplectic operator and translate our control on the symbol into a control of the pseudodifferential operator.
	\medbreak
	\begin{Lemma}\cite[Lemma 2.5]{Lucas}\label{1st_type}
		Let $b\in S\left(\jp{X^1,\Xi^0,\Theta}^N\right)$ then there exists another symbol $\tilde{b}$ with
		$\tilde{b}\in S\left(\bigl\langle X^0,\Xi^0\bigr\rangle^N\right)$ such that for all $0<h\leq 1,$
		\begin{equation*} 
			\frac{1}{(2\pi h)^d}\int_{T^*\R^{d}} e^{\frac{i}{h} \psi_2(X^1,X^0,\Xi^0,\Theta)} b\left(X^1,\Xi^0,\Theta\right) u_0\left(X^0\right)\d X^0\d \Xi^0= \mathcal{M}_h\left(d_{\boldsymbol{\rho^0}} F\right)\Op (\tilde{b}) u_0(x). 
		\end{equation*} 
		Moreover, we can control the symbol $\tilde{b}$ using $b:$ there exists a universal integer $M'\in\N$ such that for all $\alpha\in \N^2,$
		\begin{equation*} 
			\jp{\rho}^{-N} \left|\partial^\alpha \tilde{b}(\rho)\right|\leq C_\alpha \sup_{|\beta|\leq|\alpha|+M'} \sup_{\rho\in T^*\R^d} \left(|\partial^\beta b(\rho)| \jp{\rho}^{-N}\right).
		\end{equation*} 
	\end{Lemma}
	The lemma proved by Vacossin is actually in the one-dimensional setting without auxiliary variable, but the proof remains the same higher dimension, and auxiliary variables can be dealt with in a similar fashion as in the previous paragraph about terms of higher order.
	\medbreak
	Using this lemma for $h=1$, we can find a $\tilde{r}^{\text{type }1}_N$ such that 
	
	\begin{align*}
		R_N^{\text{type }1}:&=\frac{1}{(2\pi h)^{d+\mathpzc{k}/2}}\int_{\R^{2d+\mathpzc{k}}} e^{\frac{i}{h}\psi_2(X^1,X^0,\Xi^0,\Theta)} r^{\text{type }1}_N\left(X^1,\Xi^0,\Theta\right) u_0\left(X^0\right)\d X^0\d \Xi^0\d \Theta \\&=\mathcal{M}_1\left(\d_{\boldsymbol{\rho^0}}F\right) \text{Op}_1\left(\tilde{r}^{\text{type }1}_N\right)u_0.
	\end{align*}
	This gives:
	\begin{align}\label{rest 1st}
		\begin{split}
			\left\|R_N^{\text{type }1}\right\|_{L^2}&\leq \left\|\mathcal{M}_1\left(\d_{\boldsymbol{\rho^0}}F\right) \text{Op}_1\left(\tilde{r}^{\text{type }1}_N\right)u_0\right\|_{L^2}
			\\&\leq \left\|\text{Op}_1\left(\tilde{r}^{\text{type }1}_N\right)\right\|_{H_1(\jp{\rho}^{3N})\to L^2}\left\|\mathcal{M}_1(\kappa)[P\Psi_0]\right\|_{H_1(\jp{\rho}^{3N})}
			\\&\leq C\| a\|_{C^{N+M+M'}} \|\kappa\|^{3N} K_{N,\deg P} N_\infty(P).
		\end{split}
	\end{align}
	Using our control on $r^{\text{type }1}_N$ and its seminorms in $S\left(\jp{\rho}^{3N}\right)$ plus lemma \ref{control_xy} controlling
	\newline $\|u_0\|_{H_1(\jp{\rho}^{3N})}=\|\mathcal{M}_1(\kappa)[P\Psi_0]\|_{H_1(\jp{\rho}^{3N})}$.
	\medbreak

	\paragraph{Control of the second remainder term:}
	We are left with understanding
	\begin{equation*} 
		\frac{1}{(2\pi)^{d+\mathpzc{k}/2}} \int_{\R^{2d}}\int_{\R^\mathpzc{k}} e^{i\psi_2(X^1,X^0,\Xi^0,\Theta)} r_N^{\text{type }2}\left(X^1,\Xi^0,\Theta\right) u_0\left(X^0\right)\d X^0\d \Xi^0\d \Theta.
	\end{equation*} 
	where $r_N^{\text{type }2}$ defined in (\ref{def remainders}) is a integral itself: we want to interchange the order of integration of these integrals. We will look for an uniform estimate in $s\in[0,1]$ of the $L^2$ norm of:
	\begin{multline}\label{type 2}
		\tilde{R}_s u_0(x)=\frac{1}{(2\pi)^{d+\mathpzc{k}/2}} \int_{\R^{2d}}\int_{\R^\mathpzc{k}}e^{i \psi_2(X^1,X^0,\Xi^0,\Theta)+is\sqrt{h} c_3^\psi\left(X^1,\Xi^0,\Theta\right)} b_N\left(X^1,\Xi^0,\Theta\right)u_0\left(X^0\right)
		\\ \shoveright \phantom{} \d X^0\d \Xi^0\d \Theta,
	\end{multline}
	with $b_N\left(X^1,\Xi^0,\Theta\right)=c_3^\psi\left(X^1,\Xi^0,\Theta\right)^N a\left(q_1+\sqrt{h}X^1,p_0+\sqrt{h}\Xi^0,\Theta\right)$ which is in the 
	 $S\left(\jp{\rho}^{3N}\right)$ symbol class thanks to estimates such as (\ref{symbol remainder}).
	\medbreak
	Since the phase changed, this no longer looks like a metaplectic operator applied to a pseudo-operator. However, as we are looking for the $L_2$ norm, it is interesting to look at $\tilde{R}_s^*\tilde{R}_s$ which by the rules of composition of Fourier integral operators should be a pseudo-differential operator.
	
	\begin{Lemma}\cite[Lemma 2.6]{Lucas}\label{2nd_type}
		For all $s\in [0,1],$ there exists $B_s(.)\in S\left(\jp{\rho}^{6N}\right)$ such that:
		\begin{itemize}
			\item[$\bullet$] $\tilde{R}_s^*\tilde{R}_s=\text{Op}_1\left(B_s\right)$.
			\item[$\bullet$] There exists an universal constant $M\in \N$ such that for all $\alpha \in \N^d, s\in[0,1]$,
			\begin{equation*} 
				\sup_{\rho} \left|\jp{\rho}^{-6N}\partial^\alpha B_s(\rho)\right|\leq C \left(\sup_{|\beta|\leq |\alpha|+M}\sup_{\rho} \partial^\beta b_N(\rho)\jp{\rho}^{-3N}\right)^2 . 
			\end{equation*} 
		\end{itemize}
	\end{Lemma}
	Once again, Vacossin proves this lemma in a more restrictive setting; the modifications needed are the same as for lemma \ref{1st_type}.
	\medbreak
	This allows us to make the following computation:
	\begin{align*}
		\left\|\tilde{R}_s\right\|^2_{H_1(\jp{\rho}^{3N})\to L^2}&\leq \left\|\tilde{R}_s^*\tilde{R}_s\right\|_{H_1(\jp{\rho}^{3N})\to H_1(\jp{\rho}^{-3N})} 
		\\&\leq \left\|\text{Op}_1\left(B_s\right)\right\|_{H_1(\jp{\rho}^{3N})\to H_1(\jp{\rho}^{-3N})} 
		\\&\leq C\sup_{|\alpha|\leq M}\sup_\rho \left|\partial^\alpha B_s(\rho)\jp{\rho}^{-6N}\right|
		\\&\leq C \left(\sup_{|\beta|\leq 2M}\sup_\rho \partial^\beta b_N(\rho)\jp{\rho}^{-3N}\right)^2
		\\&\leq (C\|a\|_{C^{N+M}})^2.
	\end{align*}
	Hence, $\|\tilde{R}_n^\text{exp}\|\leq \int_0^1 \|\tilde{R}_s\|\d s\leq C$ and we conclude in the same way as the first remainder term, see \ref{rest 1st}.

	\section{Propagating for longer times, from squeezed states to partially Lagrangian states.}\label{section2nd}
	\textbf{Limitations of the previous method and new description:}
	The problem of the method used in the previous section is that it cannot be used to propagate our coherent states up to Ehrenfest times. Once the wavepacket spreads too much, more precisely, once $\|\kappa\|$ reaches $h^{-1/6}$, our method fails since cubic and higher terms no longer are correction terms, see (\ref{higher order}).
	\medbreak
	In order to propagate further anyway, we wish to represent our state in a more adequate way. Using the properties of our coordinates, we see that if we only look at the flow restricted to the central manifold, its ``new'' Ehrenfest time will be big enough so that, on the time scale we are considering, our method will still work in the central directions, meaning the $x$ and $\xi$ coordinates.
	\medbreak
	However, this will not be true in the transversal directions $y$ and $\eta$. For these, we use a representation as ``WKB states'', see section \ref{WKB description}.
	\medbreak
	Let us define the specific class of functions that we will be using to describe our state:
	\begin{definition}\label{def S_delta,nu}
		Consider $\delta,\nu\geq 0$ and define $S_{\delta,\nu}\left(\R^{d_\perp},\text{Poly}\times\text{Gauss}\right)$ as the set of functions $u$ from $\R^{d_{\scalerel*{\parallel}{\perp}}}\times \R^{d_\perp}$ to $\C$ which can be written as:
		\begin{align*}
			u(x,y)&=\sum_{\substack{\gamma\in \N^{d_{\scalerel*{\parallel}{\perp}}} \\|\gamma|\leq N}} u_\gamma(y) \mathcal{M}_{\text{c},h}\left(\kappa^\text{c}( y)\right)\Bigl[\left(\ \boldsymbol{\cdot}\ /\sqrt{h}\right)^\gamma \varphi_0(\ \boldsymbol{\cdot}\ )\Bigr](x), \text{ for } x\in \R^{d_{\scalerel*{\parallel}{\perp}}}, y\in \R^{d_\perp},
			\\&=\mathcal{M}_\text{c}\left(\kappa^\text{c}(y)\right)\Lambda_{h,x}\Bigl[P_y\Psi_0\Bigr](x),
		\end{align*}
		where for all $\gamma\N^{d_{\scalerel*{\parallel}{\perp}}}$,$u_\gamma\in S_\delta\left(\jp{y}^{-\infty}\right)$, $\kappa^\text{c}\in h^{-\nu}S_{0+\cdot}\left(\R^{d_\perp},\text{Symp}_{d_{\scalerel*{\parallel}{\perp}}}(\C)\right)$ and $\Lambda_{h,x} f(x,y)=h^{-d_{\scalerel*{\parallel}{\perp}}/4}\times$ \newline $ f\left(\frac{x}{\sqrt{h}},y\right)$.
		\medbreak
		We make an additional assumption on the central part of $u$ regarding $\kappa^\text{c}$ or more precisely on the associated covariance matrix $\Gamma_{\scalerel*{\parallel}{\perp}}$ and its derivatives $\frac{\partial^{\alpha}\Gamma_{\scalerel*{\parallel}{\perp}}}{\partial y^{\alpha}}$ (where $\alpha\in \N^{d_\perp}$):
		\begin{equation}\label{GammadGamma}
			\left\|\Im\Gamma_{\scalerel*{\parallel}{\perp}}^{-1/2}\frac{\partial^{\alpha}\Gamma_{\scalerel*{\parallel}{\perp}}}{\partial y^{\alpha}}  \Im\Gamma_{\scalerel*{\parallel}{\perp}}^{-1/2}\right\|\leq C h^{-2\nu}|\log h|^{|\alpha|},
		\end{equation}
		this will be useful to understand the derivatives of $u$ in $L^2$ norms.
	\end{definition}
	\begin{Remark}
		\begin{itemize}
			\item We will work with $\delta<1/2$ so that we can use the stationary phase lemma, see section \ref{WKB description}. Unfortunately,  (non-squeezed) coherent states are not in the set $S_{\delta,\nu}\left(\R^{d_\perp},\text{Poly}\times\text{Gauss}\right)$; when writing them as a product of a Gaussian in $y$ times a Gaussian in $x$, we want to set $u_0(y)$ to be the Gaussian in $y$ but it is in no $S_\delta$ class with $\delta<1/2$. However, as explained in section \ref{WKB description}, the presence of hyperbolicity in these directions tells us that propagating the state will regularize in the direction $y$. This is why we will first propagate with the method of section \ref{section1st}, after propagating for a $\epsilon_s |\log h|$ time, we expect to obtain a state in $S_{\delta,\nu}\left(\R^{d_\perp},\text{Poly}\times\text{Gauss}\right)$. Actually, for reasons that will appear later on, we will even assume that $\delta+\nu<1/2$.
			\item For $\nu,$ we will restrict ourselves to $\nu<1/6$. As we have seen for propagation for small $\log(1/h)$ times, this is the maximal size of $\kappa$ for which section \ref{section1st}'s method works. Thanks to the hypothesis of 3-normal hyperbolicity (\ref{H1'}), we can see that $\nu=1/6-$ will not be reached until times close to the Ehrenfest one. Before that, if we propagate up to a time $t=\frac{\alpha \log(1/h)}{2 \lambda_\text{max}}$, we can take $\nu=\alpha/6$.
		\end{itemize}
	\end{Remark}
	\medbreak
	\textbf{Objective: }Let $u\in S_{\delta,\nu}\left(\R^{d_\perp},\text{Poly}\times\text{Gauss}\right)$ for $\delta<1/2$ and $\nu\leq 1/6$. In this section, we want to show that, after propagating for a time $t$ independent of $h$ (see the setting in the previous section \ref{context}), we can still use this space to describe our state.
	\medbreak
	\textbf{Context of the computation:} We will use the same framework as in the previous section, see section \ref{context}, except for the use of the operator $\hat{T}$: it was previously used to describe the fact that we thought of our state as a localized around a certain point. However, in this new setting, we rather think of our state as localized around the isotropic manifold (the unstable manifold to $\rho_0$ (resp $\rho_1$) when it makes sense). 
	\newline 
	As a consequence, consider the initial manifold $\mathcal{I}_0$ to be a manifold of dimension $d_\perp$  projectable in $y^0$ passing through $\rho_0$ and parameterize it using the method described in the Introduction:
	\begin{equation*} 
		\mathcal{I}_0=\left\{\biggl(\overline{x^0}(y^0), y^0,\overline{\xi^0}(y^0), \underbrace{\nabla \phi^0(y^0)+\frac{(\nabla \overline{\xi^0}(y^0)).\overline{x^0}(y^0)-\overline{\xi^0}(y^0).(\nabla \overline{x^0}(y^0))}{2}}_{\coloneq \overline{\eta^0}(y^0)}\biggr), y^0\in D_{\epsilon_1}(q^0_y)\right\},
	\end{equation*} 
	(with $\epsilon>0$ the maximal diameter of our charts) and assume that its image under the classical dynamics $F$ (defined in (\ref{def F})) is still projectable on $y^0$:
	\begin{align*} 
		\mathcal{I}_1&=F( \mathcal{I}_0)
		\\&=\left\{\biggl(\overline{x^1}(y^1), y^1,\overline{\xi^1}(y),\underbrace{\nabla_y \phi^1(y^1)+\frac{(\nabla \overline{\xi^1}(y^1)).\overline{x^1}(y^1)-\overline{\xi^1}(y^1).(\nabla \overline{x^1}(y^1))}{2}}_{\coloneq \overline{\eta^1}(y^1) }\biggr), y^1\in D_{\epsilon_1}(q^1_y)\right\}.
	\end{align*} 
	We also assume that these manifolds are close to $\mathcal{I}_m$, the functions defining them verify for a $C>0$ independent of $h$:
	\medbreak
	\textbf{Dynamical assumption:}
	\begin{equation}\label{I proche}
		\left\|\overline{x^0}\right\|_{C^1},\left\|\overline{\xi^0}\right\|_{C^1},\left\|\overline{\eta^0}\right\|_{C^1}\leq C
	\end{equation}
	which implies that the same is true for functions that define $\mathcal{I}_1$.
	\medbreak
	Then we define an operator associated with a manifold $\mathcal{I}_j$, (see remark \ref{rem th 2}) $j=0$ or $1$: for $u\in L^2\left(\R^d\right)$
	\begin{equation}\label{T_L}
		\hat{T}_{\mathcal{I}_j}u(x,y)= \exp\left( \frac{i}{h} \Bigl[\phi^j(y)+(\overline{\xi^j}(y)).(x-\overline{x^j}(y)/2)\Bigr]\right) u\left(x-\overline{x^j}(y),y\right),
	\end{equation}
	this operator is unitary, and its inverse is given by, for $v\in L^2\left(\R^d\right)$,
	\begin{equation*} 
		\left(\hat{T}_{\mathcal{I}_j}\right)^*v(x,y)=\exp\left(- \frac{i}{h}\Bigl[\phi^j(y)+\overline{\xi^j}(y).(x+\overline{x^j}(y)/2)\Bigr]\right) u\left(x+\overline{x^j}(y), y\right).
	\end{equation*} 
	\begin{Remark}
		This modification is only needed when considering points that are not exactly on $K^\delta$, as for points in $K^\delta$, this isotropic manifold is exactly described by the variable $y$, i.e. $\mathcal{I}_j=\bigl\{(q^j_x, q^j_y+y^j,p^j_x, p^j_y), y^j\in D_{\epsilon_1}\bigr\}$ for all $j$, no drift can appear in the central direction, and we can still use the standard Weyl-Heisenberg operator $\hat{T}\left(\rho_j\right)$ rather than $\hat{T}_{\mathcal{I}_j}$.
	\end{Remark}
	For $u\in S_{\delta,\nu}$, we want to compute
	\begin{equation*} 
		\left(\hat{T}_{\mathcal{I}_1}\right)^* U^{t_0} \hat{T}_{\mathcal{I}_0}  u,
	\end{equation*} 
	just as in the previous section, we actually want to have a microscopic scaling for the $x$ variable while keeping the $y$ variable macroscopic,
	as a consequence we will instead compute:
	\begin{equation*} 
		u_1\coloneq\Bigl(\Lambda_{h,x}\Bigr)^* \Bigl(\hat{T}_{\mathcal{I}_1}\Bigr)^* U^{t_0}\hat{T}_{\mathcal{I}_0} \Lambda_{h,x} u_0,
	\end{equation*} 
	where $\Lambda_{h,x} f(x,y)=h^{-d_{\scalerel*{\parallel}{\perp}}/4} f\left(h^{-1/2}x,y\right)$ and $u_0=\left(\Lambda_{h,x}\right)^* u$.
	\medbreak
	For the sake of simplicity, we will explain the method in the case where no auxiliary variable $\Theta$ is needed. To deal with the general case, one should first remark that auxiliary variables are only needed for the central directions. Then occurrences of $\Theta$ should be treated those of $x^0$ and $\xi^0$, see the method explained in the corresponding paragraph \ref{Theta higher order}.
	\subsection{\texorpdfstring{Integration in $y^0,\eta^0$}{Integration in y0,eta0}}
	\textbf{Notation:}
	For the sake of readability, we will introduce the notations 
	\begin{equation*} 
		X^1=\left(x^1,y^1\right), \quad X^0=\left(x^0, y^0\right), \quad \Xi^0=\left(\xi^0,\eta^0\right).
	\end{equation*} 
	\medbreak
	We start by using an integral representation of the Fourier Integral Operator $\left(\hat{T}_{\mathcal{I}_1}\right)^* U^{t_0} \hat{T}_{\mathcal{I}_0}$ as
	\begin{equation*} 
		\left(\hat{T}_{\mathcal{I}_1}\right)^* U^{t_0} \hat{T}_{\mathcal{I}_0}u(X^1)=(2\pi h)^{-d/2}  \int_{\R^{2d}} e^{\frac{i}{h} \bigl(\psi(X^1,\Xi^0)-X^0.\Xi^0\bigr)} a(X^1,\Xi^0) u(X^0) \d X^0 \d \Xi^0.
	\end{equation*} 
	\begin{Remark}
		As $a$ has a compact support, so does the final state $\left(\hat{T}_{\mathcal{I}_1}\right)^* U^{t_0} \hat{T}_{\mathcal{I}_0}u$. As a consequence, estimates on its $L^2$ norm can be obtained only by looking at its $L^\infty$ norm.
	\end{Remark}
	This integral representation comes with a generating function $\psi$ that represents the symplectomorphism 
	\begin{equation}\label{mathfrak F}
		\mathfrak{F}\coloneq\left(\kappa_{\mathcal{I}_1}\right)^{-1} \circ \kappa_{\tilde{\rho_1}}\circ \Phi^{t_0}\circ \left(\kappa_{\tilde{\rho_0}}\right)^{-1} \circ\kappa_{\mathcal{I}_0}=\left(\kappa_{\mathcal{I}_1}\right)^{-1}\circ F\circ\kappa_{\mathcal{I}_0}
	\end{equation}
	where 
	\begin{multline*} 
		\kappa_{\mathcal{I}_i}(x,y,\xi,\eta)=
		\\\left(x+\overline{x^i}(y),y,\xi+\overline{\xi^i}(y),\eta+\nabla \phi^i(y)+\frac{(\nabla \overline{\xi^i}(y)).\overline{x^i}(y)-\overline{\xi^i}(y).(\nabla \overline{x^i}(y))}{2}-\nabla \overline{x^i}(y). \xi\right).
	\end{multline*} 
	In particular, this symplectomorphism $\mathfrak{F}$ sends the model isotropic manifold $\mathcal{I}_m=\{(0,y,0,0), $
	\newline$y\in D_{\epsilon_1}\}$ onto itself.
	\medbreak
	Making a change of variable in $x^0,\xi^0$ to rescale the variables in this integral representation yields:
	\begin{align}\label{integrale totale}
		\begin{split}
			u_1\left(X^1\right)&=\bigl(\bigl(\Lambda_{h,x}\bigr)^* \bigl(\hat{T}_{\mathcal{I}_1}\bigr)^* U^{t_0} \hat{T}_{\mathcal{I}_0} \Lambda_{h,x} u_0\bigr)(X^1)
			\\&=\frac{h^{-d_\perp/2}}{(2\pi)^d}\int_{\R^{2d}} e^{\frac{i}{h}\psi_\text{tot}\left(X^1,\Xi^0,X^0\right)} a\left(\sqrt{h}x^1,y^1,\sqrt{h}\xi^0,\eta^0\right)  
			\\&\qquad\qquad\qquad\times\sum_{\substack{\gamma\in \N^{d_{\scalerel*{\parallel}{\perp}}}, \\|\gamma|\leq N} }u_\gamma\left(y^0\right) \mathcal{M}_{\text{c},1}\left(\kappa^\text{c}( y^0)\right)\bigl[X^\gamma \Psi_0\bigr](x^0) \d X^0 \d \Xi^0,
		\end{split}
	\end{align}
	where 
	\begin{equation*}
		\psi_{\text{tot}}\left(X^1,\Xi^0,X^0\right)=\psi\left(\sqrt{h}x^1, y^1,\sqrt{h}\xi^0,\eta^0\right)-h x^0\xi^0
		-y^0\eta^0.
	\end{equation*}
	Let us select a term $\gamma$ in the sum inside the integrand of (\ref{integrale totale}) and compute the associated integral.
	\medbreak
	\subsubsection{Non stationary phase in $y^0,\eta^0$}
	Let us start with a preliminary lemma to estimate derivatives of a $S_\delta^\nu$ function.
	
	\begin{Lemma}\label{derive squeezed}
		Let $\alpha\in \N^{d_\perp}$, then 
		\begin{equation*} 
			\partial_{y^0}^{\alpha} \mathcal{M}_{c,1}\left(\kappa^c(y^0)\right) \bigl[X^\gamma \Psi_0\bigr](x^0)=\left|\det (\Im(\Gamma_{\scalerel*{\parallel}{\perp}}))\right|^{1/4} Q_{\alpha} \left(\Im \Gamma_{\scalerel*{\parallel}{\perp}}^{1/2} x\right)e^{i \frac{x\cdot\Gamma_{\scalerel*{\parallel}{\perp}} x}{2}},
		\end{equation*} 
		with $Q_\alpha$ a polynomial of  degree lower than $|\gamma|+2|\alpha|$ verifying $N_\infty(Q_\alpha)\leq C h^{-2\nu|\alpha|}$.
		\medbreak
		In particular, we have
		\begin{equation*} 
			\left\|\partial_{y^0}^{\alpha} \mathcal{M}_{c,1}\left(\kappa^c(y^0)\right) \bigl[X^\gamma \Psi_0\bigr](x^0)\right\|_{L^2(\d x^0)L^\infty(\d y^0)}\leq C_{\alpha,\gamma} h^{-2|\alpha|\nu} |\log h|^{|\alpha|}.
		\end{equation*} 
	\end{Lemma}
	
	\begin{proof}
		First, let us use \ref{propag excited} to write 
		\begin{equation*} 
			\mathcal{M}_{c,1}\left(\kappa^c(y^0)\right) \bigl[X^\gamma \Psi_0\bigr](x^0)=\left|\det (\Im(\Gamma_{\scalerel*{\parallel}{\perp}}))\right|^{1/4}P\left(\Im \Gamma_{\scalerel*{\parallel}{\perp}}^{1/2} x\right)e^{i \frac{x\cdot\Gamma_{\scalerel*{\parallel}{\perp}} x}{2}},
		\end{equation*} 
		where $P$ is a polynomial of degree at most $|\gamma|$ and $N_\infty(P)\leq C$.
		\medbreak 
		Using this expression, we compute derivatives for each of these terms:
		\begin{itemize}
			\item $\partial_y |\det (\Im(\Gamma_{\scalerel*{\parallel}{\perp}}))|^{1/4}=\partial_y\det (\Im(\Gamma_{\scalerel*{\parallel}{\perp}}))^{1/4}$ since the matrix $\Im(\Gamma_{\scalerel*{\parallel}{\perp}})$ is symmetric definite positive. Then
			\begin{align*}
				|\partial_y \det (\Im(\Gamma_{\scalerel*{\parallel}{\perp}}))^{1/4}|&=\frac{1}{4} (\det \Im \Gamma)^{1/4} |\text{Tr} (\Im \Gamma^{-1} \partial_y \Im \Gamma)|
				\\&\leq C (\det \Im \Gamma)^{1/4} \|\Im \Gamma^{-1/2} \partial_y \Im \Gamma \Im \Gamma^{-1/2}\|.
			\end{align*}
			For higher derivatives, we obtain similar expressions, leading to the following estimate:
			\begin{align*}
				|\partial_y^{\beta} \det (\Im(\Gamma_{\scalerel*{\parallel}{\perp}}))^{1/4}|&\leq C (\det \Im \Gamma)^{1/4} \sup_{\sum k_i\beta_i=\beta} \|\Im \Gamma^{-1/2} \partial_y^{\beta_i} \Im \Gamma \Im \Gamma^{-1/2}\|^{k_i}
				\\&\leq C (\det \Im \Gamma)^{1/4} h^{-2\nu |\beta|} |\log h|^{|\beta|},
			\end{align*}
			with the assumption \ref{GammadGamma}.
			\item For $\partial^{\beta}_ye^{i \frac{x\cdot\Gamma_{\scalerel*{\parallel}{\perp}} x}{2}}$, we obtain a sum of terms
			\begin{equation*} 
				\left( x \partial_{y}^{\beta_1} \Gamma_{\scalerel*{\parallel}{\perp}} x\right)^{k_1} \dots \left( x \partial_{y}^{\beta_l} \Gamma_{\scalerel*{\parallel}{\perp}} x\right)^{k_l}e^{i \frac{x\cdot\Gamma_{\scalerel*{\parallel}{\perp}} x}{2}},
			\end{equation*} 
			with $\sum k_i\beta_i=\beta$. Which we rewrite as some 
			\begin{equation*} 
				Q(\Im \Gamma^{1/2}x)e^{i \frac{x\cdot\Gamma_{\scalerel*{\parallel}{\perp}} x}{2}},
			\end{equation*} 
			with $Q$ polynomial of degree $2|\beta|$ and
			\begin{align*}
				N_\infty(Q)&\leq C \sup_{\sum k_i\beta_i=\beta}\left\| \Im\Gamma_{\scalerel*{\parallel}{\perp}}^{-1/2}\partial^\beta_i \Gamma_{\scalerel*{\parallel}{\perp}} \Im \Gamma_{\scalerel*{\parallel}{\perp}}^{-1/2}\right\|^{k_i}
				\\&\leq C h^{-2\nu |\beta|} |\log h|^{|\beta|},
			\end{align*}
			with the assumption \ref{GammadGamma}.
			\item For $\partial^{\beta}_y P\left(\Im (\Gamma_{\scalerel*{\parallel}{\perp}})^{1/2} x\right)$, we compute that 
			\begin{align*}
				\partial_y P\left(\Im \Gamma_{\scalerel*{\parallel}{\perp}}^{1/2} x\right)&=\langle \partial_y \Im \Gamma_{\scalerel*{\parallel}{\perp}}^{1/2}, \nabla P\left(\Im \Gamma_{\scalerel*{\parallel}{\perp}}^{1/2} x\right)
				\\&=\tilde{P}(\Im \Gamma_{\scalerel*{\parallel}{\perp}} ^{1/2}),
			\end{align*}
			with $\tilde{P}$ a polynomial whose coefficients are bounded by $C h^{-2\nu}N_\infty(P)$.
			\medbreak
			In a similar fashion, we can write higher order derivatives as 
			\begin{equation*} 
				\partial^{\beta}_y P\left(\Im (\Gamma_{\scalerel*{\parallel}{\perp}})^{1/2} x\right)=P_{\beta}\left(\Im (\Gamma_{\scalerel*{\parallel}{\perp}})^{1/2} x\right),
			\end{equation*} 
			with $N_\infty(P_\beta)\leq C h^{-2\nu |\beta|} N_\infty(P)$.
		\end{itemize}
		As a consequence, we have the following:
		\begin{equation*} 
			\partial_{y^0}^{\alpha} \mathcal{M}_{c,1}\left(\kappa^c(y^0)\right) \bigl[X^\gamma \Psi_0\bigr](x^0)=\left|\det (\Im(\Gamma_{\scalerel*{\parallel}{\perp}}))\right|^{1/4} Q_{\alpha} \left(\Im (\Gamma_{\scalerel*{\parallel}{\perp}})^{1/2} x\right)e^{i \frac{x\cdot\Gamma_{\scalerel*{\parallel}{\perp}} x}{2}},
		\end{equation*} 
		with $Q_\alpha$ a polynomial of degree lower than $\deg (P)+2|\alpha|$ verifying $N_\infty(Q_\alpha)\leq C h^{-2\nu|\beta|} N_\infty(P)$.
		\medbreak
		Finally, the computation of its $L^2$ norm yields:
		\begin{equation*} 
			\left\|\partial_{y^0}^{\alpha} \mathcal{M}_{c,1}\left(\kappa^c(y^0)\right) \bigl[X^\gamma \Psi_0\bigr](x^0)\right\|_{L^2(\d x^0)}\leq C h^{-2|\alpha_4|\nu} |\log h|^{|\alpha_4|}.
		\end{equation*} 
	\end{proof}

	We want to start by performing the integration in $y^0,\eta^0$. In this section, we will consider $x^0,\xi^0$ as external parameters.
	\medbreak
	First, we should obtain a non stationary phase lemma:
	\begin{Lemma}\label{non statio}
		If for all $X^1\in D_{\epsilon_1}$, the phase $\partial_{y^0,\eta^0}\psi$ is uniformly bounded from below by some $C>0$ on $\text{supp }a$ then 
		\begin{equation*} 
			\left\|I\right\|_{L^2(\d x^1 \d y^1)} =O(h^\infty).
		\end{equation*} 
	\end{Lemma}
	\begin{proof}
		We follow the usual proof of the non-stationary phase lemma, except that we are looking for $L^2(\d X^1)$ estimates. 
		\medbreak
		We define the usual operator 
		\begin{equation*} 
			L=\frac{h}{i}\frac{1}{|\partial_{y^0,\eta^0} \psi_\text{tot}|^2}\left\langle \partial_{y^0,\eta^0} \psi_\text{tot},\partial_{y^0,\eta^0}\right\rangle,
		\end{equation*} 
		thanks to the non-vanishing hypothesis of $d\psi_{\text{tot}}$, this operator makes sense on $\supp a$.
		\newline It is designed to verify: 
		\begin{equation*} 
			L(e^{i\psi_\text{tot}/h})=e^{i\psi_\text{tot}/h}.
		\end{equation*} 
		Injecting this identity into $I$ yields: $\forall X^1\in D_{\epsilon_1}$, $\forall K\in \N$,
		\begin{equation*} 
			I(X^1)=(2\pi h)^{-d_\perp/2} \int_{\R^{2d}} e^{\frac{i}{h}\psi_\text{tot}} (L^*)^K \bigl[a\ u_\gamma\ v\bigr]\d X^0 \d \Xi^0,
		\end{equation*} 
		with 
		\begin{equation}\label{def v}
			v\left(x^0,y^0\right)= \mathcal{M}_{c,1} \left(\kappa^c(y^0)\right) \bigl[X^\gamma\Psi_0\bigr]\left(x^0\right).
		\end{equation}
		\medbreak
		Computing $(L^*)^K$ gives a differential operator of order $K$ with coefficients involving derivatives of $\psi_\text{tot}$ at the point $\left(X^1,\Xi^0,X^0\right)$ and a global factor $h^K$.
		\medbreak
		Once we remove this global factor $h^K$, we are left with the estimation of a sum of integrals of type
		\begin{equation} \label{FIO pseudo}
			I_b=(2\pi h)^{-d_\perp/2} \int_{\R^{2d}} e^{\frac{i}{h}\psi_\text{tot}} b\ \partial^{\alpha}_y \bigl[u_\gamma\ v\bigr]\d X^0 \d \Xi^0,
		\end{equation}
		with $b\in C^\infty_c$ composed of derivatives of $\psi$ and $a$.
		\medbreak
		Using composition of Fourier Integral Operators (or similar arguments as in lemma \ref{1st_type}), we can write this integral as the composition of the Fourier Integral Operator $\Bigl(\hat{T}_{\mathcal{I}_1}\Bigr)^* U^{t_0}\hat{T}_{\mathcal{I}_0}$ and a pseudodifferrential operator $\Op (\tilde{b})$ as follows:
		\begin{equation*} 
			I_b=\Bigl(\Lambda_{h,x}\Bigr)^* \Bigl(\hat{T}_{\mathcal{I}_1}\Bigr)^* U^{t_0}\hat{T}_{\mathcal{I}_0} \Op (\tilde{b}) \partial^{\alpha}_y \bigl[u_\gamma\ v\bigr],
		\end{equation*} 
		with a $\tilde{b}\in \mathcal{S}$.
		\medbreak
		Then, using the unitarity of the operators $\Lambda_{h,x}$, $\Bigl(\hat{T}_{\mathcal{I}_1}\Bigr)^* U^{t_0}\hat{T}_{\mathcal{I}_0}$ and $L^2$-continuity of $\Op (\tilde{b})$ we obtain:
		\begin{equation*} 
			\|I_b\|_{L^2}\leq C \|\partial^{\alpha}_y \bigl[u_\gamma\ v\bigr]\|_{L^2}.
		\end{equation*} 
		As $u_\gamma\in S^{\delta}(\jp{y}^{-\infty})$, we have
		\begin{align}\label{fin est L2} 
			\begin{split}
				\|\partial^{\alpha}_y \bigl[u_\gamma\ v\bigr]\|_{L^2(d X^1)}&\leq \sup_{\alpha_1+\alpha_2=\alpha} \|\partial^{\alpha}_y u_\gamma\|_{L^2(\d y)} \|\partial^{\alpha_2} v\|_{L^2(\d x)L^\infty(\d y)}
				\\&\leq C\sup_{\alpha_1+\alpha_2=\alpha} \|\partial^{\alpha_1}_y u_\gamma\|_{L^\infty(\d y)} \|\partial^{\alpha_2} v\|_{L^2(\d x)L^\infty(\d y)}
				\\&\leq C\sup_{\alpha_1+\alpha_2=\alpha} h^{-\delta |\alpha_1|} h^{-2\nu |\alpha_2|}.
			\end{split}
		\end{align}
		Hence, 
		\begin{equation*}
			\|I(X^1)\|_{L^2(\d X^1)} \leq C h^{(1-\max(\delta,2\nu))K},
		\end{equation*} 
		which concludes the proof with our assumptions on $\delta$ and $\nu$.
	\end{proof}
	
	\subsubsection{Stationary points}
	The next step is the identification of stationary points. We consider $x^1,y^1,x^0,\xi^0$ to be external parameters, so that we obtain two equations:
	\begin{align}\label{pt_crit}
		\begin{split}
			y^0&=\partial_{\eta} \psi\Bigl(\sqrt{h}x^1, y^1,\sqrt{h}\xi^0,\eta^0\Bigr)
			\\\eta^0&=0.
		\end{split} 
	\end{align}
	Note that this system is equivalent to finding $\tilde{x^0},\tilde{y}^0,\tilde{\xi^1},\tilde{\eta^1}$ such that
	\begin{equation*} 
		\mathfrak{F}\left(\tilde{x}^0,\tilde{y}^0,\sqrt{h}\xi^0, 0 \right)=\left(\sqrt{h}x^1,y^1,\tilde{\xi}^1,\tilde{\eta}^1\right).
	\end{equation*} 
	But when $x^1=0$ and $\xi^0=0$, we already had an explicit solution to this equation: we can choose on the right hand-side a point of the isotropic manifold $\mathcal{I}_m$ and as $\mathfrak{F}(\mathcal{I}_m)=\mathcal{I}_m$, we know that there exists $y^0$ such that 
	\begin{equation}\label{I_pt_crit}
		\mathfrak{F}\left(0,y^0,0, 0 \right)=\left(0,y^1,0,0\right).
	\end{equation}
	\subsubsection{Morse lemma with parameter}
	The next step of the proof is to obtain a change of variable close to the critical set that will allow us to reduce our study to the quadratic case. This part can be done in the same way as classical proofs, which can be found in \cite{FIO}. 
	\medbreak
	
	\begin{Lemma}\cite[Lemma 1.2.2]{FIO}\label{Morse}
		There exists a $C^\infty$ mapping $\kappa:\R^{3d} \to\R^{2d_\perp}$ such that
		\begin{equation*} 
			\kappa\left(x^1,y^1,\xi^0,x^0, y^0,\eta^0\right)=(z,\zeta),
		\end{equation*} 
		with 
		\begin{equation*} 
			z=y^0-\tilde{y}^0+O\left(|y^0-\tilde{y}^0|^2\right),\quad \zeta=\eta^0+O\left(|\eta^0|^2\right),
		\end{equation*} 
		where $\tilde{y}^0\left(x^1,y^1,\xi^0\right)$ and $\tilde{\eta}^0\left(x^1,y^1,\xi^0\right)=0$ are the critical points
		and
		\begin{equation*} 
			\psi_\text{tot}\left(X^1,\Xi^0,X^0\right) = \psi_\text{tot}\left(X^1,\xi^0, \tilde{\eta}^0,x^0,\tilde{y}^0\right) + \frac{1}{2}\left\langle Q(X^1,\xi^0) (z,\zeta),(z,\zeta)\right\rangle,
		\end{equation*}  
		\begin{equation}\label{Q_phase_statio}
			Q(X^1,\xi^0) = \begin{pmatrix}
				0 & - I
				\\ -I & \partial^2_{\eta,\eta} \psi(\sqrt{h}x^1,y^1,\sqrt{h}\xi^0,0).
			\end{pmatrix}
		\end{equation}
		\medbreak
		Moreover, at fixed $X^1$, $\xi^0$, $x^0$; $\tilde{\kappa}:y^0,\eta^0\to (z,\zeta)$ gives a diffeomorphism between $y^0,\eta^0\in \R^{2d_\perp}$ and $z,\zeta\in \R^{2d_\perp}$.
	\end{Lemma}
	\begin{Remark}\label{det Q}
		The matrix $Q\left(X^1,\xi^0\right)$ is the Hessian matrix in $y,\eta$ of $\psi_{\text{tot}}$ at the critical point, it is non-degenerate. In particular, we have $|\det Q(X^1,\xi^0)|=1$ and $\text{sgn}\ Q(X^1,\xi^0)=0$.
	\end{Remark}
	\medbreak
	\subsubsection{Stationary phase lemma}
	Here the proof is similar to the usual one (see, for instance, \cite[Theorem 3.11]{Zwbook}), the only modification comes from the estimation of rests which is done just as in the non stationary phase setting, lemma \ref{non statio}.
	\medbreak
	Let us start by stating the result for a fixed $x^1, y^1, \xi^0$:
	
	\begin{Proposition}
		Consider $\tilde{y}^0\left(\sqrt{h}x^1,\sqrt{h}\xi^0\right),\tilde{\eta}^0\left(\sqrt{h}x^1,\sqrt{h}\xi^0\right)=0$ associated critical points defined in equation (\ref{pt_crit}), then:
		\newline There exists, for each $l \in\N$, a differential operator $A_{2l}(y,\eta,D_y,D_{\eta})$ in the variables $y^0$ and $\eta^0$, with coefficients that depend smoothly on $X^1,x^0,\xi^0$ and of
		order less than or equal to $2l$, such that for all $K\in \N,$ we have 
		\begin{multline*}
			(2\pi h)^{-d_\perp/2} \int_{\R^{2d_\perp}} e^{\frac{i}{h}\psi_\text{tot}\left(X^1,\Xi^0,X^0\right)} a\left(\sqrt{h}x^1,y^1,\sqrt{h}\xi^0,\eta^0\right) u_\gamma\left(y^0\right) \mathcal{M}_{\text{c},1}\left(\kappa^\text{c}( y^0)\right)\bigl[X^\gamma \Psi_0\bigr](x^0) 
			\\\shoveright {\phantom{}\d y^0 \d \eta^0}
			\\=\sum_{l=0}^K h^l \Bigl[A_{2l}(y,\eta,D_y,D_{\eta}) b\Bigr]\left(X^1, \xi^0, \tilde{\eta}^0,x^0,\tilde{y}^0\right)+R(X^1,x^0,\xi^0),
		\end{multline*}
		where
		\begin{equation*} 
			b\left(X^1,\Xi^0,X^0\right)=a\left(\sqrt{h}x^1, y^1,\sqrt{h}\xi^0,\eta^0\right) \sum_{\substack{\gamma\in \N^{d-d_\text{hyp}} \\|\gamma|\leq N} }u_\gamma(y^0) \mathcal{M}_{\text{c},1}(\kappa^\text{c}( y^0))\bigl[X^\gamma \Psi_0\bigr](x^0),
		\end{equation*} 
		and 
		\begin{equation*} 
			\int_{\R^d} \left( \int_{\R^{2d_{\scalerel*{\parallel}{\perp}}}} R(X^1,x^0,\xi^0)\d x^0 \d \xi^0 \right)^2 \d X^1 \leq C h^{(1-2\max(\delta, 2 \nu))2K}.
		\end{equation*} 
		Moreover, we have $A_0 =\left|\det Q(X^1,\xi^0)\right|^{-1/2}=1$.
	\end{Proposition}
	\begin{proof}
		\begin{itemize}
			Let 
			\begin{align*}
				I_\gamma(X^1,\xi^0,x^0)\coloneq(2\pi h)^{-d_\perp/2}& \int_{\R^{2d_\perp}} e^{\frac{i}{h}\psi_\text{tot}\left(X^1,\Xi^0,X^0\right)} a\left(\sqrt{h}x^1,y^1,\sqrt{h}\xi^0,\eta^0\right) u_\gamma\left(y^0\right) 
				\\&\times\mathcal{M}_{\text{c},1}\left(\kappa^\text{c}( y^0)\right)\bigl[X^\gamma \Psi_0\bigr](x^0) \d y^0 \d \eta^0.
			\end{align*}
			First, we want to reduce the study to the quadratic phase case: using the lemma \ref{Morse}, we can make the change of variables $y^0 \leftarrow z, \eta^0 \leftarrow \zeta$ and obtain the following
			\begin{equation*} 
				I_\gamma(X^1,\xi^0,x^0)\coloneq(2\pi h)^{-d_\perp/2} \int_{\R^{2d_\perp}}e^{\frac{i}{2h} Q(X^1,\xi^0)(z,\zeta)^{\otimes 2}}b(X^1,\sqrt{h}\xi^0,x^0,(\tilde{\kappa})^{-1}(z,\zeta))\d z \d \zeta,
			\end{equation*} 
			hence we are now working with an amplitude
			\begin{equation*} 
				\tilde{b}(X^1,\sqrt{h}\xi^0,x^0,z,\zeta)=a\left(\sqrt{h}x^1, y^1,\sqrt{h}\xi^0,\eta^0(z,\zeta)\right) u_\gamma(y^0(z,\zeta)) \mathcal{M}_{\text{c},1}(\kappa^\text{c}( y^0(z,\zeta)))\bigl[X^\gamma \Psi_0\bigr](x^0).
			\end{equation*} 
			Now, to deal with the quadratic case, we use the usual method: 
			\medbreak
			We begin by using the Fourier transform to get that (using remark \ref{det Q})
			\begin{equation*} 
				I_\gamma(X^1,\xi^0,x^0)= \int_{\R^{2d}}e^{-\frac{ih}{2}\langle Q^{-1}(X^1,\xi^0)(\check{z},\check{\zeta}),(\check{z},\check{\zeta})\rangle} \check{\tilde{b}} (X^1,\sqrt{h}\xi^0,x^0,\check{z},\check{\zeta}) \d \check{z}\d\check{\zeta}.
			\end{equation*} 
			Set 
			\begin{equation*} 
				J_{\tilde{b}}(h,X^1,\xi^0,x^0)=\int_{\R^{2d}}e^{-\frac{ih}{2}\langle Q^{-1}(X^1,\xi^0)(\check{z},\check{\zeta}),(\check{z},\check{\zeta})\rangle} \check{\tilde{b}} (X^1,\sqrt{h}\xi^0,x^0,\check{z},\check{\zeta}) \d \check{z}\d\check{\zeta},
			\end{equation*} 
			then 
			\begin{align*}
				\partial_hJ_{\tilde{b}}(h,X^1,\xi^0,x^0)&=\int_{\R^{2d_\perp}}e^{-\frac{ih}{2}\langle Q^{-1}(\check{z},\check{\zeta}),(\check{z},\check{\zeta})\rangle}\left(-\frac{i}{2}\langle Q^{-1}(\check{z},\check{\zeta}),(\check{z},\check{\zeta})\rangle \check{\tilde{b}} (X^1,\sqrt{h}\xi^0,x^0,\check{z},\check{\zeta}) \right) \d \check{z}\d\check{\zeta}
				\\&=J_{P\tilde{b}}(h,X^1,\xi^0,x^0),
			\end{align*}
			for $P=-\frac{i}{2}\langle Q^{-1}(X^1,\xi^0) D_{y,\eta},D_{y,\eta}\rangle. $
			Therefore, the Taylor expansion gives:
			\begin{equation*} 
				J_{\tilde{b}}(h,X^1,\xi^0,x^0)=\sum_{l=0}^{K-1}\frac{h^l}{l!}J_{P^l\tilde{b}}(0,X^1,\xi^0,x^0)+ \frac{h^K}{(K-1)!} R_K(h,\tilde{b}),
			\end{equation*} 
			for 
			\begin{equation*} 
				R_K(h,\tilde{b})=\int_0^1 (1-t)^{K-1}J_{P^K\tilde{b}}(th,X^1,\xi^0,x^0) \d t.
			\end{equation*} 
			We then compute that 
			\begin{equation*} 
				J_{P^l\tilde{b}}(0,X^1,\xi^0,x^0)=(2\pi)^{d_\perp} P^l\tilde{b}(0),
			\end{equation*} 
			hence, the Taylor expansion written above gives the expansion of the proposition.
			\medbreak
			Finally, we are left with the estimation of the remainder term $R_K(h,\tilde{b})$, by Jensen inequality, we get the following:
			\begin{align*}
				\int_{\R^{2d}} \biggl( \int_{\R^{2d_{\scalerel*{\parallel}{\perp}}}} &R_K(X^1,x^0,\xi^0)\d x^0 \d \xi^0 \biggr)^2 \d X^1
				\\&\leq \int_{\R^{2d}} \int_0^1 (1-t)^{2(k-1)} \left( \int_{\R^{2d_{\scalerel*{\parallel}{\perp}}}} J_{P^K\tilde{b}}(th,X^1,\xi^0,x^0)\d x^0 \d \xi^0\right)^2 \d t \d X^1.
			\end{align*}
			But we notice that 
			\begin{equation*} 
				\int_{\R^{2d_{\scalerel*{\parallel}{\perp}}}} J_{P^K\tilde{b}}(th,X^1,\xi^0,x^0)\d x^0 \d \xi^0= (2\pi h)^{-d/2} \int_{\R^{2d}} e^{ \frac{i}{th}\psi_{\text{tot}}} P^K b \d X^0 \d \Xi^0,
			\end{equation*} 
			where $b$ is defined in the proposition.
			\medbreak
			Notice that we are in the same situation as in the proof of the non-stationary phase lemma, see (\ref{FIO pseudo}). We follow the same strategy and write this integral as the composition of a Fourier Integral Operator and a pseudodifferential one and conclude by estimating the $L^2$ norm just as in (\ref{fin est L2}) (recalling that $P^K$ is of order $2K$), we obtain 
			\begin{equation*} 
				\int \left(\int h^K|R_K(h,\tilde{b})|\d x^0 \d \xi^0\right)^2 \d X^1\leq Ch^{\bigl(1-2\max(\delta,2\nu)-\bigr)2K},
			\end{equation*} 
			which gives us the announced control on the remainder.
		\end{itemize}
	\end{proof}
	
	\begin{Corollary}
		Consider $\tilde{y}^0\left(\sqrt{h}x^1,\sqrt{h}\xi^0\right),\tilde{\eta}^0\left(\sqrt{h}x^1,\sqrt{h}\xi^0\right)=0$ the associated critical points defined in equation (\ref{pt_crit}),
		then
		\begin{enumerate}
			\item There exist for each $l \in\N$, differential operators $A_{2l}(x,D)$ in the variables $y^0$ and $\eta^0$, of
			order less than or equal to $2l$, such that for all $K\in \N,$
			\begin{equation*} 
				u_1=u_1^K+O_{L^2(\d X^1)}(h^{(1-2\max(\delta, 2 \nu))K}),
			\end{equation*} 
			where 
			\begin{equation}\label{après phase statio}
				u_1^K=h^{-d_\perp}\int\int e^{\frac{i}{h}\psi_{\text{tot}}(X^1, \xi^0,\tilde{\eta}^0,x^0,\tilde{y}^0)} \sum_{l=0}^K h^l \Bigl[A_{2l}(x,D) b\Bigr]\left(X^1, \xi^0, \tilde{\eta}^0,x^0,\tilde{y}^0\right)\d x^0 \d \xi^0,
			\end{equation}
			with
			\begin{equation*} 
				b\left(X^1,\Xi^0,X^0\right)=a\left(\sqrt{h}x^1, y^1,\sqrt{h}\xi^0,\eta^0\right) \sum_{\substack{\gamma\in \N^{d-d_\text{hyp}} \\|\gamma|\leq N} }u_\gamma(y^0) \mathcal{M}_{\text{c},1}(\kappa^\text{c}( y^0))\bigl[X^\gamma \Psi_0\bigr](x^0).
			\end{equation*} 
			\item  In particular,
			\begin{equation*} 
				A_0=1.
			\end{equation*} 
		\end{enumerate} 
	\end{Corollary}
	\subsection{\texorpdfstring{Integration in $x^0,\xi^0$}{Integration in x0,xi0}}\label{DL dir centrale}
	To conclude, we still have to integrate in the $x^0$ and $\xi^0$ variables, for this section $y^1$ is considered as an external parameter, when $y^1$ is not mentioned in an estimate, it means that said estimate is uniform in $y^1$. Let us consider the $l$-th term of the expansion, select a term in the sums that make up the integrand of \ref{après phase statio}, if we write
	\begin{equation*} 
		A_{2l}=\sum_{|\alpha|\leq 2l} A_{\alpha}(X^1,\xi^0) \partial^\alpha_{y,\eta},
	\end{equation*} 
	then a term of the sum is given by some
	\begin{equation*} 
		A_{\alpha} \partial^\alpha_{y,\eta}[au_\gamma \mathcal{M}_{\text{c},1}(\kappa^\text{c})\bigl[X^\gamma \Psi_0\bigr]]=A_{\alpha} \sum_{|\alpha_1|+|\alpha_2|+|\alpha_3|=|\alpha|}c_{\alpha_1,\alpha_2,\alpha_3}\ \partial^{\alpha_1}_{y,\eta} a\ \partial^{\alpha_2}_{y} u_\gamma\ \partial^{\alpha_3}_{y} \mathcal{M}_{\text{c},1}(\kappa^\text{c})\bigl[X^\gamma \Psi_0\bigr].
	\end{equation*} 
	Consequently, let us study
	\begin{multline*} 
		h^{2l}I_{\alpha,\gamma}(X^1)=h^{2l}\int\int e^{\frac{i}{h} \psi_{\text{tot}}(X^1,\xi^0,x^0)} b_{\alpha,\gamma}\left(\tilde{y}^0(X^1,\xi^0)\right) \partial^{\alpha_3}_{y} \mathcal{M}_{c,1}\left(\kappa^c(\tilde{y}^0(X^1,\xi^0))\right)\bigl[X^{\gamma'} \Psi_0\bigr](x^0) 
		\\ \shoveright \phantom{}\d x^0 \d\xi^0,
	\end{multline*} 
	with $b_{\alpha,\gamma}$ composed of $A_\alpha\ c_{\alpha_1,\alpha_2,\alpha_3}\ \partial^{\alpha_1}_{y,\eta} a\ \partial^{\alpha_2}_{y} u_\gamma$ and $|\alpha_1|+|\alpha_2|+|\alpha_3|\leq 2l$.
	\subsubsection{Expansion in $x^1,\xi^0$}
	We have to expand every term depending on $x^1$ or $\xi^0$, including terms depending on the stationary point $\tilde{y}^0$.
	\medbreak
	Let us start by doing the Taylor expansion of $\tilde{y}^0$:
	\begin{equation*}
		\tilde{y}^0\left(\sqrt{h}x^1,y^1,\sqrt{h}\xi^0\right)=\tilde{y}^0\left(0,y^1,0\right)+\sqrt{h}\Bigl[\partial_{x^1}\tilde{y}^0\left(0,y^1,0\right).x^1+\partial_{\xi^0}\tilde{y}^0\left(0,y^1,0\right).\xi^0\Bigr]+\text{l.o.t.},
	\end{equation*}
	one can show that $\partial_{x^1}\tilde{y}^0\left(0,y^1,0\right),\partial_{\xi^0}\tilde{y}^0\left(0,y^1,0\right)$ have $L^\infty(\d y^1)$ norms uniformly bounded by a constant $C$. This is done by differentiating equation (\ref{pt_crit}) with respect to $x^1$ (resp. $\xi^0$).
	The same can be done with higher derivatives of $\tilde{y}^0$.
	\medbreak 
	Due to the special role played by the values of $\tilde{y}^0$ at $x^1=\xi^0=0$, we will denote them by $\overline{y^0}(y^1)$, in adequation with the equation \ref{I_pt_crit}.
	\medbreak
	\textbf{Expansion of the phase:} Let us now compute the expansion of the phase $\psi_{\text{tot}}$ in $\sqrt{h}x^1,\sqrt{h}\xi^0$ (not writing dependencies in $y^1$)
	\begin{itemize}
		\item the $0$-order term in $\sqrt{h}x^1,\sqrt{h}\xi^0$ is given by 
		\begin{equation*} 
			\psi_0= \psi\left(0, y^1,0,0\right),
		\end{equation*} 
		which actually is independent of $y^1$ as we have 
		\begin{equation*}
			\partial_{y^1} \psi_0\left(y^1\right)=\partial_y \psi (0,y^1,0,0)=0.
		\end{equation*}
		Hence, the associated phase $e^{\frac{i}{h}\psi_0\left(y^1\right)}$ is in fact a constant in $x^1,y^1$: $e^{i\theta /h}$ which we can ignore.
		\item For the $1$-st order term in $\sqrt{h}x^1,\sqrt{h}\xi^0$ we can see that all the terms cancel out:
		\begin{align*}
			\psi_1\left(x^1,\xi^0\right)&=\partial_x \psi \left(0, y^1,0,0\right).\big[\sqrt{h}x^1]  +\partial_\xi \psi\left(0, y^1,0,0\right).\big[\sqrt{h}\xi^0]-\overline{\eta^0}\cdot \partial y^0 .[\sqrt{h}x^1,\sqrt{h}\xi^0]
			&=0.
		\end{align*}
		\item For the $2$-nd order term in $\sqrt{h}x^1,\sqrt{h}\xi^0$, there are some cancelations that leave us with 
		\begin{equation*}
			\psi_2\left(x^1,x^0,\xi^0\right)=h\partial^2_{x} \psi .\bigl(x^1,x^1\bigr)+2h\partial^2_{x,\xi}\psi.\bigl(x^1,\xi^0\bigr)+h\partial^2_{\xi} \psi.(\xi^0,\xi^0)-hx^0\cdot\xi^0,
		\end{equation*}
		where the differentials of $\psi$ are computed at the point $(0,y^1,0,0)$.
		\item Just as in section \ref{section1st}, higher order terms will be expanded with the exponential (see equation (\ref{DL exp})).
	\end{itemize}
	\textbf{Expansion of the amplitudes:} Then, we expand the cubic and higher terms of the phase along with the amplitudes, this computation is the same as in (\ref{def remainders}):
	\begin{align}\label{DL_exp_a}
		\begin{split}
			e&^{\frac{i}{h}\psi_\text{tot}(X^1,\Xi^0,X^0)} \Bigl[A_\alpha\partial^{\alpha_1}a\Bigr]\left(\sqrt{h}x^1, y^1,\sqrt{h}\xi^0,\right)= \exp\left(i\psi_2(y^1)\right)
			\Biggl[  \sum_{k=0}^{M-1}h^{k/2} \tilde{P}_k\left(x^1,\xi^0\right)
			\\&
			+\underbrace{h^{M/2} \sum_{k=0}^{M-1} P_k(x^1,\xi^0) r^a_{M-k}\left(x^1,\xi^0\right)
				+h^{M/2} R_{M}^{\text{corr}}a\left(\sqrt{h}x^1,\sqrt{h}\xi^0\right)}_{r^{\text{type}1}_M}+\underbrace{r^{\text{exp}}_Ma\left(\sqrt{h}x^1,\sqrt{h}\xi^0\right) }_{r^{\text{type}2}_M}    \Biggr],
		\end{split}
	\end{align}
	where \begin{itemize}
		\item $\psi_2$ is computed above.
		\item  $P_k\left(x^1,\xi^0\right)$ is a polynomial of order $3k$ in $\left(x^1,\xi^0\right)$ with coefficients involving derivatives of $\psi$ taken at the point $\left(0,y^1,0,0\right)$ up to the $3+k$th order.
		\item$\tilde{P}_k$ is a polynomial of order $3k$ in $(x^1,\xi^0)$ given by
		\begin{equation*} 
			\tilde{P_k}=\sum_{k_1+k_2=k} P_{k_1}\left(x^1,\xi^0\right)\times \left( \frac{1}{k_2!} D^{k_2}_{x,\xi}\tilde{a}\left(0,y^1,0\right) .\bigl[x^1,\xi^0\bigr]^{\otimes k_2}\right),
		\end{equation*} 
		with 
		\begin{equation*} 
			\tilde{a}\left(x^1,y^1,\xi^0\right)=\Bigl[A_{\alpha}\partial^{\alpha_1}a\Bigr]\left(x^1, y^1,\xi^0,\tilde{y}^0(x^1,y^1,\xi^0)\right).
		\end{equation*} 
		\item $r_p^a$ being the integral of the Taylor expansion of $\tilde{a}$ up to order $p-1$:
		\begin{equation*} 
			r_p^a=\frac{1}{p!}\int_0^1 (1-s)^{p-1}D^{p} \tilde{a}(s\sqrt{h}x^1,y^1,s\sqrt{h}\xi^0).\bigl[x^1,\xi^0\bigr]^{\otimes p}\d s.
		\end{equation*} 
		\item $R_n^{\text{corr}}$ is the rest coming from corrections of order $N$ and more for the phase verifying 
		\begin{equation*} 
			\left|\partial^\alpha R^\text{corr}_M\left(x^1,\xi^0\right)\right|\leq C \|a\|_{C^{M+|\alpha|}} \jp{ (x^1,\xi^0)}^{3M}.
		\end{equation*} 
		Note that this control can be performed independently of $y^1$ and $\eta^0$ as $a$ and $\psi$ are uniformly bounded.
		\item $r_M^\text{exp}$ is the error that results from the expansion of the exponential 
		\begin{equation*} 
			\frac{i^M h^{M/2}}{(M-1)!} c_3^\psi\left(X^1,\Xi^0\right)^M \int_0^1 e^{i\sqrt{h}sc_3^\psi\left(X^1,\Xi^0\right)}(1-s)^{M-1}\d s,
		\end{equation*} 
		where $\sqrt{h}c_3^\psi=\psi_{\text{tot}}/h-\psi_0/h-\psi_1/\sqrt{h}-\psi_2$.
	\end{itemize}
	\textbf{Expansion of the $S_{\delta,\nu}$ function:} Finally, we are left with the expansion of 
	\begin{equation*} 
		v_\gamma\left(x^0,y^0\right)=\partial^{\alpha_2}_y u_\gamma\left(\tilde{y}^0\right) \partial^{\alpha_3}_y\mathcal{M}_{c,1}\left(\kappa^c(\tilde{y}^0)\right)\bigl[X^{\gamma'} \Psi_0\bigr](x^0),
	\end{equation*} 
	(which depends on $x^1,\xi^0$ through $\tilde{y}^0(x^1,\xi^0)$).
	\newline 
	Let us first recall that, by lemma \ref{derive squeezed}, we have
	\begin{equation*} 
		\partial^{\alpha_3}_y\mathcal{M}_{c,1}\left(\kappa^c(\tilde{y}^0)\right)\bigl[X^{\gamma'} \Psi_0\bigr](x^0)=\left|\det (\Im(\Gamma_{\scalerel*{\parallel}{\perp}}))\right|^{1/4} Q_{\alpha} \left(\Im (\Gamma_{\scalerel*{\parallel}{\perp}})^{1/2} x^0\right)e^{i \frac{x^0\cdot\Gamma_{\scalerel*{\parallel}{\perp}} x^0}{2}},
	\end{equation*} 
	with $Q_\alpha$ a polynomial of degree lower than $|\gamma|+2|\alpha|$ verifying $N_\infty(Q_\alpha)\leq C h^{-2\nu|\beta|}$.
	\medbreak
	Expanding $v_\gamma$ yields:
	\begin{multline} \label{v_gamma}
		v_\gamma\left(x^0,\tilde{y}^0(\sqrt{h}x^1,y^1,\sqrt{h}\xi^0)\right)=v_\gamma\left(x^0,\overline{y^0}(y^1)\right)+\sqrt{h} \partial_y v_\gamma\left(x^0,\overline{y^0}(y^1)\right).\partial_{x^1,\xi^0} \tilde{y}^0.\bigl(x^1,\xi^0\bigr)
		\\\shoveright \phantom{}+\text{l.o.t.}+r_M^v,
	\end{multline}
	where the remainder of this expansion up to order $M$ in $\sqrt{h} x^1,\xi^0$ is denoted by $r_M^v$ and verifies
	\begin{equation*} 
		r^v_M=\frac{h^{M/2}}{M!} \int_0^1 (1-s)^{M-1} \partial^M_{x^1,\xi^0}v_\gamma(x^0,\tilde{y}^0(s\sqrt{h}x^1,y^1,s\sqrt{h}\xi^0).[x^1,\xi^0]^{\otimes M}\d s.
	\end{equation*} 
	\subsubsection{The central metaplectic operator}
	A term given by the combination of these expansions is some 
	\begin{equation*} 
		\int e^{i \psi_2(x^1,x^0,\xi^0)} Q(x^1,\xi^0)\mathcal{M}_{c,1}\left(\kappa^c(\tilde{y}^0)\right)\bigl[X^{\gamma'} \Psi_0\bigr](x^0) \d \xi^1 \d x^0,
	\end{equation*} 
	with $Q$ a polynomial in $x^1,\xi^0$ (depending on $y^1$). The phase inside this integral being quadratic, we can interpret this integral as the action of some metaplectic $\mathcal{M}_{c,1}(d^c\mathfrak{F}(y^1))$ (quantizing a linear symplectic map $d^c\mathfrak{F}(y^1)$) on a squeezed state.
	\medbreak
	Using the formula that relates the quadratic phase of a metaplectic operator and the linear dynamics recalled in (\ref{kappa psi}) we can identify the matrix $d^c\mathfrak{F}(y^1)$ of this operator:
	\begin{Lemma}\label{compare dcF}
		The linear symplectic application generated by the quadratic function
		\begin{equation*} 
			\psi_2\left(y^1,x^1,x^0,\xi^0\right)=\partial^{2}_{x,x}\psi.\bigl(x^1,x^1\bigr)+2\partial^2_{x,\xi} \psi .\bigl(x^1,\xi^0\bigr)+\partial_{\xi^,\xi}^2 \psi.\bigl(\xi^0,\xi^0\bigr)-x^0\cdot \xi^0,
		\end{equation*} 
		(where the derivatives of $\psi$ are computed at the point $(0,y^1,0,0)$) corresponds to the linearization of $\mathfrak{F}$ (defined in (\ref{mathfrak F})) restricted and corestricted to the space $\text{Span }\left(\frac{\partial }{\partial x},\frac{\partial}{\partial \xi}\right)$ i.e. 
		\begin{equation*} 
			d^c\mathfrak{F}(y^1)= \begin{pmatrix}
				d_x\mathfrak{F}_x & d_\xi \mathfrak{F}_x
				\\ d_x \mathfrak{F}_\xi & d_\xi \mathfrak{F}_\xi
			\end{pmatrix}.
		\end{equation*} 
		Moreover, we can compare this matrix with the central part linearization of $F$ at $0$: in the notation of \ref{dF0},
		\begin{equation*} 
			\left\| d^c\mathfrak{F}(y^1)-\begin{pmatrix} A & E
				\\ F & C
			\end{pmatrix}\right\|\leq C\left\|(\overline{x^0}(y^0),y^0,\overline{\xi^0}(y^0),\overline{\eta^0}(y^0))\right\|.
		\end{equation*} 
	\end{Lemma}
	We postpone the proof of this lemma until section \ref{metap_simple}.
	\medbreak
	For now, we can see that each term of this expansion is given by some
	\begin{equation*} 
		\mathcal{M}_{c,1}\left(d^c\mathfrak{F}(y^1)\circ \kappa^c(\tilde{y}^0)\right) \bigl[Q \Psi_0\bigr](x^1),
	\end{equation*} 
	where $Q$ is a polynomial in $x,\xi$ whose coefficients depend on $y^1$. So, we can relate its $L^2(\d X^1)$ norm with the sup norm of $Q$'s coefficients.
	\subsection{Study of the expansion}\label{sec:study-of-the-expansion}
	Now, we can state the following:
	\begin{Proposition}[Rewriting the expansion]\label{Qjkl}
		Let $u\in S_{\delta,\nu}\left(\R^{d_\perp},\text{Poly }\times \text{ Gauss}\right)$ with $\nu<1/6$ and $\delta+\nu<1/2$ and write it as 
		\begin{align*}
			u(x,y)&=\sum_{\substack{\gamma\in \N^{d_{\scalerel*{\parallel}{\perp}}} \\|\gamma|\leq N}} u_\gamma(y) \mathcal{M}_{\text{c},h}\left(\kappa^\text{c}( y)\right)\Bigl[\left(\ \boldsymbol{\cdot}\ /\sqrt{h}\right)^\gamma \varphi_0(\ \boldsymbol{\cdot}\ )\Bigr](x), \text{ for } x\in \R^{d_{\scalerel*{\parallel}{\perp}}}, y\in \R^{d_\perp},
			\\&=\mathcal{M}_\text{c}\left(\kappa^\text{c}(y)\right)\Lambda_{h,x}\Bigl[P_y\Psi_0\Bigr](x).
		\end{align*}
		Then $\forall N\in\N^*,$ there exists a family of polynomials in $x$ named $Q^{k,l}(P)$ such that:
		\begin{itemize}
			\item For every $N\in \N$,
			\begin{align*} 
				\bigl(\hat{T}_{\mathcal{I}_1}\bigr)^* U^{t_0} \hat{T}_{\mathcal{I}_0}u(x,y)&=\mathcal{M}_{\text{c}}\left(d^\text{c} \mathfrak{F}(y)\circ \kappa^\text{c}(y^0(y))\right)\Lambda_{h,x}\left[\sum_{2l+k<2N} h^{k/2+l} Q^{(k,l)}(P)_{y} \Psi_0\right](x)
				\\&\qquad+R_N(x,y),
			\end{align*}
			with 
			\begin{align*} 
				\left\|R_N\right\|_{L^2}&\leq C  \sup_{2l+k=2N}h^N\sup_{\substack{|\alpha_1|+|\alpha_2|\leq 2l , \\ |\beta_1|+|\beta_2|=k}} \left\|N_\infty\left(\partial^{\alpha_1+\beta_1}_y P_y\right)\right\|_{L^2(\d y)} \left\|\kappa^\text{c}\right\|^{|\beta_1|+2|\alpha_2|+3|\beta_2|+d_{\scalerel*{\parallel}{\perp}}}
				\\&\leq C h^{-\nu d_{\scalerel*{\parallel}{\perp}}} \max ( h^{2N(1/2-\delta-\nu)}, h^{N(1/2-3\nu)}).
			\end{align*}
			\item $y^0(y)$ defined by the following between the isotropic manifolds $\mathcal{I}_1$ and $\mathcal{I}_0$:
			\begin{equation*}
				F\left(\overline{x^0}(y^0(y)),y^0(y),\overline{\xi^0}(y^0(y)),\overline{\eta^0}(y^0(y))\right)=\left(\overline{x^1}(y),y,\overline{\xi^1}(y),\overline{\eta^1}(y)\right).
			\end{equation*}
			\item $Q^{(0,0)}(P)=\left|\det \partial_{y^1} y^0(y^1)\right|^{1/2} P_{y^0(y)}$.
			\item $Q^{(k,l)}$ is a polynomial of degree $\deg P+3k+4l,$ the map $P\mapsto Q^{(k,l)}(P)$ is linear, with coefficients controlled by:
			\begin{equation*} 
				\left\|N_\infty\left(Q^{(k,l)}(P)\right)(y)\right\|_{L^2}\leq C\sup_{\substack{|\alpha_1|+|\alpha_2|\leq 2l , \\ |\beta_1|+|\beta_2|=k}} \left\|N_\infty\left(\partial^{\alpha_1+\beta_1}_y P_y\right)\right\|_{L^2(\d y)} \left\|\kappa^\text{c}\right\|^{|\beta_1|+2|\alpha_2|+3|\beta_2|}.
			\end{equation*} 
			\item Moreover, the coefficients of $Q^{(k,l)}(P)$ are sums of products between functions involving $a, \psi$ or their derivatives taken at the point $(0,y,0,0)$ and coefficients of $P$ or their derivatives taken at the point $y^0(y)$.
			\item Each term of the sum is in the $ S_{\delta,\nu}\left(\R^{d_\perp},\text{Poly }\times \text{ Gauss}\right)$ space.
		\end{itemize}
		
	\end{Proposition}
	
	\begin{proof}
		This expansion can be obtained by reordering the terms of the expansion with respect to the order of the stationary phase and order of expansion for the integrand we are considering.
		\medbreak
		Hence, $Q^{(k,l)}(P)$ corresponds to the polynomial obtained when considering the term of order $h^{k/2}$ for the expansion of the integrand and the one of order $h^l$ for the stationary phase lemma.
		\medbreak
		\textbf{Let us focus on the study of $Q^{(k,l)}(P)$.}
		\begin{itemize}
			\item Let us start with the study of $Q^{(0,0)}$. Looking at the leading term of the expansion of the integrand, we can see that we obtain a factor given by 
			\begin{equation*} 
				\frac{a\left(0,y^1,0,0\right)}{\left|\det \partial^2_{x,\xi} \psi_2\left(0,y^1,0,0\right)\right|^{1/2}}.
			\end{equation*} 
			Recall that the main order of $a$ is given by (see \ref{a0} and the remarks above it)
			\begin{equation*} 
				\left|a_0\left(0,y^1,0,0\right)\right|=\left|\det \partial_{X,\Xi}^2 \psi\left(0,y^1,0,0\right)\right|^{1/2}.
			\end{equation*} 
			Similarly to the proof of lemma \ref{compare dcF}, we use the lower triangular block-structure of 
			\newline$\partial_{X,\Xi}^2 \psi\left(0,y^1,0,0\right)$ to deduce that the factor is nothing else than
			\begin{equation*} 
				\left|\det \partial_{y,\eta}^2\psi (0,y^1,0,0)\right|^{1/2}.
			\end{equation*} 
			
			Moreover, from equation (\ref{pt_crit}), we can see that
			\begin{equation*} 
				\partial^2_{y,\eta} \psi(0,y^1,0,0)=d \overline{y^0}(y^1).
			\end{equation*} 
			
			\item The computation of the degree of $Q^{(k,l)}(P)$ is also an easy derivation of the previous computations. 
			For the stationary phase, the degree can change when a derivative in $y^0$ hits $\mathcal{M}_{c,1}(\kappa^c(y^0))$, it can increase by at most $2$ per derivative (see the proof of lemma \ref{non statio}). As the order $h^l$ of the stationary phase lemma involves $2l$ derivatives, we obtained the $4l$ term as expected.
			\medbreak
			Regarding the expansion in $x^1,\xi^0$, the worse case scenario for the polynomial's degree next to the $h^{k/2}$ in the integrand expansion is the same as in section \ref{section1st}, giving an increase of $3k$.
			\item The linearity in $P$ comes from the linearity of the integral.
			\item Estimation of the $L^2$ norm of the coefficients of $Q^{(k,l)}(P)$:
			\medbreak
			Let us focus our study on the terms that might be problematic, we need to see what are the worse terms for the derivatives of the $A_{2l}$ operator to hit and what are the worse terms to expand in $x^1,\xi^0$.
			\begin{itemize}\label{worse cases}
				\item For derivatives in $y^0$, we look at the proof of lemma \ref{non statio}, the worse places for derivatives to fall are either on $u^\gamma$ (hence coefficients of our polynomial $P$ in our setting) which are estimated with $N_\infty(\partial^{\alpha_1}P)$ or on the matrices $\Gamma_{\scalerel*{\parallel}{\perp}}(y^0)$ of the squeezed coherent state which are estimated with (\ref{derive squeezed}), giving a $\left\|\kappa^{\text{c}}\right\|^{2|\alpha_2|}$ factor.
				\item For the expansion in $x^1,\xi^0$, there are three scenarios to take into account: 
				\medbreak
				Expanding the cubic and higher order portion of the phase, which corresponds to the polynomial $\tilde{P}_k$, estimated just as in (\ref{higher order}) by $\left\|\kappa^\text{c}\right\|^{3|\beta_2|}$.
				\medbreak
				Expanding the critical point $\tilde{y}^0$ inside the $S_\delta$ part (that is coefficients of $P_{\tilde{y}^0}$): this expansion involves $|\beta_1|$ derivatives of the coefficients of $P_y$ along with multiplication by a (otherwise bounded) polynomial of degree $|\beta_1|$, see (\ref{v_gamma}). 
				\newline This leads to a control with $\left\|N_\infty\left(\partial^{\beta_1}P\right)\right\|_{L^2} \left\|\kappa^\text{c}\right\|^{|\beta_1|}$.
				\medbreak
				Expanding the critical point $\tilde{y}^0$ inside $\Gamma_{\scalerel*{\parallel}{\perp}}(y^0)$, similarly this is controlled by (\ref{derive squeezed}) and (\ref{higher order}) and also gives an estimate by $\left\|\kappa^\text{c}\right\|^{3|\beta_2|}$.
			\end{itemize}
			Now that we have described the effect of derivatives on each terms, it is easy to see that terms where we mix the derivatives can be dealt the same way which gives the result.
		\end{itemize}
		\textbf{Finally, we explain how to control remainders of this expansion.} 
		
		The goal of the estimate on remainders is to see whether they can still be made arbitrarily small as $M$ goes to infinity. Using similar arguments as in section \ref{section1st},
		we can see that it is the case.
		Just as in section \ref{DL dir centrale}, we consider the $l$-th term of the expansion, select a term in the sums composing the integrand and write it as:
		\begin{multline*} 
			h^{2l}I_{\alpha,\gamma}(X^1)=h^{2l}\int\int e^{\frac{i}{h} \psi_{\text{tot}}(X^1,\xi^0,x^0)} b_{\alpha,\gamma}\left(\tilde{y}^0(X^1,\xi^0)\right) \partial^{\alpha_3} \mathcal{M}_{c,1}\left(\kappa^c(\tilde{y}^0(X^1,\xi^0))\right)\bigl[X^{\gamma'} \Psi_0\bigr](x^0) 
			\\\shoveright\phantom{}\d x^0 \d\xi^0,
		\end{multline*} 
		(with a $b_{\alpha,\gamma}$ composed of $A_\alpha(X^1,\xi^0)\partial^{\alpha_1} a \partial^{\alpha_2} u_\gamma$ and $|\alpha_1|+|\alpha_2|+|\alpha_3|\leq 2l$).
		Then, we expand it as in section \ref{DL dir centrale} and study the remainders that appear (after rescaling them with the $(\Lambda_{h,x})^*$ operator).
		\begin{itemize}
			\item The remainders of type 1 defined in (\ref{DL_exp_a}) verify similar estimates as their counterpart from section \ref{section1st}, see (\ref{symbol remainder}):
			\begin{equation*} 
				|\partial^\beta_{x^1,\xi^0} r_M^{\text{type}1}|\leq C_\beta \jp{x^1,\xi^0}^{3M}, \quad \forall \beta\in \N^{2d_{\scalerel*{\parallel}{\perp}}`} 
			\end{equation*} 
			uniformly in $y^1$, however, the integral associated is slightly different and is given by 
			\begin{equation}\label{rest 1st type}
				R_M^{\text{type} 1}(X^1)=\frac{1}{(2\pi)^d}\int_{T^*\R^{d}} e^{i \psi_2(x^1,x^0,\xi^0)} r_M^{\text{type}1}\left(x^1,\xi^0\right) v_\gamma(x^0,\tilde{y^0}(\sqrt{h}x^1,\sqrt{h} \xi^0))\d x^0\d \xi^0,
			\end{equation}
			with 
			\begin{align*}
				v_\gamma(x^0,\tilde{y}^0(\sqrt{h}x^1,\sqrt{h} \xi^0))&=\partial^{\alpha_2}_y u_\gamma\left(\tilde{y}^0\right) \partial^{\alpha_3}_y\mathcal{M}_{c,1}\left(\kappa^c(\tilde{y}^0)\right)\bigl[X^{\gamma'} \Psi_0\bigr](x^0).
				\\&=\partial^{\alpha_2}_y u_\gamma\left(\tilde{y}^0\right) \left|\det (\Im(\Gamma_{\scalerel*{\parallel}{\perp}}(\tilde{y}^0)))\right|^{1/4} Q_{\alpha_3} \left(\Im \Gamma_{\scalerel*{\parallel}{\perp}}(\tilde{y}^0)^{1/2} x\right)e^{i \frac{x\cdot\Gamma_{\scalerel*{\parallel}{\perp}}(\tilde{y}^0) x}{2}},
			\end{align*}
			where $N_\infty(Q_{\alpha_3})\leq C h^{-2\nu|\alpha_3|}$
			\medbreak
			The obstacle that prevents us from using lemma \ref{1st_type} is that $v_\gamma$ depends on $x^1,\xi^0$. To avoid this issue, we will obtain direct estimates on the $L^2(\d X^1)$ norm of this integral by integration by parts. The first thing to notice is that this integral has a compact support in $y^1$ thanks to the amplitude $a$, but this is not the case in $x^1$ as $a$ is evaluated at point $\sqrt{h}x^1$. As a consequence, it is enough to obtain $L^2(\d x^1) L^\infty(\d y^1)$ estimates to control the $L^2(\d X^1)$ norm. In what follows, all estimates will be uniform with respect to $y^1$.
			\medbreak
			Recall that the phase $\psi_2$ is quadratic and can be written as 
			\begin{equation*} 
				\psi_2(x^1,x^0,\xi^0)=\left\langle \underbrace{\begin{pmatrix}
						A_{1,1} & A_{1,2} \\ A_{1,2}^T & A_{2,2}
				\end{pmatrix}}_{\coloneq A} \begin{pmatrix} x^1 \\ \xi^0\end{pmatrix},\begin{pmatrix} x^1 \\ \xi^0\end{pmatrix} \right\rangle-x^0\cdot \xi^0.
			\end{equation*} 
			Performing first the integral in $x^0$ gives a Fourier transform:
			\begin{align*}\label{rest 1st type}
				R_M^{\text{type } 1}(X^1)=\frac{1}{(2\pi)^d}\int_{\R^{d}}& e^{i A(x^1,\xi^0)^{\otimes 2}} r_M^{\text{type}1}\left(x^1,\xi^0\right) \widehat{v_\gamma}(\xi^0,\tilde{y}^0(\sqrt{h}x^1,\sqrt{h} \xi^0))\d \xi^0,
			\end{align*}
			where 
			\begin{equation}\label{wide v}
				\widehat{v_\gamma}(\xi^0,\tilde{y}^0(\sqrt{h}x^1,\sqrt{h} \xi^0))=\partial^{\alpha_2}_y u_\gamma\left(\tilde{y}^0\right) \left|\det (\Im(\Gamma_{\scalerel*{\parallel}{\perp}}(\tilde{y}^0)))\right|^{-1/4} Q_{\alpha_3} \left(\Im \Gamma_{\scalerel*{\parallel}{\perp}}(\tilde{y}^0)^{-1/2} \xi^0\right)e^{i \frac{\xi^0\cdot\Gamma_{\scalerel*{\parallel}{\perp}}(\tilde{y}^0)^{-1} \xi^0}{2}}.
			\end{equation}
			To obtain $L^2$ estimates in $x^1$, for $x^1\neq 0$, we introduce the operator 
			\begin{equation*} 
				L\coloneq\frac{\langle A_{1,2} x^1, \partial_{\xi^0}\rangle}{2\|A_{1,2} x^1\|^2},
			\end{equation*} 
			(which is well-defined as $\det A_{1,2}\neq 0$ because we treat the case without auxiliary variables) so that $L e^{2i  A_{1,2} x^1\cdot\xi^0}=e^{2i  A_{1,2} x^1\cdot\xi^0}$.
			Hence, for any $k\in \mathbb{N}$,
			\begin{align*}
				\left|R_M^{\text{type }1}(X^1)\right|\leq C \int_{\R^d} \Big|(L^*)^k \Big[r_M^{\text{type}1} \widehat{v_\gamma}\Big]\Big| \d \xi^0.
			\end{align*}
			
			Then we check using lemma \ref{derive squeezed} that 
			\begin{equation*} 
				\Big|(L^*)^k \Big[r_N^{\text{type}1} \hat{v_\gamma}\Big]\Big|\leq C \sum_{M_1+M_2=3M} \jp{x^1}^{M_1-k} \jp{\xi^0}^{M_2} \sum_{|\beta|\leq k} \partial^\beta_{\xi^0} [\widehat{v_\gamma} ].
			\end{equation*} 
			Hence, using Hölder,
			\begin{align*}
				\left\|R_M^{\text{type }1}\right\|_{L^2(\d x^1)L^\infty(\d y^1)}\leq C \left\|\jp{x^1}^{-d_{\scalerel*{\parallel}{\perp}}}\right\|_{L^2(\d x^1)} \sup_{|\beta_1|+|\beta_2|\leq 3M+d_{\scalerel*{\parallel}{\perp}}} \left\| \jp{\xi^0}^{|\beta_1|} \partial^{\beta_2}_{\xi^0} \widehat{v_\gamma}\right\|_{L^2(\d \xi^0) L^\infty(\d x^1)}. 
			\end{align*}
			Hence, we have to estimate $\| \jp{\xi^0}^{|\beta_1|} \partial^{\beta_2}_{\xi^0} \widehat{v_\gamma}\|_{L^2(\d \xi^0)} $, for this we notice that derivatives $\partial^{\beta_2}_{\xi^0}$ can either fall on the stationary point  $\tilde{y^0}(\sqrt{h}x^1,\sqrt{h}\xi^0)$ or on ``regular'' appearances of $\xi^0$ inside the polynomial and Gaussian of $\widehat{v_\gamma}$, see its definition in (\ref{wide v}). 
			Due to the $\sqrt{h}\xi^0$ inside the stationary point and lemma \ref{derive squeezed}, we check that when derivatives hit the stationary point, the norm is not increased, hence we only have to use a bound similar to lemma \ref{control_xy}.
			\medbreak
			Finally, when we add the global factor $h^{2l+3M}$, we get
			\begin{equation*} 
				h^{2l+3M}\left\|R_M^{\text{type }1}\right\|_{L^2(\d x^1)L^\infty(\d y^1)}\leq C h^{2l+3M} \left\|N_\infty\left(\partial^{\alpha_2}_y P_y\right)\right\|_{L^2(\d y)} \left\|\kappa^\text{c}\right\|^{2|\alpha_3|+3M+d_{\scalerel*{\parallel}{\perp}}}.
			\end{equation*} 
			\item For remainders of type 2 defined in (\ref{DL_exp_a}), we adapt the estimate of section \ref{section1st} in a similar way.
			\item Finally, in order to control terms involving $r_M^v$, we also want to use the same method as above, the main difference being that we have a different estimate on the remainder $r_M^v$:
			\begin{equation}
				|r_N^v|\leq C \jp{ x^1,\xi^0}^N \sup_{|\beta_1|+|\beta_2|=N}\left\|N_\infty\left(\partial^{\alpha_2+\beta_1}_y P_y\right)\right\|_{L^2(\d y)} \left\|\kappa^\text{c}\right\|^{2|\alpha_3|+2|\beta_2|}. 
			\end{equation}
			\medbreak
			We need a bound for 
			\begin{equation*} 
				\left|\int e^{\frac{i}{h} \psi_2\left(x^1,x^0,\xi^0\right)} \tilde{P}_k(x^1,y^1,\xi^0)r_{M-k}^v(x^1,y^1,\xi^0,x^0) \d x^0 \d \xi^0\right|,
			\end{equation*} 
			by changing the order of integration, we will instead look for a bound to
			\begin{equation*} 
				\left|\int e^{\frac{i}{h} \psi_2\left(x^1,x^0,\xi^0\right)} \partial^M_{x^1,\xi^0}v_\gamma(x^0,\tilde{y}^0(s\sqrt{h}x^1,y^1,s\sqrt{h}\xi^0).[x^1,\xi^0]^{\otimes M} \d x^0 \d \xi^0\right|.
			\end{equation*} 
			This integral is very similar to the one of (\ref{rest 1st type}), and the same change of variables is needed to control it.
			\medbreak
			Then, doing the same argument based on lemma \ref{1st_type} and using estimates of lemma \ref{control_xy} and lemma \ref{derive squeezed} we obtain the announced bound.
		\end{itemize}
		
	\end{proof}	
	
	\begin{Remark}
		Assumptions $\nu<1/6$ and $\delta+\nu<1/2$ ensure that the expansion is convergent.
	\end{Remark}

	\subsubsection{The parallel metaplectic operator}\label{metap_simple}
	In this section, we explain the proof of lemma \ref{compare dcF}:
	
	\begin{proof}
		To obtain the first part of the lemma, we use the fact that $\mathfrak{F}(\mathcal{I}_m)=\mathcal{I}_m$ so that its linearization of points $(0,y^1,0,0)\in \mathcal{I}_m$ has the following structure:
		\begin{equation}\label{struct}
			d\mathfrak{F}(0,y^1,0,0)=\begin{pmatrix}
				\ast & 0 & \ast & \ast 
				\\ \ast & \ast & \ast  & \ast
				\\ \ast & 0 & \ast & \ast
				\\ 0 & 0 &0 & \ast
			\end{pmatrix}.
		\end{equation}
		From the link between symplectomorphism and associated generating function, we know that
		\begin{equation*} 
			d\mathfrak{F}=\begin{pmatrix}
				(\partial^2_{X,\Xi}\psi)^{-1} & -(\partial^2_{X,\Xi} \psi)^{-1} \partial^2_{\Xi,\Xi}\psi
				\\ \partial^2_{X,X}\psi(\partial^2_{X,\Xi}\psi)^{-1} & (\partial^2_{X,\Xi} \psi)^T-\partial^2_{X,X} \psi (\partial^2_{X,\Xi} \psi)^{-1} \partial^2_{\Xi,\Xi} \psi
			\end{pmatrix},
		\end{equation*} 
		and that $\psi_2(y^1)$ generates the linear symplectic application
		\begin{equation*} 
			\begin{pmatrix}
				(\partial^2_{x,\xi}\psi)^{-1} & -(\partial^2_{x,\xi} \psi)^{-1} \partial^2_{\xi,\xi}\psi
				\\ \partial^2_{x,x}\psi(\partial^2_{x,\xi}\psi)^{-1} & (\partial^2_{x,\xi} \psi)^T-\partial^2_{x,x} \psi (\partial^2_{x,\xi} \psi)^{-1} \partial^2_{\xi,\xi} \psi
			\end{pmatrix}.
		\end{equation*} 
		From (\ref{struct}), we can see that $(\partial^2_{X,\Xi}\psi)^{-1}$ is lower block triangular and deduce that 
		\begin{equation*} 
			(\partial^2_{x,\xi}\psi)^{-1}=d_x\mathfrak{F}_x.
		\end{equation*} 
		Then looking at the lower-left block in (\ref{struct}) yields 
		\begin{equation*} 
			\partial^2_{x,x}\psi(\partial^2_{x,\xi}\psi)^{-1}=d_x\mathfrak{F}_{\xi}.
		\end{equation*} 
		Finally, we obtain the last two remaining blocks by studying the lower-right block in (\ref{struct}).
		\medbreak
		For the second part of this lemma, we have to study the differences between the linearization of $\mathfrak{F}$ and the one of $F$.
		\medbreak
		We notice that for any $\rho\in \R^{2d},$ $d\kappa_{\mathcal{I}_0}(\rho)$ and $(d\kappa_{\mathcal{I}_1}(\rho))^{-1}$ have the following structure:
		\begin{equation*} 
			\begin{pmatrix}
				I & \ast & 0 & 0
				\\ 0 & I & 0 & 0
				\\0 & \ast & I & 0
				\\ 0 & \ast & \ast & I
			\end{pmatrix},
		\end{equation*} 
		with all matrices $\ast$ bounded by as constant $C>0$ by (\ref{I proche}).
		Hence, by writing, for any $\rho\in \R^{2d},$ 
		\begin{equation*} 
			d F(\rho)=\begin{pmatrix}
				A_{1,1} & A_{1,2} & A_{1,3} & A_{1,4}
				\\ A_{2,1} & A_{2,2} & A_{2,3} & A_{2,4}
				\\ A_{3,1} & A_{3,2} & A_{3,3} & A_{3,4}
				\\ A_{4,1} & A_{4,2} & A_{4,3} & A_{4,4}
			\end{pmatrix},
		\end{equation*} 
		we can see that 
		\begin{equation*} 
			d \mathfrak{F}(0,y^0,0,0)=\begin{pmatrix}
				A_{1,1}+\ast A_{2,1} & \ast & A_{1,3}+A_{1,4}\ast + \ast(A_{2,3}+A_{2,4}\ast) & \ast
				\\ \ast & \ast & \ast &\ast
				\\ A_{3,1}+\ast A_{2,1} & \ast & A_{3,3}+A_{3,4}\ast +\ast(A_{2,3}+A_{2,4}\ast) & \ast
				\\ \ast & \ast & \ast &\ast
			\end{pmatrix},
		\end{equation*} 
		(where the corresponding $\rho$ at which matrices $A_{i,j}$ are computed is $(\overline{x^0}(y^0),y^0,\overline{\xi^0}(y^0),\overline{\eta^0}(y^0))$) and then
		\begin{equation*} 
			d^c\mathfrak{F}(y^1)=\begin{pmatrix}
				A_{1,1}+\ast A_{2,1}& A_{1,3}+A_{1,4}\ast + \ast(A_{2,3}+A_{2,4}\ast)
				\\ A_{3,1}+\ast A_{2,1} & A_{3,3}+A_{3,4}\ast +\ast(A_{2,3}+A_{2,4}\ast) 
			\end{pmatrix}.
		\end{equation*} 
		\medbreak
		Finally, we know that 
		\begin{equation*} 
			d F(0,0,0,0)=\begin{pmatrix}
				A & 0 & E & 0
				\\	0 & B & 0 & 0
				\\	F  & 0 & C & 0
				\\	0 & 0 & 0 & D
			\end{pmatrix}.
		\end{equation*} 
		Hence, we can see using the mean value theorem that 
		\begin{align*}
			\left\| d^c\mathfrak{F}(y^1)-\begin{pmatrix} A & E
				\\ F & C
			\end{pmatrix}\right\|&\leq C(\|A_{1,1}-A\|+\|A_{1,3}-E\|+\|A_{3,1}-F\|+\|A_{3,3}-C\|
			\\&\quad+\|A_{2,1}\|+\|A_{2,3}\|+\|A_{2,4}\|+\|A_{1,4}\|+\|A_{3,4}\|)
			\\&\leq C \left\|(\overline{x^0}(y^0),y^0,\overline{\xi^0}(y^0),\overline{\eta^0}(y^0))\right\|.
		\end{align*}
	\end{proof}

	\section{Reduction of the study} \label{section reduction}
	In this part and the next one, we want to prove the result of theorem \ref{th_trapped_set}. We are looking for a description of the propagator $\chi^w e^{-it\mathpzc{p}/h}$ acting on coherent states located at distance at most $h^{\tau}$ of $K^\delta$ (with $\chi$ cutoff in a $h^{\epsilon_1}$ neighborhood in the transverse direction to the trapped set) for times $t\leq \tau\frac{|\log h|}{6 \lambda^{\text{c}}}$ (or $C|\log h|$ with any $C>0$ if $\lambda^{\text{c}}=0$).
	\medbreak
	Let us start by making some reductions.
	\subsection{Localization at all times}
	
	The first step is to add a cutoff all along the trajectory forcing the localization close to $K^\delta$. We will use the time $t_0$ defined in section \ref{preli_dyn_clas} (recall that $t=nt_0$).
	\begin{Proposition}\label{cutoff_each_step}
		Let $\chi\in C^{\infty}_c(\R^{2d})$ such that $\chi=1$ on $K^\delta(\epsilon_1/2)\cap p^{-1}(-\delta,\delta)$ and 
		\newline 
		$\supp \chi \subset K^{3\delta/2}(\epsilon_1)\cap p^{-1}(-3\delta/2,3\delta/2)$, $t_0$ defined in section \ref{preli_dyn_clas} and $n\leq C \log(1/h)$ then 
		\begin{equation*} 
			\left\|\chi^w e^{-int_0\mathpzc{p}/h}\chi^w-\chi^w e^{-it_0\mathpzc{p}^w/h} \chi^w e^{-it_0 P/h}\chi^w \dots \chi^w e^{-it_0\mathpzc{p}^w/h}\chi^w\right\|_{L^2}= O(h^\infty).
		\end{equation*} 
	\end{Proposition}
	We will prove this proposition in the two following sections.
	\subsubsection{Properties of the classical dynamics}
	Let us start by recalling the classical side of the picture. In section \ref{fuite_def_esc}, we have used escape functions to show that the trajectory we consider must stay inside a compact set. More precisely, we can use the setting given by lemma \ref{G2} about escape functions: we can construct a smooth function $G_2$ which verifies
	\begin{itemize}
		\item On $p^{-1}(-2\delta,2\delta)\setminus K^\delta(\epsilon_1/2),$
		\begin{equation*} 
			H_pG_2 \geq C_0,
		\end{equation*} 
		with $C_0$ a constant.
		\item On $p^{-1}(-2\delta,2\delta)$, $G_2=0$ on a neighborhood of $K^{\delta}$.
		\item  On $p^{-1}(-2\delta,2\delta)$, we have
		\begin{equation*} 
			H_pG_2\geq 0,
		\end{equation*} 
		which implies that $G_2$ is non-decreasing along the flow trajectories. 
	\end{itemize}
	Our goal here is to obtain a control over $\Phi^{-t_0} (\supp \chi)$ and $\Phi^{t_0}(\supp \chi)$ (i.e. points entering or leaving $\supp \chi$ under the dynamics).
	\medbreak
	Let us first notice that, as $K^{2\delta}(\epsilon_1)$ is compact, there exists a $\alpha>0$ such that 
	\begin{equation*} 
		K^{2\delta}(\epsilon_1)\subset (-\alpha \leq G_2 \leq \alpha).
	\end{equation*}
	We may now assume that $t_0$ is taken large enough so that 
	\begin{equation}\label{C_0t_0} 
		C_0t_0\geq 3 \alpha.
	\end{equation} 
	Then we recall another property of $t_0$ that was established in lemma \ref{def_exit}:
	\medbreak
	If $\rho$ is a point in $K^\delta(\epsilon_1/2)$ and $\Phi^t(\rho)$ is also in $K^\delta(\epsilon_1/2)$ for a $t>2t_0$, then 
	\begin{equation}\label{non retour}
		\forall s\in [t_0,t-t_0],\ \Phi^{s}(\rho)\in K^\delta(\epsilon_1/2).
	\end{equation}
	
	Recall that $\chi\in C^\infty(p^{-1}(-2\delta,2\delta),[0,1])$ verifies $\chi=1$ on $K^{\delta}(\epsilon_1/2)$ and $\supp \chi \subset K^{2\delta}(\epsilon_1)$.
	\medbreak
	Let us now consider $\chi_+ \in C^\infty(p^{-1}(-2\delta,2\delta),[0,1])$ such that $\chi_+=1$ on $\{G_2\geq 2 \alpha\}$ and 
	\newline\indent $\text{supp } \chi_+ \subset \{G_2 \geq \alpha\}$.
	\medbreak
	We then define $\chi_- \in C^\infty(p^{-1}(-2\delta,2\delta),[0,1])$ such that $\chi+\chi_-+\chi_+=1$ on $p^{-1}(-7\delta/4,7\delta/4)$.
	\medbreak
	From this construction, we get a few properties: 
	\begin{itemize}
		\item Consider a point $\rho \in \supp \chi_+$, then $G_2(\Phi^t(\rho))\geq \alpha$ for all $t\geq 0$, hence $\Phi^t(\rho)\notin K^\delta(\epsilon_1/2)$, which, in turn, implies that $H_pG_2(\Phi^t(\rho))\geq C_0$ and:
		\begin{equation*} 
			G_2(\Phi^{t_0}(\rho))\geq C_0t_0+\alpha \geq2\alpha.
		\end{equation*} 
		As a consequence, $\chi_+(\Phi^{t_0}(\rho))=1$ which means that 
		\begin{equation}\label{image chi_+}
			\Phi^{t_0}(\supp \chi_+)\cap (\supp \chi \cup \supp \chi_-)=\emptyset.
		\end{equation}
		\item Similarly, let $\rho\in \supp \chi$ and assume that $\Phi^{t_0}(\rho)\notin K^{\delta}(\epsilon_1/2)$ (or equivalently $\phi^{t_0}(\rho)\in \supp \chi_+\cup \supp \chi_-$). Then we can show that 
		\begin{equation}\label{strict croissance}
			\forall s\geq t_0,\ \Phi^s(\rho)\notin K^{\delta}(\epsilon_1/2).
		\end{equation}
		\medbreak 
		Indeed, if $\rho\in K^{\delta}(\epsilon_1/2)$, it is a consequence of (\ref{non retour}), otherwise $\rho\in K^{2\delta}(\epsilon_1)\setminus K^{\delta}(\epsilon_1/2)$ but since $\Phi^{t_0}(\rho)\notin K^{\delta}(\epsilon_1/2),$ we can use lemma \ref{fuite_def_esc} (with $V_1=K^{\delta}(\epsilon_1/2)$ and $\mathbf{U_0}=K^{2\delta}(\epsilon_1)$) to deduce that $\Phi^{t_0}\notin K^{2\delta}(\epsilon_1)$. Finally, we can apply lemma \ref{def_exit} to $K^{2\delta}(\epsilon_1)$. (To be precise, doing this might require to consider a bigger $t_0$ as we apply lemmas \ref{fuite_def_esc} and \ref{def_exit} to different sets than in section \ref{fuite_def_esc}.)
		\medbreak
		Equation (\ref{strict croissance}) gives us as in the previous case:
		\begin{equation*} 
			G_2(\Phi^{2t_0}(\rho))\geq C_0t_0+G_2(\Phi^{t_0}(\rho))\geq C_0t_0+G_2(\rho)\geq C_0t_0-\alpha \geq 2\alpha.
		\end{equation*}
		As a consequence, $\chi_+(\Phi^{2t_0}(\rho))=1$ which means that
		\begin{equation}\label{image chi}
			\Phi^{2t_0}\left(\supp \chi\cup \Phi^{-t_0}(\supp \chi_{\pm})\right)\cap \left(\supp \chi \cup \supp \chi_-\right)=\emptyset.
		\end{equation}
	\end{itemize}
	\subsubsection{Proof of proposition \ref{cutoff_each_step}}
	Now that we have information on the classical behavior of $\supp\chi$, we may move to the quantum side of the picture. Let us quantize $\chi,\chi_-,\chi_+$ (with Weyl Quantization for instance) so that $R_\infty\coloneq\chi_A^w+\chi_-^w+\chi_+^w-I$ is also a pseudodifferential operator and verifies
	\begin{equation*} 
		\text{WF}_h(R_\infty)\cap p^{-1}(-7\delta/4,7\delta/4)=\emptyset.
	\end{equation*}
	Moreover, using the semiclassical wavefront sets of these operators and properties of Fourier Integral Operators, we can translate equation (\ref{image chi_+}) into
	\begin{equation*} 
		\left.\begin{array}{lr}
			\|\chi_-^w e^{-it_0\mathpzc{p}/h} \chi_+^w\|_{L^2 \to L^2}\\
			\|\chi^w e^{-it_0\mathpzc{p}/h} \chi_+^w\|_{L^2 \to L^2}
		\end{array}\right\} = O(h^\infty),
	\end{equation*} 
	and (\ref{image chi}) into 
	\begin{equation*}
		\left.\begin{array}{lr}
			\|\chi_-^w e^{-it_0\mathpzc{p}/h} \chi_{\pm}^we^{-it_0\mathpzc{p}/h} \chi^w \|_{L^2 \to L^2}\\
			\|\chi^w e^{-it_0\mathpzc{p}/h} \chi_{\pm}^we^{-it_0\mathpzc{p}/h} \chi^w \|_{L^2 \to L^2}
		\end{array}\right\} = O(h^\infty).
	\end{equation*}
	We can now conclude the proof of proposition \ref{cutoff_each_step}:
	\medbreak
	Let us set $\mathcal{A}=\{A,+,-,\infty\}$ and define for $a\in \mathcal{A},$
	\begin{equation*} 
		B_a=\left\{
		\begin{aligned}
			&e^{-it_0\mathpzc{p}/h}\chi^w\quad\ \text{if } a=A,
			\\&e^{-it_0\mathpzc{p}/h}\chi_a^w\quad\ \text{if } a\in\{+,-\},
			\\& e^{-it_0\mathpzc{p}/h} R_\infty \quad\text{if } a= \infty.
		\end{aligned}\right.
	\end{equation*} 
	Then as $\sum_{a\in \mathcal{A}}B_a=e^{-it_0\mathpzc{p}/h}$, we have
	\begin{equation}\label{decompo Ba}
		\chi_A^w e^{-int_0\mathpzc{p}^w/h} \chi_A^w=\chi_A^w \sum_{\boldsymbol{a} \in \mathcal{A}^{n-1}} B_{\boldsymbol{a}_{n-1}}\dots B_{\boldsymbol{a}_1} B_A.
	\end{equation} 
	Let us assume consider some $\boldsymbol{a}\in \mathcal{A}^{n-1}$ that is not $(A,\dots,A)$.
	Our goal is to show that 
	\begin{equation}\label{traj impossible}
		\|\chi_A^w B_{\boldsymbol{a}_{n-1}}\dots B_{\boldsymbol{a}_1} e^{-it_0\mathpzc{p}^w/h} \chi_A^w\|=O(h^\infty).
	\end{equation}
	We have several different cases:
	\begin{enumerate}
		\item There exists a $j_\infty\in\llbracket 1, n-1 \rrbracket$ such that $\boldsymbol{a}_{j_\infty}=\infty$, then we should follow the proof of \cite[Lemma 6.5]{NZ_pression_topo}. We know that $\|R_\infty\psi\|_{L^2}$ is small when $\text{WF}_h(\psi)\subset p^{-1}(-7\delta/4,7\delta/4)$, we want to use the fact that $\supp \chi \subset p^{-1}(-3\delta/2,3\delta/2)$ to say that 
		\begin{equation*} 
			R_\infty B_{\boldsymbol{a}_{j_\infty-1}}\dots B_{\boldsymbol{a}_1}e^{-it_0\mathpzc{p}^w/h} \chi^w=O(h^\infty),
		\end{equation*} 
		where the implicit constant is independent of $(\boldsymbol{a}_{j_\infty-1},\dots,\boldsymbol{a}_{1})$.
		This result is obtained by inserting nested $j$ cutoffs in energy ranging from $3\delta/2$ to $7\delta/4$, we refer to \cite[Lemma 6.5]{NZ_pression_topo} for more details.
		\medbreak
		In conclusion, we use $\|e^{-int_0\mathpzc{p}^w/h}\|_{L^2\to L^2}\leq 1$ and $\|\chi^w\|_{L^2\to L^2},\|\chi_a^w\|_{L^2\to L^2}\leq 2$ for $a\in \mathcal{A}$ to obtain equation (\ref{traj impossible}).
		\item Let us now assume that for all $j\in\llbracket 1, n-1 \rrbracket$, $\boldsymbol{a}_j\neq\infty$. 
		\begin{enumerate}
			\item If there exists a $j\in\llbracket 1, n-1 \rrbracket$ such that $\boldsymbol{a}_{j}=+,$ we look at the largest of those $j$ and denote it by $j_+$. We have two cases:
			\begin{enumerate}
				\item Either $j_+\neq n-1$, in which case $a_{j_++1}\in \{A,-\}$ but we have  
				\begin{equation*} 
					\left.\begin{array}{lr}
						\|\chi_-^w e^{-it_0\mathpzc{p}^w/h} \chi_+^w\|_{L^2 \to L^2}\\
						\|\chi^w e^{-it_0\mathpzc{p}^w/h} \chi_+^w\|_{L^2 \to L^2}
					\end{array}\right\} = O(h^\infty),
				\end{equation*}
				so, we again obtain equation (\ref{traj impossible}).
				\item Otherwise, $a_{n-1}=+$ but again we use $\|\chi_A^w e^{-it_0\mathpzc{p}^w/h} \chi_+^w\|_{L^2 \to L^2}\leq \epsilon $,
				to conclude that we have equation (\ref{traj impossible}).
			\end{enumerate}
			\item If there exists no such $j_+$ then there must exist a $j\in\llbracket 1, n-1 \rrbracket$ such that $\boldsymbol{a}_{j}=-,$ we consider $j_-$ the smallest of those. We then know that $\boldsymbol{a}_{j_--1}=A$ (setting $\boldsymbol{a}_0=A$ due to (\ref{decompo Ba})).
			\begin{enumerate}
				\item Either $j_-<n-1$ and $\boldsymbol{a}_{j_-+1}=-$ in which case we can use 
				\begin{equation*}
					\|\chi_-^w e^{-it_0\mathpzc{p}^w/h} \chi_-^w e^{-it_0\mathpzc{p}^w/h} \chi^w\|=O(h^\infty),
				\end{equation*}
				to obtain equation (\ref{traj impossible}),
				\item or we have $j_-=n+1$ or $\boldsymbol{a}_{j_-+1}=A$ and we should use 
				\begin{equation*}
					\|\chi^w e^{-it_0\mathpzc{p}^w/h} \chi_-^w e^{-it_0\mathpzc{p}^w/h} \chi^w\|=O(h^\infty),
				\end{equation*}
				to get equation (\ref{traj impossible}).
			\end{enumerate}
		\end{enumerate}
	\end{enumerate}
	
	Finally, equation (\ref{traj impossible}) implies that 
	\begin{equation*} 
		\|\chi^w e^{-int_0\mathpzc{p}^w/h} \chi^w-\chi^we^{-it_0\mathpzc{p}^w/h}\chi^w \dots e^{-it_0\mathpzc{p}^w/h} \chi^w\|_{L^2\to L^2} \leq (|\mathcal{A}^{n-1}|-1) O(h^\infty),
	\end{equation*} 
	which, using $|\mathcal{A}^{n-1}|=4^{n-1}$ and $n\leq C|\log h|$, finally gives that
	\begin{equation*} 
		\|\chi^w e^{-int_0\mathpzc{p}^w/h} \chi^w-\chi^we^{-it_0\mathpzc{p}^w/h}\chi^w \dots e^{-it_0\mathpzc{p}^w/h} \chi^w\|_{L^2\to L^2}=O(h^\infty).
	\end{equation*}

	\subsection{Introducing coordinates charts along the trajectory}\label{charts}
	In this section, we explain how to replace the propagator $e^{-int_0\mathpzc{p}^w/h}$ with the product of $n$ propagators for a time $t_0$ that are localized. 
	\medbreak
	Recall that in definition \ref{nth chart}, we have explained the coordinates we wanted to use to describe the evolution of the coherent state centered at $\rho$ which is at a distance $h^{\tau}$ from $K^{\delta}$. It involved looking at the iterates of the point that reaches the distance between $\rho $ and $K^{\delta}$ that we denoted by $\tilde{\rho}$. Those iterates were the ``base points'' and are denoted by
	\begin{equation*}
		\tilde{\rho}^k=\Phi^{kt_0}(\tilde{\rho}), \quad \forall k\in \llbracket 1,n\rrbracket.
	\end{equation*}
	We also introduce the iterates of $\rho$ as $\rho^k=\Phi^{kt_0}(\rho)$.
	\medbreak
	Recall that $\tilde{\rho}^k$'s charts were introduced as $(U_k,\kappa^{(k)})$, see definition \ref{charts}.
	\begin{Remark}
		In the case where $\rho\in K$, we can obviously take $\tilde{\rho}=\rho$ which simplifies the computations. If we were just interested in propagation in this case, some problems that we will encounter would not have occurred. We will try to keep mentioning such shortcuts throughout the proof.
	\end{Remark}
	We now introduce an associated cutoff $\pi_k$ such that $\pi_k$ is equal to unity near some $\tilde{U}_k \Subset U_k$ and vanishes outside of $U_k$. Then we quantify these cutoffs into $\Pi^{k}=\Op \left(\pi_{k}\right)$.
	\medbreak
	This gives us local propagators 
	\begin{equation*} 
		\left(\Pi^{k}\right)^* e^{-it_0\mathpzc{p}^w/h} \Pi^{k-1},
	\end{equation*}
	which we will see in $\R^d$ by quantifying our coordinates $\kappa^{(k)}$ via a Fourier integral operator $\mathcal{U}^{k}:L^2(\R^{d}) \to L^2\left(\R^d\right)$ which leads to the study of:
	\begin{equation}\label{def Tk}
		T_{k}=\mathcal{U}^{k} \left(\Pi^{k}\right)^* e^{-it_0\mathpzc{p}^w/h} \Pi^{k-1} \left(\mathcal{U}^{k-1}\right)^*,\quad \text{ for }k\in \llbracket 2,n\rrbracket,
	\end{equation}
	and 
	\begin{equation}\label{def T1}
		T_1=\mathcal{U}^{1} \left(\Pi^{1}\right)^* e^{-it_0\mathpzc{p}^w/h} \Pi_{\alpha_0} \left(\mathcal{U}_{\alpha_0}\right)^*.
	\end{equation}
	We can see that our state stays microlocalized in a microscopic region: for the transverse direction it is a consequence of lemma \ref{cutoff_each_step} and for the central one it is due to the fact that the times considered are smaller than the central Ehrenfest time. As a consequence, the charts introduced above are enough to describe the entire state, contributions in other charts would be $O(h^\infty)$. This implies the following:
	\begin{Lemma}\label{introduce charts}
		In the notation of theorem \ref{th_trapped_set},
		\begin{equation*}
			\mathcal{U}_{\beta}\ \chi^w e^{-int_0\mathpzc{p}^w/h}  (\mathcal{U}_\alpha)^* \varphi_{\kappa_{\alpha}^{-1}(\rho)}=\mathcal{U}_{\beta}(\mathcal{U}^{n})^*T_n \dots T_1 \varphi_{\kappa_{\alpha}^{-1}(\rho)}+O(h^\infty),
		\end{equation*}
		where the $T_j$'s are defined in equations (\ref{def Tk}) and (\ref{def T1}).
	\end{Lemma}
	\section{Description of propagated coherent states.}\label{iter_method}
	\subsection{Successive applications of the method of expansion of the integrand} \label{app_1st}
	\subsubsection{Recalling the previous result}
	Just as in \cite{Lucas}, we want to apply the method of section \ref{section1st} for the propagation during $n_s=\epsilon_s |\log(h)|$ iterations (with $0<\epsilon_s<\tau/(6t_0\lambda_{\text{max}})$) before propagating even further with the second method.
	\begin{Remark}
		This choice of time $n_s$ for which we switch method is arbitrary, any multiple of $\log(1/h)$ smaller than $\frac{\tau|\log(h)|}{6t_0\lambda_{\text{max}}}$ would work, if needed, we might take $\epsilon_s$ arbitrarily small (independent of $h$).
	\end{Remark} 
	\begin{Proposition}\cite[Proposition 5.1]{Lucas}\label{succ_1st}
		For every $0\leq p \leq n_s$ and $K+l< 2N$, there exist polynomials $P_p^{(k,l)}$ and a remainder $R^{(N)}_{n_s}$ such that if we introduce
		\begin{equation*} 
			u_p^{(k,l)}= T\left(\rho^p\right) \mathcal{M}\left(d_\rho [\kappa^{p,0} \Phi^{pt_0} (\kappa_{\alpha_0})^{-1}]\right) \Lambda_h\bigl[P_p^{(j,k,l)} \Psi_0\bigr],
		\end{equation*}
		we have
		\begin{equation*} 
			T_{p}\dots T_2T_{1} \chi_{\alpha_0}^w \varphi_\rho=\sum_{k+l<2N} h^{k/2+l/2} u_{p}^{(k,l)}+R_{p}^{(N)},
		\end{equation*}
		with the polynomials verifying 
		\begin{equation*} 
			P_p^{(0,0)}=1, \quad \deg P_p^{(k,l)}\leq 3k+2l , \quad N_\infty\left(P_{p}^{(k,l)}\right)\leq C p^{k} \left\|d\Phi^{pt_0}(\rho)\right\|^{3k},
		\end{equation*}
		and the remainder:
		\begin{equation*} 
			\left\|R_{n_s}^{(N)}\right\|_{L^2} \leq C h^{N}\left(\sigma_\text{tot}^{(nt_0)}\right)^{6N},
		\end{equation*}
		with $\sigma_\text{tot}^{(t)}=\sup_{\rho\in K^\delta} \left\|d\Phi^t(\rho)\right\|$.
	\end{Proposition}
	\begin{Remark}
		The dependence in $p$ should not be neglected as $p$ can go up to $n$ which is a multiple of $\log(1/h)$. However, compared to $\left\|d\Phi^{pt_0}(\rho)\right\|^{3k}$, the term $p^{2j+k}$ will play a minor role. In fact, $\left\|d\Phi^{pt_0}(\rho)\right\|$ grows as we propagate further and further, deteriorating the precision of our asymptotics and can be a limitation to the validity of the expansion.
		For our expansion to be convergent, it is necessary to have:
		\begin{equation*} 
			\left\|d\Phi^{pt_0}(\rho)\right\|^3 < ch^{-1/2+\epsilon},
		\end{equation*}
		for some $\epsilon>0$ so that the blow up of the coefficients of our polynomials $P_p^{(j,k)}$ does not compensate for the powers of $h$ that we have gained.
		\medbreak
		This shows the necessity to change the method used when this no longer works and explains why we will switch to expansion of integrand+stationary phase.
	\end{Remark}
	\subsubsection{Link with the $S_{\delta,\nu}$ space}\label{link Sdeltanu}
	The only thing left to do is to see that this propagated state at time $n_st_0$ fits in the $S_{\delta,\nu}$ setting (see the definition \ref{def S_delta,nu}) so that we can use the second method to propagate even further.
	\begin{Lemma}\label{1st_methodK}
		Set $\delta_{(\text{ini})}=1/2-t_0(\nu_{\text{min}}^\perp-\epsilon_2/3) \epsilon_s$ and let $\rho$ be in a $h^{\tau}$ neighborhood of $K^\delta$ and $n_s=\epsilon_s |\log(h)|$ (with $0<\epsilon_s<\tau/(6t_0\lambda_{\text{max}})$).
		Then, for all $2j+k+l<2N$, there exists a family $\left(u_\gamma\right)_{|\gamma|\leq 3k+2l}\in S^{-\infty}_{\delta_{(\text{ini})}}\left(\R^{d_\perp}\right)$ and $\kappa^\text{c}\in Sp(2d_{\scalerel*{\parallel}{\perp}})$ such that 
		\begin{align*} 
			\mathcal{M}\left(d_\rho [\kappa^{n_s,0} \Phi^{n_st_0} (\kappa_{\alpha_0})^{-1}]\right) \Lambda_h\bigl[P_{n_s}^{(k,l)}\Psi_0\bigr](x,y)=\sum_{\substack{\gamma\in \N^{d_{\scalerel*{\parallel}{\perp}}}, \\|\gamma|\leq N}}  u_\gamma\left(y\right)  &\mathcal{M}_{\text{c}}\left(\kappa^{\text{c},(\text{ini})}\right)\Lambda_h\bigl[(\ \boldsymbol{\cdot}\ )^\gamma \Psi_0^{\text{c}}(\ \boldsymbol{\cdot}\ )\bigr](x)
			\\&+O_{L^2}(h^{N(\tau-6t_0\lambda_{\text{max}}\epsilon_s)}),
		\end{align*}
		with $\left\|\kappa^{\text{c},(\text{ini})} \right\|\leq Ch^{-\nu}$ for $\nu=\epsilon_s/3$.
	\end{Lemma}
	\begin{Remark}
		Here we write the result at a macroscopic scale as it is the one used to define the admissible states for our next method.
	\end{Remark}
	\begin{proof} 
		\begin{itemize}
			\item \textbf{Case where $\rho\in K^\delta$:}
			\medbreak
			We will make use of the special form of $d_\rho \bigl[\kappa^{i,0} \Phi^{it_0} \left(\kappa_{\alpha_0}\right)^{-1}\bigr]$, let us introduce the notation  
			\begin{equation*} 
				\boldsymbol{\Phi^{(i)}}\coloneq\kappa^{i,0} \Phi^{it_0} \left(\kappa_{\alpha_0}\right)^{-1}.
			\end{equation*}
			\medbreak
			The first step is to decompose $P_{n_s}^{(j,k,l)}$ into monomials: we need to understand for $\gamma\in \N^d$,
			\begin{equation*} 
				v_\gamma\coloneq\mathcal{M}\left(d_\rho \bigl[\kappa^{i,0} \Phi^{it_0} (\kappa_{\alpha_0})^{-1}\bigr]\right) \Lambda_h\bigl[c_{\gamma}X^\gamma \Psi_0(X)\bigr].
			\end{equation*}
			From lemma \ref{propag excited}, we know that it can be written as
			\begin{equation*} 
				v_\gamma(x,y)=h^{-d/4} \left|\det \Im(\Gamma) \right|^{1/4} h_\gamma\left(h^{-1/2}\Im(\Gamma)^{1/2}(x,y)\right) e^{\frac{i}{h}\langle (x,y), \Gamma (x,y)\rangle }.
			\end{equation*}
			We will now use properties of our charts to obtain information on $\Gamma$:
			\medbreak
			With our choice of coordinates in the first chart $\alpha,$ we see that we can write 
			\begin{equation} \label{block struct}
				d_\rho \Phi^{(1)} =\begin{pmatrix}
					A & 0 & E & 0 
					\\ 0 & B & 0 & G
					\\ F & 0 & C & 0
					\\ 0 & H & 0 & D
				\end{pmatrix}.
			\end{equation}
			Indeed, from the properties of the coordinates, we know that the space $E_{\scalerel*{\parallel}{\perp}}\coloneq\{(x,0,\xi,0),$
			$(x,\xi)\in \R^{2d_{\scalerel*{\parallel}{\perp}}}\}\subset T_\rho \R^{2d}$ is preserved by $d_\rho \Phi^{(1)}$ ($K^\delta$ is sent onto itself). But we also know that $d_\rho \Phi^{(1)}$ is symplectic, hence verifies:
			\begin{equation*} 
				d_\rho \Phi^{(1)}(E_{\scalerel*{\parallel}{\perp}}^{\perp_\omega})=(d_\rho \Phi^{(1)}(E_{\scalerel*{\parallel}{\perp}}))^{\perp_\omega},
			\end{equation*}
			but as $E_{\scalerel*{\parallel}{\perp}}^{\perp_\omega}=\{ (0,y,0,\eta),(y,\eta)\in \R^{2d_\perp}\}\subset T_\rho \R^{2d}$ this implies the block-diagonal structure of (\ref{block struct}).
			\medbreak
			Then, as the other coordinates are adapted to the point $\rho$, we compute (using (\ref{dF0})):
			\begin{equation*}
				d_\rho \Phi^{(n_s)}=\begin{pmatrix}
					A_{n_s} & 0 & E_{n_s} & 0
					\\ 0 & B_{n_s} & 0 & B_{n_s-1}G 
					\\ F_{n_s} & 0 & C_{n_s} & 0
					\\ 0 & D_{n_s-1}H & 0 & D_{n_s}
				\end{pmatrix}.
			\end{equation*}
			Hence, if we compute the associated $\Gamma,$ we obtain:
			\begin{align*}
				\Gamma&=\begin{pmatrix}
					\left(F_{n_s}+iC_{n_s}\right)\left(A_{n_s}+iE_{n_s}\right)^{-1} & 0
					\\ 0 & D_{n_s-1}(H+iD)(B+iG)^{-1}B_{n_s-1}^{-1}
				\end{pmatrix}
				\\&=\begin{pmatrix}
					\Gamma_\text{c} & 0
					\\ 0 & \Gamma_\text{hyp}
				\end{pmatrix},
			\end{align*}
			which yields:
			\begin{equation*} 
				v_\gamma(x,y)=h^{-d/4} h_\gamma(x,y) e^{\frac{i}{h}\langle y, \Gamma_\text{hyp} y\rangle } e^{\frac{i}{h}\langle x, \Gamma_\text{c} x\rangle }.
			\end{equation*}
			From the properties of the flow, we have some specific information on the quantities involved there:
			\begin{itemize}
				\item $h_\gamma$ has separated variables (i.e. is some $h_\gamma^x(x)h_\gamma^y(y)$) because $X^\gamma$ does and $d_\rho F^{(n_s)}$ also have the tensorial product structure.
				\item $\left|\Gamma_\text{hyp}\right|\leq C e^{-2n_st_0 \nu_{\text{min}}^\perp}=Ch^{\nu_{\text{min}}^\perp t_0 \epsilon_s}$ thanks to the definition of $t_0$ and $n_s$ ($G$ and $H$ does not have any influence on the estimate since they are only one term in a infinite product).
				\item we recognize in $\Gamma_\text{c}$ the $\Gamma$ coming from the parallel metaplectic operator $\mathcal{M}_c(\left(\kappa^\text{c}\right)$, which is independent of $y$.
			\end{itemize}
			Combining those facts, we see that we can write 
			\begin{align*}
				v&_\gamma(x,y)= 
				\\& \underbrace{h^{-d_\perp/4}\left|\det \Im(\Gamma_\text{hyp}) \right|^{1/4} h_{\gamma}^{\text{hyp}}\left(\Im(\Gamma_\text{hyp})^{1/2}h^{-1/2} y\right) e^{\frac{i}{h}\langle y, \Gamma_\text{hyp} y \rangle }}_{u_{\gamma_x}(y)} h^{-d_{\scalerel*{\parallel}{\perp}}/4}\mathcal{M}_{\text{c}} (\kappa)\Lambda_h\bigl[x^{\gamma_x} \Psi_0\bigr](x),
			\end{align*}
			where $h_{\gamma}^{\text{hyp}}$ is a polynomial of degree $\gamma_y$ and coefficients uniformly bounded with respect to $h$.
			\medbreak
			To conclude the proof, we have to show that $u_{\gamma_x}\in S^{-\infty}_{\delta_{(\text{ini})}}$. Notice first that by changing the variable, we have
			\begin{equation*} 
				\|u_\gamma\|_{\infty}\leq C \text{ independent of } h.
			\end{equation*}
			Then, notice that differentiating in a $y$ variable (say the $i$-th) has two effects: when the derivative hits $h_{\gamma}^{\text{hyp}}$, we obtain 
			\begin{equation*}
				h^{-d_{\scalerel*{\parallel}{\perp}}/4}\left|\det \Im(\Gamma_\text{hyp}) \right|^{1/4} \langle \nabla h_{\gamma}^{\text{hyp}}\left(\Im(\Gamma_\text{hyp})^{1/2}h^{-1/2} y\right), h^{-1/2}\Im(\Gamma_\text{hyp})^{1/2} e_i\rangle e^{\frac{i}{h}\langle y, \Gamma_\text{hyp} y \rangle},
			\end{equation*}
			hence a bound with $Ch^{-1/2}|\Im(\Gamma_\text{hyp})^{1/2}|$; otherwise the derivative can hit the Gaussian which adds a $2\langle y, \Gamma_\text{hyp} e_i\rangle/h$ factor, after change of variable it can be bounded by $Ch^{-1/2}|\Gamma_\text{hyp}\Im(\Gamma_\text{hyp})^{-1/2}|$.
			\medbreak
			As a consequence, using (\ref{GammaImGamma}) and bounds on the matrix coefficients, we obtain that
			\begin{equation*} 
				\|\partial^\alpha u_\gamma\|_{\infty}\leq C h^{-|\alpha|(1/2-(\nu_{\text{min}}^\perp-\epsilon_2/3) t_0 \epsilon_s)}=C h^{-|\alpha|\delta_{(\text{ini})}}.
			\end{equation*}

			\item \textbf{General case, $\rho$ in a $h^{\tau}$ neighborhood of $K^\delta$:}
			\medbreak
			In the general case, $\rho$ is at distance at most $h^{\tau}$ from a point $\tilde{\rho}$ in $K^\delta$ for which the charts are adapted. Just as in the previous case, we use lemma \ref{propag excited} to write
			\begin{equation*} 
				v_\gamma(x,y)=h^{-d/4} \left|\det \Im(\Gamma) \right|^{1/4} h_\gamma\left(h^{-1/2}\Im(\Gamma)^{1/2}(x,y)\right) e^{\frac{i}{h}\langle (x,y), \Gamma (x,y)\rangle },
			\end{equation*}
			and study what can be said about $\Gamma$.
			\medbreak
			At the iterates of $\tilde{\rho}$, for which the charts are tailored, we get as the previous case:
			\begin{equation*} 
				d \Phi^{(n_s)}_{\tilde{\rho}}=\begin{pmatrix}
					A_{n_s} & 0 & E_{n_s} & 0
					\\ 0 & B_{n_s} & 0 & B_{n_s-1}G 
					\\ F_{n_s} & 0 & C_{n_s} & 0
					\\ 0 & D_{n_s-1}H & 0 & D_{n_s},
				\end{pmatrix},
			\end{equation*}
			then using the comparison of the linear dynamics established in the classical dynamics lemma \ref{compare_linear}
			\begin{equation*} 
				\left\|d_{\tilde{\rho}}\Phi^{(n_s)}-d_\rho \Phi^{(n_s)}\right\|= O\left(h^{\tau-2\epsilon_st_0\lambda_{\text{max}}}\right).
			\end{equation*}
			\medbreak
			Hence, if we compute the $\Gamma$ associated with $d_\rho \Phi^{(n)}$ we obtain:
			\begin{equation*} 
				\Gamma_\rho= \begin{pmatrix}
					\Gamma_\text{c}(\tilde{\rho})+ E_1 & E_2
					\\ E_2^T & \Gamma_\text{hyp}(\tilde{\rho})+E_3
				\end{pmatrix}\coloneq\begin{pmatrix} \Gamma_\text{c}(\tilde{\rho}) &0 \\ 0 & \Gamma_\text{hyp}(\tilde{\rho}) \end{pmatrix} +E,
			\end{equation*}
			with $\|E_p\|\leq  C h^{\tau-4t_0\lambda_{\text{max}} \epsilon_s}$.
			\medbreak
			As a consequence, we can write that
			\begin{equation*}
				v_\gamma(x,y)=h^{-d/4} \left|\det \Im(\Gamma) \right|^{1/4} h_\gamma\left(h^{-1/2}\Im(\Gamma)^{1/2}(x,y)\right) e^{\frac{i}{h}\langle x, \Gamma_\text{c}(\tilde{\rho}) x\rangle} e^{\frac{i}{h}\langle y, \Gamma_\text{hyp}(\tilde{\rho}) y\rangle} e^{\frac{i}{h}\langle (x,y), E (x,y)\rangle }.
			\end{equation*}
			
			\medbreak
			As $E$ is small, we should do a Taylor expansion of the $e^{\frac{i}{h}\langle (x,y), E (x,y)\rangle }$ factor to have a description that fits in the framework of lemma \ref{1st_methodK}. Let us write
			\begin{equation} \label{Dl exp E}
				e^{\frac{i}{h}\langle (x,y), E (x,y)\rangle}=\sum_{k=0}^N \frac{(i\langle (x,y), E (x,y)\rangle)^k}{h^{k}k!}+R_N(x,y),
			\end{equation}
			with 
			\begin{equation*} 
				|R_N|\leq \frac{|\langle (x,y), E (x,y)\rangle |^N}{h^NN!} \exp\left(-\frac{\langle (x,y), \Im E (x,y)\rangle}{h}\right).
			\end{equation*}
			\textbf{Reordering the terms to obtain an expansion as in lemma \ref{1st_methodK}:}
			\medbreak
			Let us explain how to reorder the factors to obtain the expansion of lemma \ref{1st_methodK}. For now, we have the following:
			\begin{equation}\label{Q+lot}
				v_\gamma(x,y)=h^{-d/4} \left|\det \Im(\Gamma) \right|^{1/4} Q(x,y)e^{\frac{i}{h}\langle y, \Gamma_\text{hyp}(\tilde{\rho}) y\rangle}e^{\frac{i}{h}\langle x, \Gamma_\text{c}(\tilde{\rho}) x\rangle} + \text{l.o.t.},
			\end{equation}
			with $Q$ a polynomial in $(x,y)$.
			\medbreak 
			From lemma \ref{propag excited}, we know that each $\tilde{\gamma}_x\in \N^{d_{\scalerel*{\parallel}{\perp}}},$
			\begin{equation*}
				\mathcal{M}_{\text{c}}\left(\kappa^{\text{c},(\text{ini})}(\tilde{\rho})\right)\Lambda_h\bigl[(\ \boldsymbol{\cdot}\ )^\gamma \Psi_0^{\text{c}}(\ \boldsymbol{\cdot}\ )\bigr](x)=h^{-d_\perp/4}\left|\det \Im(\Gamma_\text{c}) \right|^{1/4} h_{\tilde{\gamma}_x}\left(h^{-1/2}\Im(\Gamma_{\text{c}})^{1/2}x\right) e^{\frac{i}{h}\langle x, \Gamma_{\text{c}} x\rangle },
			\end{equation*}
			where $h_{\tilde{\gamma}_x}$ is of degree $\tilde{\gamma}_x$.
			\medbreak
			We notice that the $(h_{\tilde{\gamma}_x})_{\tilde{\gamma}_x}$ form a basis of $\mathbb{C}[X_1,\dots,X_{d_{\scalerel*{\parallel}{\perp}}}]$. From this observation, we see that we can decompose 
			\begin{equation*} 
				Q(x,y)\in C[X_1,\dots,X_{d_{\scalerel*{\parallel}{\perp}}},Y_1,\dots,Y_{d_\perp}]\simeq C[Y_1,\dots,Y_{d_\perp}][X_1,\dots,X_{d_{\scalerel*{\parallel}{\perp}}}]
			\end{equation*} 
			with this basis:
			\begin{equation*}
				Q(x,y)=\sum_{|\tilde{\gamma}_x|\leq \deg_xQ} Q_{\tilde{\gamma}_x}(y)h_{\tilde{\gamma}_x}(x).
			\end{equation*}
			Plugging this in (\ref{Q+lot}), gives a description as in lemma \ref{1st_methodK} provided we can show that for every $\tilde{\gamma}_x$, $u_{\tilde{\gamma}_x}\coloneq Q_{\tilde{\gamma}_x}e^{\frac{i}{h}\langle y, \Gamma_\text{hyp}(\tilde{\rho}) y\rangle}$ is in $S^{-\infty}_{\delta_{(\text{ini})}}$ and that the remainder is indeed small.
			\medbreak
			\textbf{The $S_\delta$ estimate:}
			\medbreak
			Let us start by the $S^{-\infty}_{\delta_{(\text{ini})}}$ estimate. We start by doing a change of variables 
			\begin{align}\label{chg var}
				\begin{split}
					x^{\text{new}}&=\Im \Gamma_{\text{c}}(\tilde{\rho})^{1/2}x
					\\ y^{\text{new}}&=\Im \Gamma_{\text{hyp}}(\tilde{\rho})^{1/2}y.
				\end{split}
			\end{align} 
			We can see that the processes of doing this change of variable and doing the decomposition in polynomial as above are commutative. But by doing this change of variables, we are led to consider
			\begin{equation*}
				v_\gamma(\Im \Gamma_{\text{c}}(\tilde{\rho})^{-1/2}x^{\text{new}}, \Im \Gamma_{\text{hyp}}(\tilde{\rho})^{-1/2}y^{\text{new}}).
			\end{equation*}
			Notice that this expression involves the coefficients of two new matrices:
			\begin{itemize} 
				\item A new matrix $M$ (or more precisely its square root) that appears in the polynomials $h_\gamma$, it verifies
				\begin{equation*}
					\begin{pmatrix}
						\Im \Gamma_\text{c}+\Im E_1 & \Im E_2
						\\ \Im E_2^T & \Im \Gamma_\text{hyp} +\Im E_3
					\end{pmatrix} = \begin{pmatrix} \Im \Gamma_\text{c}^{1/2} & 0
						\\ 0 & \Im \Gamma_\text{hyp}^{1/2}\end{pmatrix}M
					\begin{pmatrix} \Im \Gamma_\text{c}^{1/2} & 0
						\\ 0 & \Im \Gamma_\text{hyp}^{1/2}\end{pmatrix}.
				\end{equation*}
				We can compute that 
				\begin{equation}\label{Fi}
					M\coloneq\begin{pmatrix}
						I_{d_{\scalerel*{\parallel}{\perp}}}+F_1 & F_2
						\\ F_2^T & I_{d\perp}+F_3
					\end{pmatrix},
				\end{equation}
				with $F_i=\Im \Gamma_\text{hyp}^{-1/2} \Im E_i \Im \Gamma_\text{c}^{-1/2}$ and that $\|F_i\|\leq C h^{\tau-6t_0\lambda_{\text{max}} \epsilon_s}$. Hence this matrix is bounded independently of $h$.
				\item And a very similar matrix that comes from the expansion (\ref{Dl exp E}):
				\begin{equation*}
					E_\Gamma\coloneq\begin{pmatrix} \Im \Gamma_\text{c}^{-1/2} & 0
						\\ 0 & \Im \Gamma_\text{hyp}^{-1/2}\end{pmatrix}E
					\begin{pmatrix} \Im \Gamma_\text{c}^{-1/2} & 0
						\\ 0 & \Im \Gamma_\text{hyp}^{-1/2}\end{pmatrix},
				\end{equation*}
				which is also bounded independently of $h$.
			\end{itemize}
			Consequently, we can conclude by undoing the change of variable that 
			\begin{equation*} 
				\|\partial^{\alpha} u_{\tilde{\gamma}_x}\|_{\infty}\leq C(h^{-1/2}|\Gamma(\tilde{\rho})|)^{|\alpha|}\leq C h^{-|\alpha|\delta_{(\text{ini})}}.
			\end{equation*}
			\textbf{Control of the remainder:}
			\medbreak
			We are now left with showing that $R_N$ can indeed be considered as a remainder by estimating
			\begin{equation}\label{mathcal N}
				\mathcal{N}\coloneq\|h^{-d/4}|\det \Gamma_\rho|^{1/4} R_N e^{\frac{i}{h}\langle x ,\Gamma_\text{c}(\tilde{\rho}) x \rangle } e^{-\frac{1}{h}\langle y ,\Gamma_\text{c}(\tilde{\rho}) y \rangle }\|_{L^2(\d x\d y)},
			\end{equation}
			the general estimate where we add polynomials in $x,y$ is similar.
			\medbreak
			We start doing the same change of variable (\ref{chg var}) in the integral defining the norm $\mathcal{N}$ in (\ref{mathcal N}) which gives:
			\begin{equation*} 
				\mathcal{N}^2\leq C\int |\det M|^{1/2} |\langle (x,y), (M-I) (x,y)\rangle |^{2N}e^{- (x,y)\cdot M(x,y)} \d x \d y.
			\end{equation*}
			We then reduce the quadratic form associated with $M$ by noticing that 
			\begin{multline*}
				M =
				\\+ \begin{pmatrix} I_{d_{\scalerel*{\parallel}{\perp}}} & 0 \\ F_2^T(I_{d_{\scalerel*{\parallel}{\perp}}}+F_1)^{-1} & I_{\perp} \end{pmatrix} \begin{pmatrix} I_{d_{\scalerel*{\parallel}{\perp}}}+F_1 & 0 \\ 0 & I_{d_\perp}+F_4-F_2^T(I_{d_{\scalerel*{\parallel}{\perp}}}+F_1)^{-1}F_2 \end{pmatrix} \begin{pmatrix} I_{d_{\scalerel*{\parallel}{\perp}}} & (I_{d_{\scalerel*{\parallel}{\perp}}}+F_1)^{-1}F_2 \\ 0 & I_{d_\perp} \end{pmatrix},
			\end{multline*}
			and as $I_{d_\perp}+F_1$, $I_{d_\perp}+F_4-F_2^T(I_{d_{\scalerel*{\parallel}{\perp}}}+F_1)^{-1}F_2$ are invertible, 
			\begin{equation*} 
				\|(I_{d_\perp}+F_1)^{-1}\|,\|I_{d_\perp}+F_4-F_2^T(I_{d_{\scalerel*{\parallel}{\perp}}}+F_1)^{-1}F_2\|\leq 2,
			\end{equation*}
			for $h$ small enough since $\|F_i\|\leq C h^{\tau-6t_0\lambda_{\text{max}} \epsilon_s}$ and $\upsilon\coloneq\tau-6t_0\lambda_{\text{max}} \epsilon_s>0$ by assumption on $\epsilon_s$.
			\medbreak
			After a new change of variables given by
			\begin{align*}
				\begin{pmatrix} x^{\text{new}} \\ y^{\text{new}} \end{pmatrix} &=\begin{pmatrix} \left[I_{d_{\scalerel*{\parallel}{\perp}}}+F_1\right]^{1/2} & 0 \\ 0 & \left[I_{d_\perp}+F_4-F_2^T(I_{d_{\scalerel*{\parallel}{\perp}}}+F_1)^{-1}F_2\right]^{1/2} \end{pmatrix} \begin{pmatrix} I_{d_{\scalerel*{\parallel}{\perp}}} & (I_{d_{\scalerel*{\parallel}{\perp}}}+F_1)^{-1}F_2 \\ 0 & I_{d_\perp} \end{pmatrix} \begin{pmatrix} x^{\text{old}} \\ y^{\text{old}}\end{pmatrix}
				\\&\coloneq P\begin{pmatrix} x^{\text{old}} \\ y^{\text{old}}\end{pmatrix},
			\end{align*}
			we obtain (using that the $\|P^{-1}\|\leq C$):
			\begin{align*} 
				\mathcal{N}^2&\leq C\int |\langle (x,y), P^{-1}E_\Gamma P^{-1}(x,y)\rangle |^{2N} e^{-\|x\|^2-\|y\|^2}\d x \d y
				\\&\leq C \|E_\Gamma\|^{2N}
				\\&\leq C h^{2N\upsilon}.
			\end{align*}
		\end{itemize}
	\end{proof}
	\subsection{Successive application of expansion of the integrand+stationary phase}\label{app_2nd}
	We will follow the same steps to described the iterated state as in the previous section.
	\medbreak
	From the previous part, we see that each term of the expansion can be written as 
	\begin{equation*} 
		\hat{T}_{\mathcal{I}_{n_s}} \Lambda_h u^{(\text{ini})},
	\end{equation*}
	with $u^{(\text{ini})}\in S_{\delta,\nu}$, $\delta=\delta_{(\text{ini})}=1/2-t_0\nu_{\text{min}}^\perp \epsilon_s$, $\nu=\epsilon_s/3$ and 
	\begin{equation*} 
		\mathcal{I}^{n_s}=\bigl\{\left(q^{n_s}_x,y^{n_s},p^{n_s}_x,p^{n_s}_y\right), y^{n_s}\in D_{\epsilon_1} \bigr\}.
	\end{equation*}
	We have used the fact that for this specific $\mathcal{I}^{n_s}$, which is a translation of $\mathcal{I}_m$, we have 
	\begin{equation*} 
		\hat{T}_{\mathcal{I}_{n_s}}=\hat{T}(q^{n_s}_x,0,p^{n_s}_x,p^{n_s}_y),
	\end{equation*}
	and the $S_\delta$ class is invariant by the translation by $q^{n_s}_y$ in $y$.
	\medbreak
	Let us write this $u^{(\text{ini})}$ as 
	\begin{equation}\label{u_ini}
		u^{(\text{ini})}=\sum_{k+l<2N}h^{k/2+l/2}\mathcal{M}_{\text{c},1}(\kappa^{\text{c}(\text{ini})})[Q_{n_s}^{(k,l)}(y,\boldsymbol{\cdot}) \Psi_0](x),
	\end{equation}
	with $Q_{n_s}^{(j,k,l)}(y,x)$ a polynomial in the $x$ variable of degree $3k+l$ whose coefficients are $S_\delta\left(\jp{y}^{-\infty}\right)$ in $y$ so that each term of this expansion is in $S_{\delta,\nu}$ (we have simply reconstituted the expansion by summing the result given by lemma \ref{1st_methodK} over $k$ and $l$).
	\medbreak
	We define the sequence of manifolds $\mathcal{I}_p\coloneq\kappa^{(p)}\Phi^{(p-n_s)t_0}(\kappa^{(n_s)})^{-1}(\mathcal{I}^{n_s})$ (i.e. the iterates of $\mathcal{I}_{n_s}$ read in the charts of $\tilde{\rho}_{p}$). The hypothesis on these manifolds required for the method \ref{section2nd} are verified in the appendix \ref{L_j}, those estimates are similar to the Inclination Lemma which is a result of hyperbolic dynamics stating that manifolds which are transverse enough to the stable manifold remains in the same unstable cone.
	\medbreak
	Let us introduce the coordinates of $\mathcal{I}_p$ in the chart:
	\begin{equation}\label{coord Ip}
		\mathcal{I}_p=\{(\overline{x^p}(y),y,\overline{\xi^p}(y),\overline{\eta^p}(y), y\in D_{\epsilon_1})\}.
	\end{equation}
	\medbreak
	We note that each $T_p$ defined in (\ref{def Tk}) is a Fourier integral operator as described in the beginning of section \ref{context}, see (\ref{propagator}).
	\medbreak
	For each $p\in \llbracket n_s+1,n \rrbracket$, we apply the main result of section \ref{section2nd}, proposition \ref{Qjkl}, to $T_p$ for an initial state given by an element of $S_{\delta,\nu}$ defined as 
	\begin{equation*} 
		u= \mathcal{M}_{\text{c},1}\left(\kappa^{\text{c},(p-1)}(y)\right)\Bigl[P(y,x)\Psi_0\Bigr]\in S_{\delta,\nu},
	\end{equation*}
	with 
	\begin{equation}\label{kappa c p}
		\kappa^{\text{c},(p-1)}(y^{p-1})\coloneq d^c \mathcal{F}_{p-1}\circ\dots\circ d^c\mathcal{F}_{n_s+1} \text{ where }\mathcal{F}_k\coloneq\left(\kappa_{\mathcal{I}_k}\right)^{-1} \circ \kappa_{\tilde{\rho_k}}\circ \Phi^{t_0}\circ \left(\kappa_{\tilde{\rho_{k-1}}}\right)^{-1} \circ\kappa_{\mathcal{I}_{k-1}},
	\end{equation}
	in the notations of section \ref{section2nd} and $\kappa^{\text{c},(\text{ini})}$ defined in lemma \ref{1st_methodK}.
	\medbreak
	This proposition \ref{Qjkl} gives us a new family of polynomials in $x$, $(Q_p^{(k,l)}(P))$ verifying:
	\begin{itemize}
		\item For every $N\in \N$,
		\begin{equation*} 
			\hat{T}_{\mathcal{I}_{p}}^* T_p \hat{T}_{\mathcal{I}_{p-1}} u\left(x^{p},y^{p}\right)=\mathcal{M}\left(d^c\mathcal{F}_p\right)\Bigl[\sum_{k,l} Q_p^{(k,l)}(P) \Psi_0\Bigr]+R_N,
		\end{equation*}
		with 
		\begin{equation}\label{norm R_N}
			\|R_N\|_{L^2}\leq C h^N \sup_{2l+k=2N}\sup_{\substack{|\alpha_1|+|\alpha_2|\leq 2l , \\ |\beta_1|+|\beta_2|=k}} \left\|N_\infty\left(\partial^{\alpha_1+\beta_1}_y P_y\right)\right\|_{L^2(\d y)} \|\kappa^\text{c}\|^{|\beta_1|+2|\alpha_2|+3|\beta_2|+d_{\scalerel*{\parallel}{\perp}}}.
		\end{equation}
		\item $y^{p-1}(y^p)$ defined by the following between the isotropic manifolds $\mathcal{I}_p$ and $\mathcal{I}_{p-1}$
		\begin{equation*}
			F(\overline{x^{p-1}}(y^{p-1}(y^p)),y^{p-1}(y^p),\overline{\xi^{p-1}}(y^{p-1}(y^p)),\overline{\eta^{p-1}}(y^{p-1}(y^p)))=(\overline{x^p}(y^p),y^p,\overline{\xi^p}(y^p),\overline{\eta^p}(y^p)).
		\end{equation*}
		\item $Q_p^{(0,0,0)}(P)= =\left|\det \partial_{y^p} y^{p-1}(y^p)\right|^{1/2} P_{y^{p-1}(y^p)}= \alpha_p(y^p) P_{y^{p-1}(y^p)}$ multiplies by a constant in $x^i$.
		\item $Q_p^{(k,l)}$ is a polynomial of degree $\deg P+3k+4l$ in $x$, the map $P\mapsto Q^{(k,l)}$ is linear, with coefficients controlled by:
		\begin{equation*} 
			\left\|N_\infty\left(Q^{(k,l)}(P)\right)(y)\right\|_{L^2} \leq C\sup_{\substack{|\alpha_1|+|\alpha_2|\leq 2l , \\ |\beta_1|+|\beta_2|=k}} \left\|N_\infty\left(\partial^{\alpha_1+\beta_1}_y P_y\right)\right\|_{L^2(\d y)} \left\|\kappa^\text{c}\right\|^{|\beta_1|+2|\alpha_2|+3|\beta_2|},
		\end{equation*}
		where the $N_\infty$ means the largest coefficient of this polynomial in $x$.
	\end{itemize}
	Then we define recursively: $P_0^{(0,0,0)}=P_\text{ini}$ a polynomial appearing in the expansion (\ref{u_ini}) of $u_{\text{ini}}$ and $P_0^{(k,l)}=0$ otherwise,
	\begin{equation}\label{def poly iteres}
		P_p^{(k,l)}=\sum_{k_1+k_2=k} \sum_{l_1+l_2=l} Q_p^{(k_1,l_1)}(P_{p-1}^{(k_2,_2l)}),
	\end{equation}
	and the remainder term 
	\begin{equation}\label{def restes}
		r_N^{(p)}=T_p\left(r_N^{(p-1)}\right)+\sum_{k+l<2N} h^{k/2+l} R_p^{(2N-k-l)}\left(P_{p-1}^{(k,l)}\right).
	\end{equation}
	
	\begin{Lemma}
		With these notations, we have:
		\begin{equation*} 
			u_p=T_{n_s+p}\dots T_{n_s+1}u_0=\sum_{k+l<2N} h^{k/2+l} e^{\frac{i}{h} \phi_p(\sqrt{h} y)}u_p^{(k,l)}+r^{(p)}_N,
		\end{equation*}
		where 
		\begin{equation*} 
			u_p^{(k,l)}=\hat{T}\left(\rho^p\right)\mathcal{M}_\text{c} \left(\kappa^{\text{c},(p)}(y)\circ \kappa^{\text{c},(\text{ini})}\right)\Lambda_{h,x}\left[P_p^{(k,l)}(y,x)\Psi_0\right].
		\end{equation*}
	\end{Lemma}
	We will now estimate the different quantities involved in this lemma.
	\medbreak
	\subsubsection{The analysis of $\kappa^{\text{c},(p)}$}\label{estim_metap}
	Let us start with $\kappa^{\text{c},(p)}$ and the following lemma
	\begin{Lemma}
		There exists a $C>0$ independent of $h$ such that for any $p\in \llbracket n_s, n\rrbracket$,
		\begin{equation*}
			\|\kappa^{\text{c},(p)}\|_{L^\infty(\d y)}\leq C e^{p(\lambda_c+\epsilon_2/3)t_0}.
		\end{equation*}
	\end{Lemma}
	\begin{proof}
		By definition,
		\begin{equation*}
			\kappa^{\text{c},(p)}(y^{p})\coloneq d^c \mathcal{F}_{p}(y^p)\circ\dots\circ d^c\mathcal{F}_{n_s+1}(y^{n_s+1}),
		\end{equation*}
		with the $\mathcal{F}_{k}$'s defined in (\ref{kappa c p}) and $y^k$ defined by the equation between $\mathcal{I}_p$ and $\mathcal{I}_k$:
		\begin{equation*} 
			(\overline{x^{k-1}}(y^{k-1}),y^{k-1},\overline{\xi^{k-1}}(y^{k-1}),\overline{\eta^{k-1}}(y^{k-1})))=(F_p\circ\dots F_k)^{-1}(\overline{x^p}(y^p),y^p, \overline{\xi^p}(y^p),\overline{\eta^p}(y^p)).
		\end{equation*}
		\medbreak
		Then, using lemma (\ref{compare dcF}) we have 
		\begin{equation} \label{dcFk}
			\|d^c\mathcal{F}_k(y^k)-dF_k\vert_\text{c}(0)\|\leq C \|(\overline{x^{k-1}}(y^{k-1}),y^{k-1},\overline{\xi^{k-1}}(y^{k-1}),\overline{\eta^{k-1}}(y^{k-1}))\|,
		\end{equation}
		with $F_k=\kappa_{\tilde{\rho_1}} \Phi^{t_0} \kappa_{\tilde{\rho_0}}$, $dF_k\vert_\text{c}(0)$ the central block of its differential, see (\ref{dF0}).This then yields 
		\begin{align}\label{dcF acc}
			\begin{split}
				\|\kappa^{\text{c},(p)}(y^{p})\| &\leq \prod_{k=n_s+1}^p \|d F_k\vert_\text{c}(0)\|\left(1+C\|(\overline{x^{k-1}}(y^{k-1}),y^{k-1},\overline{\xi^{k-1}}(y^{k-1}),\overline{\eta^{k-1}}(y^{k-1}))\|\right)
				\\ &\leq C e^{p(\lambda_c+\epsilon_2/3)t_0} \prod_{k=n_s+1}^p\left(1+C\|(\overline{x^{k-1}}(y^{k-1}),y^{k-1},\overline{\xi^{k-1}}(y^{k-1}),\overline{\eta^{k-1}}(y^{k-1}))\|\right),
			\end{split}
		\end{align}
		where we used our choice of $t_0$ to compare $\|d F_k\vert_\text{c}(0)\|$ with $\lambda_c$, see (\ref{lambda max et t0}).
		\medbreak
		Finally, as the points $(\overline{x^{k}}(y^{k}),y^{k},\overline{\xi^{k}}(y^{k}),\overline{\eta^{k}}(y^{k}))$ corresponds to the backwards iterates of 
		\begin{equation*} 
			(\overline{x^p}(y^p),y^p, \overline{\xi^p}(y^p),\overline{\eta^p}(y^p))
		\end{equation*}
		under the flow (read in charts) $\Phi^{t_0}$, we can apply bound their norms using the section on classical dynamics, or more precisely lemma \ref{proof_dyn_class} (with remark \ref{dyn class pas rho}):
		\medbreak
		We want to define $\underline{\rho}$ by 
		\begin{equation*}
			\Phi^{(p-n_s)t_0}(\underline{\rho})\coloneq(\kappa^{(p)})^{-1}(\overline{x^p}(y^p),y^p, \overline{\xi^p}(y^p),\overline{\eta^p}(y^p))\in K^{\delta}(\epsilon_1).
		\end{equation*}
		Then, we may assume that all the iterates $\Phi^{kt_0}(\underline{\rho})\in K^\delta(\epsilon_1)$ for $0\leq k\leq p-n_s$, otherwise we would get a point outside of charts which would give us a $O(h^\infty)$ result. 
		\medbreak
		We finally need information on the initial coordinates of $\underline{\rho}$ in the chart of $\tilde{\rho}^{n_s}$ (i.e. replacing the $h^\tau$ of lemma \ref{proof_dyn_class} by some other estimate). To do this, we simply notice that $\kappa^{(n_s)}(\underline{\rho})\in \mathcal{I}_{n_s}$ and recall that 
		\begin{equation*} 
			\mathcal{I}^{n_s}=\bigl\{\left(q^{n_s}_x,y^{n_s},p^{n_s}_x,p^{n_s}_y\right), y^{n_s}\in D_{\epsilon_1} \bigr\},
		\end{equation*}
		where $q^{n_s}_x,p^{n_s}_x,p^{n_s}_y$ are coordinates of $\rho^{n_s}$.
		\medbreak
		Hence, the $x,\xi,\eta$ coordinates of $\underline{\rho}$ in the charts of $\tilde{\rho}^{n_s}$ are bounded using the coordinates of $\rho^{n_s}$ in the chart of $\tilde{\rho}^{n_s}$. But this time, the coordinates can be controlled using lemma \ref{proof_dyn_class} as the time $n_s$ is not big enough for $\rho^{n_s}$ to be outside of the $\epsilon_1$ neighborhood of $K^\delta$ (we have a crude estimation of this distance by $Ch^{\tau} e^{(\lambda_{\text{max}}+\epsilon_2/3)n_s}$).
		\medbreak
		As a consequence, the assumptions to apply lemma \ref{proof_dyn_class} are verified and we have the estimate 
		\begin{equation*} 
			\|\overline{x^k}(y^k)\|,\|\overline{\xi^k}(y^k)\|,\|\overline{\eta^k}(y^k)\|\leq (1+\epsilon_1) e^{(\lambda^{\text{c}}+2\epsilon_2/3)kt_0} h^{\tau},
		\end{equation*}
		\begin{equation*}
			\|y^k\|\leq (1+\epsilon_1)\max\bigl( e^{-(p-n_s-k)t_0(\nu_{\text{min}}^\perp-2\epsilon_2/3)},e^{(\lambda^{\text{c}}+2\epsilon_2/3)kt_0} h^{\tau}\bigr),
		\end{equation*}
		which let us show that 
		\begin{equation*} 
			\sum_{k=n_s+1}^{p} \|(\overline{x^{k-1}}(y^{k-1}),y^{k-1},\overline{\xi^{k-1}}(y^{k-1}),\overline{\eta^{k-1}}(y^{k-1}))\|=O(1),
		\end{equation*}
		and then we obtain using (\ref{dcF acc}) that 
		\begin{equation*}
			\|\kappa^{\text{c},(p)}(y^{p})\|\leq C e^{p(\lambda_c+\epsilon_2/3)t_0}.
		\end{equation*}
	\end{proof}
	
	\subsubsection{The analysis of the polynomial $P_p^{(k,l)}$}
	
	\begin{Lemma}
		For all $p\in \llbracket n_s, n\rrbracket$ and $(k,l)$ such that $k+l<2N$, $P_p^{(k,l)}$ is of degree less than $\deg P+3k+2l$.
	\end{Lemma}
	This is obtained by induction on the formula (\ref{def poly iteres}).
	\begin{Lemma}\label{norm_2ndmethod}
		For all $p\in \llbracket n_s,n\rrbracket,$ and $(k,l)$ such that $k+l<2N,$ $\forall \alpha\in \N^{d_\perp},$
		\begin{equation*} 
			\left\|N_\infty\left(\partial^\alpha_y P_p^{(k,l)}\right)\right\|_{L^2}\leq C  (p^{2N}+1) \sup_{\substack{|l_1|+|l_2|\leq 2l , \\ |k_1|+|k_2|=k}} \left(h^{-1/2} e^{-pt_0(\nu_{\text{min}}+\epsilon)}\right)^{|l_1|+|k_1|+|\alpha|} \left(\sigma_{\text{c}}^{(p)}\right)^{|k_1|+2|l_2|+3|k_2|}.
		\end{equation*}
	\end{Lemma}
	
	\begin{proof}
		The proof is similar to the one done in \cite[Lemma 5.6]{Lucas}. We will prove the result by strong induction on $p$.
		\medbreak
		We start with $p=n_s$, by definition of $P_{\text{ini}}$ with (\ref{u_ini}) and lemma \ref{1st_methodK}, we have that 
		\begin{align*}
			\left\|N_\infty\left(\partial^\alpha_y P_{n_s}^{(0,0)}\right)\right\|_{L^2}&\leq C (h^{-1/2} |\Gamma_{\text{hyp}}|^{1/2})^{|\alpha|}
			\\&\leq C  \left(h^{-1/2}e^{-pt_0(\nu_{\text{min}}+\epsilon)}\right)^{|\alpha|}.
		\end{align*}
		Now, let us fix $(k,l)$ such that $k+2l<2N$. By iterating the result defining $P_p^{(k,l)}$, we obtain:
		\begin{equation*} 
			\partial^\alpha_y P_p^{(k,l)}= \sum_{\substack{k_1+\dots+k_p=k\\ l_1+\dots+l_p=l}} \partial^\alpha_y Q_p^{(k_p,l_p)} \circ \dots \circ Q_1^{(k_1,l_1)}\left(P_\text{ini}\right).
		\end{equation*}
		Just as before, the key argument is that for a term of this sum, there are only a finite number of couples $\left(k_s,l_s\right)$ that are not $(0,0)$: $p$ can be large (ie depending on $h$) but $k$ and $l$ are finite. 
		\medbreak
		We want to bound all the terms of this sum by the worse case scenario and then control the number of terms.
		\medbreak 
		Let us denote by $I_{k,l}$ the set of sequences $\left(k_s,l_s\right)$ such that $k_1+\dots+k_p=k$, $l_1+\dots+l_p=l$ and pick such a sequence in $I_{k,l}$. Out of this sequence, most couples are $(0,0)$, we only need to take into account the other ones as $Q_p^{(0,0)}(P)$ is well understood. Let us define $I$ the set of the indexes $s$ such that at least one of the parameters $k,l$ is not $0$. We rename the parameters of $I$: $(k_{s_1},l_{s_1})$ ,\dots,$(,k_{s_m},l_{s_m})$.
		\medbreak 
		We start by writing (using our knowledge of $Q_p^{(0,0)},$ see proposition \ref{Qjkl}):
		\begin{align*}
			Q_p^{(k_p,l_p)}\circ \dots \circ Q_1^{(k_1,l_1)}(P_\text{ini})_{y^p}&=Q_p^{(0,0)}\circ \dots Q_{s_m}^{k_{s_m},l_{s_m}}\circ\dots Q_1^{k_1,l_1}\left(P_\text{ini}\right)_{y^p}
			\\&= \left|\det\left(\frac{\partial y^{s_m}}{\partial y^p}(y^p)\right)\right|^{1/2} Q_{s_m}^{k_{s_m},l_{s_m}}\circ\dots Q_1^{k_1,l_1}\left(P_\text{ini}\right)_{y^{s_{m}}(y^p)}.
		\end{align*}
		Looking at the derivatives in $y$, we can write its coefficients as a product of derivatives of:
		\begin{itemize}
			\item The product of the determinant factors coming from every time we chose the main order term in the expansion (the $Q_s^{(0,0)}$): $\left|\det\left(\frac{\partial y^{s_m}}{\partial y^p}(y^p)\right)\right|^{1/2}$.
			\item Coefficients coming from $Q_{s_m}^{k_{s_m},l_{s_m}}$: derivatives of the amplitude and phase of the FIO taken at point $y$ multiplied by some power of $\left\|\frac{\partial y^{s_m}}{\partial y^p}(y^p)\right\|$ coming from $y^{s_m}(y^p)$.
			\item Coefficients of $P_{s_m-1}\coloneq Q_{s_m-1}^{(k_{s_m-1},l_{s_m-1})}\circ \dots \circ Q_1^{(k_1,l_1)}\left(P_\text{ini}\right)$ evaluated at point $y^{s_m-1}(y^{s_m})$ multiplied by the some power of $\left\|\frac{\partial y^{s_m}}{\partial y^p}(y^p)\right\|$, just like in the second item.
		\end{itemize}
		Now, let us see how we should control each type of terms.
		\medbreak
		For elements of first type, we use the fact that derivatives of a contracting flow have similar size to the flow itself, see lemma \ref{der_contract}. We get
		\begin{equation*} 
			\frac{\left|d^\alpha(\det(\frac{\partial y^p}{\partial y^n}))\right|}{\left|\det(\frac{\partial y^p}{\partial y^n})\right|}=O(p^{|\alpha|})=O(\log(1/h)^{|\alpha|}),
		\end{equation*}
		which implies that 
		\begin{equation}\label{deriv det} 
			\left|\det (\frac{\partial y^p}{\partial y^n})\right|^{-1/2} \left|d^\alpha(\det(\frac{\partial y^p}{\partial y^n}))\right|=O(\log(1/h)^{|\alpha|}).
		\end{equation}
		\medbreak
		Elements of second type of the product will not matter as they are $S_0$: they come from a propagation in time $t_0$, independent of $h$. 
		\medbreak
		For elements of last type, we will use the induction.
		\medbreak
		As we look at the $L^2$ norm of the coefficients, a change of variables $y^p \mapsto y^{s_m}(y^p)$ is needed due to the estimates (\ref{deriv det}) for the first type, we know how to had the corresponding determinant into the estimate. Hence we get the estimate
		\begin{align}\label{d alpha Q}
			\begin{split}
				\mathcal{N}:&=\left\|N_\infty\left(\partial^\alpha_{y^p} Q_p^{(k_p,l_p)} \circ \dots \circ Q_1^{(k_1,l_1)}(P_\text{ini})_{y^p}\right)\right\|_{L^2(\d y^p)} 
				\\&\leq C \sup_{\alpha_1+\alpha_2=\alpha}  \log(1/h)^{|\alpha_1|} \left\|\frac{\partial y^{s_m}}{\partial y^p}(y^p)\right\|^{|\alpha_2|}  \times\left\|N_\infty\left(Q_{s_m}^{(k_{s_m},l_{s_m})}(\partial^{\alpha_2}_y P_{s_m-1})_{y^{s_m}}\right)\right\|_{L^2( \d y^{s_m})}.
			\end{split}
		\end{align}
		Then, we use the estimate on coefficients of $Q^{(k,l)}$ obtained in the proof of proposition \ref{Qjkl} and comparison between the iterated metaplectic operator and the linearized flow in appendix \ref{estim_metap} to obtain:
		\begin{align*}
			\Bigl\|N_\infty\Bigl(Q_{s_m}^{(k_{s_m},l_{s_m})}&(P_{s_m-1})\Bigr)\Bigr\|_{L^2}
			\\&\leq  C\sup_{\substack{|\alpha_1|+|\alpha_2|\leq 2l_{s_m}  \\ |\beta_1|+|\beta_2|=k_{s_m}}} \left\|N_\infty\left(\partial^{\alpha_1+\beta_1}_y (P_{s_m-1})_y)\right)\right\|_{L^2} \left\|\kappa^{\text{c},(s_m)}\right\|^{|\beta_1|+2|\alpha_2|+3|\beta_2|}
			\\&\leq  C\sup_{\substack{|\alpha_1|+|\alpha_2|\leq 2l_{s_m}  \\ |\beta_1|+|\beta_2|=k_{s_m}}} \left\|N_\infty\left(\partial^{\alpha_1+\beta_1}_y (P_{s_m-1})_y)\right)\right\|_{L^2} \left(\sigma_c^{(s_m)}\right)^{|\beta_1|+2|\alpha_2|+3|\beta_2|}. 
		\end{align*}
		We then need to study $\partial_y^\alpha(P_{s_m-1})_y$ and the norm of its coefficients by this is done by induction.
		\medbreak
		Overall, we see that the worse case scenario  in (\ref{d alpha Q}) is for the derivatives to hit onto coefficients of $P_\text{ini}$ which leads to the following estimate for $\mathcal{N}$ defined in \eqref{d alpha Q}:
		\begin{equation*} 
			\mathcal{N} \leq C \sup_{\substack{|l_1|+|l_2|\leq 2l , \\ |k_1|+|k_2|=k}} \left( \left\|\frac{\partial y^{s_m}}{\partial y^p}(y^p)\right\||\Gamma^{\text{ini}}|h^{-1/2}\right)^{l_1+k_1+|\alpha|} 
			\left(\sigma_{\text{c}}^{(p)}\right)^{|k_1|+2|l_2|+3|k_2|}.
		\end{equation*}
		Then we use, the estimate of lemma \ref{der_contract} to control $\left\|\frac{\partial y^{s_m}}{\partial y^p}(y^p)\right\|$ and lemma \ref{1st_methodK} for $\left|\Gamma^{\text{ini}}\right|$.
		\medbreak
		Finally, we have to estimate of the number of terms in the sum.
		It is not difficult to see that there are at most $i^{k+l}$ terms if $i>0$. This can be seen using to the injection:
		\begin{equation*} 
			\begin{array}{rcl}
				I_{k,l} &\to&  \llbracket 1,i \rrbracket^k\times \llbracket 1,i \rrbracket^l\\
				(k_q),(l_q) &\mapsto & (\alpha_1,\dots,\alpha_k,\beta_1,\dots,\beta_l),
			\end{array}
		\end{equation*}
		where 
		\begin{equation*}
			\alpha_m=\inf \left\{ L, \sum_{q=1}^L k_q \geq m\right\},\quad \beta_m=\inf \left\{ L, \sum_{q=1}^L l_q \geq m\right\}.
		\end{equation*}
	\end{proof}
	\begin{Remark}
		This result implies, in particular, that all the $u_p^{(k,l)}$ are in $L^2\left(\R^d\right)$.
	\end{Remark}

	This result, along with the preemptive result, let us see that our state is still in the $S_{\delta,\nu}$ even after propagating for long times: 
	\begin{Lemma}\label{Ninf_poly}
		For all $p\in \llbracket n_s,n\rrbracket,$ and $(k,l)$ such that $k+2l<2N,$ $\forall \alpha\in \N^{d_\perp},$
		\begin{equation*} 
			h^{k/2+l} u_p^{(k,l)}\in S_{\delta_p,\nu},
		\end{equation*}
		for $\nu_p=pt_0\left(\lambda^{\text{c}}+\epsilon_3\right)/|\log h|,\quad \delta_p=\frac{1}{2}\left(1-p\frac{t_0}{|\log h|}\nu^{\text{min}}_{\perp}\right)+$.
		\medbreak
		Moreover, we have the following bound
		\begin{align*}
			\left\|N_\infty\left(\partial^\alpha_y P_p^{(k,l)}\right)\right\|_\infty &\leq C p^{2N}  \sup_{\substack{|l_1|+|l_2|\leq 2l , \\ |k_1|+|k_2|=k}} \left( h^{-1/2}e^{-pt_0(\nu_{\text{min}}+\epsilon)}\right)^{l_1+k_1+|\alpha|} \left(\sigma_{\text{c}}^{(p)}\right)^{|k_1|+2|l_2|+3|k_2|} 
			\\&\times \left(J_u^{(pt_0)}(\tilde{\rho})\right)^{-1/2} \left\|N_\infty\left(P_\text{ini}\right)\right\|_{L^2},
		\end{align*}
		and the essential support of $P_p^{(k,l)}$ is included in a ball centered at $\rho^k$ of size $CJ_u^{(t)}(\tilde{\rho})\sqrt{h}$.
	\end{Lemma}
	
	\begin{proof}
		From the study of the iterated metaplectic operator, we have seen that quantities for the general case can be estimated using the linearized dynamic along $K^\delta$ associated with points $\rho\in K$ (see section \ref{estim_metap}). But we have a control on this dynamic with $\lambda^{\text{c}}$:
		\begin{equation*} 
			\left\| d^\text{c}_\rho \Phi^{(p)}\right\|\leq Ch^{-p\left(\lambda^{\text{c}}+\epsilon_3\right)/|\log h|}.
		\end{equation*}
		This shows that we can take the $\nu$ announced above.
		\medbreak
		The $S_\delta^{-\infty}$ behavior can be shown using the same kind of study as above but on 
		\newline $\left\|N_\infty \left(\partial^\alpha_y P_p^{(k,l)}\right)\right\|_\infty$ rather than its $L^2$ norm. For the sake of simplicity, we will deal with the case $\alpha=0$, the behavior of derivatives can be deduced from the same computations as in the previous proof.
		\medbreak
		Firstly, the $\jp{y}^{-\infty}$ behavior comes from the fact that we always have an exponential in every term of the expansion, namely $e^{-y^{n_s}(y^p) \Gamma^{(\text{ini})}y^{n_s}(y^p)/h}$ . This term is also responsible for the property on the essential support.
		\medbreak 
		We then follow the previous proof and notice that the main difference comes from the fact that the determinant factors obtained when choosing the main order term in the expansion (i.e. for any $k_q=0$, $l_q=0$) do not disappear (we do not have to do a change of variable as we look at $L^\infty$ norms rather than $L^2$ ones). We will keep a factor 
		\begin{equation*} 
			\underset{s\notin I_{k,l}}{\prod} \left|\det\left(\frac{\partial y^{s}}{\partial y^{s+1}}(y^{s+1})\right)\right|  \sim  \left|\det\left(\frac{\partial y^{0}}{\partial y^{p}}(y^{p})\right)\right|=O \left(J_u^{(pt_0)}(\rho)^{-1/2}\right),
		\end{equation*}
		as shown in lemma \ref{det y}.
		\medbreak
		If we take a look at other factors composing the coefficients of $Q_p^{(k_p,l_p)}\circ \dots \circ Q_1^{(k_1,l_1)}(P_\text{ini})_y$,
		we see that they can be written as some 
		\begin{equation*} 
			\left|\det\left(\frac{\partial y^{0}}{\partial y^{p}}(y^{p})\right)\right| a^{(s_1)}(y^{s_1}(y^p))\dots a^{(s_m)}(y^{s_m}(y^p)) a^{(0)}(y^0(y^p)) e^{-y^0(y^p) \Gamma^{(\text{ini})} y^0(y^p)}{h},
		\end{equation*}
		where $a^{(s_j)}$ involves the amplitude, the phase and the coefficients of the stationary phase of the $s_j$-th Fourier integral operator and $a^{(0)}$ involves coefficients of $P_\text{ini}$ and their derivatives.
	\end{proof}
	\begin{Remark}
		As we can see from the value of $\delta_p$ and $\nu_p$ obtained here, from the point of view of $u_p^{(k,l)}$, there is a smoothing effect on the amplitude as $\delta_p$ gets smaller with the propagation but on the opposite, $\nu_p$ tends to grow. Note that this smoothing only happens for the amplitude: when considering derivatives in $y$ of the function $v_\gamma$ of theorem \ref{th_K} and theorem \ref{th_trapped_set}, they can fall not only on the polynomials $P_p^{(k,l)}$ but also on the $\Gamma$ factor (or functions defining $\mathcal{I}$), which is estimated differently.
	\end{Remark}
	\medbreak
	\subsubsection{Control of the remainder and consequences} 
	Using lemma \ref{norm_2ndmethod} and the iterative formula for the rest $r^{(p)}_N$, we can deduce a control of the remainder:
	\begin{Lemma}
		\begin{equation*} 
			\left\|r^{(p)}_N\right\|\leq C h^{2N\epsilon}N_\infty\left(P_\text{ini}\right).
		\end{equation*}
	\end{Lemma}
	\begin{proof}
		The equation \ref{def restes} combined with equation \ref{norm R_N} and lemma \ref{norm_2ndmethod} yields:
		\begin{align*}
			\left\|r^{(p)}_N\right\|&\leq \left\|r^{(p-1)}_N\right\|+\sum_{k+l<2N} Ch^{k/2+l+N} \times \\&\qquad\sup_{2l'+k'=2N}\sup_{\substack{|\alpha_1|+|\alpha_2|\leq 2l' , \\ |\beta_1|+|\beta_2|=k'}}\|\kappa^\text{c}\|^{|\beta_1|+2|\alpha_2|+3|\beta_2|+d_{\scalerel*{\parallel}{\perp}}}
			\left\|N_\infty\left(\partial^{\alpha_1+\beta_1}_y P_{p-1}^{(k,l)}\right)\right\|_{L^2} 
			\\&\leq  \left\|r^{(p-1)}_N\right\|+\sum_{k+l<2N} Ch^{k/2+l+N}((p-1)^{2N}+1)  \sup_{2l'+k'=2N}\sup_{\substack{|\alpha_1|+|\alpha_2|\leq 2l' , \\ |\beta_1|+|\beta_2|=k'}}\|\kappa^\text{c}\|^{|\beta_1|+2|\alpha_2|+3|\beta_2|+d_{\scalerel*{\parallel}{\perp}}} 
			\\&\qquad \times\sup_{\substack{|l_1|+|l_2|\leq 2l , \\ |k_1|+|k_2|=k}} \left(h^{-1/2} e^{-pt_0(\nu_{\text{min}}+\epsilon)}\right)^{|l_1|+|k_1|+|\alpha_1+\beta_1|} \left(\sigma_{\text{c}}^{(p)}\right)^{|k_1|+2|l_2|+3|k_2|},
		\end{align*}
		which gives
		\begin{align*}
			\left\|r^{(p)}_N\right\|\leq &\left\|r^{(p-1)}_N\right\|+C|\log h|\sum_{k+l<2N} h^{k/2+l}  \max(h^{-3\nu_p k}, h^{-(1/2-\delta_p)k}) \max( h^{-4l\nu_p},h^{-2l(1/2-\delta_p)})
			\\&\times h^N \sup_{2l'+k'=2N} \max(h^{-2(1/2-\delta_p)l'}, h^{-2l'\nu_p} ) \max(h^{-(1/2-\delta_p-\nu_p)k'},n^{-3k'\nu_p}).
		\end{align*}
	\end{proof}
	These results imply that 
	$T_n \dots T_1 \varphi_{\kappa_{\alpha}^{-1}(\rho)}$ has the form mentionned in thereom \ref{th_trapped_set}. This is enough to conclude for the case of  theorem \ref{th_K} using the few simplifications pointed out through out the paper. However, according to lemma \ref{introduce charts}, to get the result of theorem \ref{th_trapped_set}, we still have to apply the operator $\mathcal{U}_\beta (\mathcal{U}^n)^*$ to $T_n \dots T_1 \varphi_{\kappa_{\alpha}^{-1}(\rho)}$. 
	\medbreak
	It turns out that adding this last operator is essentially harmless: the operator $\mathcal{U}_\beta (\mathcal{U}^n)^*$ is also a local Fourier Integral Operator that verifies the assumption of section \ref{section2nd}. Indeed, the main assumption is that our isotropic manifold $\mathcal{I}^n$ is sent onto another isotropic manifold that is still projectable in $y$. This true for $\mathcal{U}_\beta (\mathcal{U}^n)^*$: the isotropic manifold $\mathcal{I}^n$ is really close to $\mathcal{I}_m$ thanks to the estimates of section \ref{phase}, hence its image by $\kappa^{(n)}$ is close to $W^u_{\tilde{\rho}^n}$. But as $\kappa_\beta$ is adapted to $K^\delta$ in the sense of definition \ref{K adapted}, it sends $W^u_{\tilde{\rho}^n}$ onto a manifold that is projectable in $y$. As a consequence, the same is true for $\kappa_\beta (\kappa^{(n)})^{-1}(\mathcal{I}^n)$.
	\medbreak
	The only real change of having in section \ref{section2nd} a classical dynamics that does not come from charts adapted to points is that we do not have lemma \ref{compare dcF} that relates the quadratic phase of the metaplectic operator to the linearized classical dynamics. This lemma was particularly useful when we applied the method of section \ref{section2nd} a large number $n$ of times, to ensure that all the quantities that appear grew in a controlled way with $n$. However, if we only need to apply the method of section \ref{section2nd} once, we do not need precise information on the application $d^c\mathfrak{F}$ that is quantized by the metaplectic operator.

	\appendix
	
	\section{Existence of adapted coordinates}\label{Appendix coord}
	In this appendix, we aim at proving that there exists coordinates satisfying the definition \ref{coord}.
	\medbreak
	Let us start by some facts about the symplectic geometry of $r-$normal hyperbolicity:
	\begin{Lemma}\cite[Lemma 4.1]{Steph}\label{est iso}
		Let us make the normal hyperbolicity assumption (\ref{H1}), then if $\omega$ is the canonical symplectic form on $\R^{2d}$, we have
		\begin{equation*} 
			\forall \rho\in K^\delta,\ \rho'\in W^{u/s}(\rho), \quad \omega_{\rho'} \vert_{T_{\rho'}W^{u/s}_{\rho}}=0,
		\end{equation*}
		i.e. $W^{u/s}_\rho$ are isotropic manifolds of $\R^{2d}$.
		\medbreak
		Moreover, if we assume $r$-normal hyperbolicity (\ref{H1'}), then for any $\rho\in K^\delta$, $V^{u/s}_\rho$ are included in the symplectic orthogonal of $T_{\rho} K^\delta$:
		\begin{equation*} 
			\forall u\in V^{u/s}_\rho,\ v\in T_\rho K^\delta,\quad \omega_\rho(u,v)=0.
		\end{equation*}
	\end{Lemma}
	\begin{proof}
		These properties rely on the preservation of the symplectic structure $\omega$ by the flow $\Phi^t$.
		\medbreak
		Let $\rho\in K^\delta$, we start by showing that $V^{u/s}_\rho$ are isotropic subspaces of $T_{\rho} \R^{2d}$:
		\newline
		Let $u$, $v\in V^{u/s}_\rho$ then
		\begin{equation*} 
			\omega_\rho(u,v)=\omega_{\Phi^{\mp t}(\rho)}\left(\d \Phi^{\mp t}_\rho(u),\d \Phi^{\mp t}_\rho(v)\right)\to 0, \quad t\to +\infty,
		\end{equation*}
		as both terms in $\omega_{\Phi^{\mp t}(\rho)}$ decay exponentially fast as $t$ goes to $+\infty$ by assumption.
		\newline
		To show the isotropy of the tangent space to $W^{u/s}_\rho$ at a point $\rho'\in W^{u/s}_\rho$, we use similarly that the flow is contracting along $W^{u/s}_\rho$ and that all points are past/forward asymptotic to $\rho$.
		\medbreak
		If we rather assume (\ref{H1'}) and consider $u\in V^{u/s}_\rho, v\in T_\rho K^{\delta}$ then we also get 
		\begin{equation*} 
			\omega_\rho(u,v)=\omega_{\Phi^{\mp t}(\rho)}\left(\d \Phi^{\mp t}_\rho(u),\d \Phi^{\mp t}_\rho(v)\right)\to 0, \quad t\to +\infty,
		\end{equation*}
		as the left-hand term inside $\omega_{\Phi^{\mp t}(\rho)}$ decay exponentially with a rate big enough to overcome the potential exponential growth of the right-hand term inside $\omega_{\Phi^{\mp t}(\rho)}$.
	\end{proof}
	Recall that $K^\delta$ is symplectic by assumption. Let us pick $\rho\in K^\delta$ that will be the center of chart in definition \ref{coord}.
	We start by considering coordinates $\kappa:\text{neigh}(\rho)\to \text{neigh}(0)$ (non necessary symplectic) such that:
	\begin{enumerate}
		\item\begin{equation} \label{coord K}
			\kappa\left(K^\delta\cap U_{\rho}\right)=\left\{(x',0,\xi',0),(x',\xi')\in \R^{2d_{\scalerel*{\parallel}{\perp}}}\right\}\cap V_{\rho},
		\end{equation}
		\item\begin{equation}\label{coord W^u}  
			\kappa\left(W^u_\rho\cap U_{\rho}\right)=\left\{(0,y',0,0), y'\in \R^{d_\perp}\right\}\cap V_{\rho},
		\end{equation}
		\item\begin{equation}\label{coord W^s}  
			\kappa\left(W^s_\rho\cap U_{\rho}\right)=\left\{(0,0,0,0), \eta'\in \R^{d_\perp}\right\}\cap V_{\rho},
		\end{equation}
		\item $d\kappa$ is symplectic when computed on points of $K^\delta$, $W^u_{\rho}$ or $W^s_{\rho}$.
	\end{enumerate}
	The last point is obtained by adjusting the vector fields associated with the coordinates on these submanifolds. The fact that we can do it is ensured by lemma \ref{est iso}.
	\medbreak
	From these coordinates, we want to construct symplectic ones that satisfy similar properties using Moser's trick. We introduce the symplectic form
	$\omega_1=(\kappa^{-1})^* \Omega_0$ (where $\Omega_0$ is the canonical symplectic form on $\R^{2d}$) and the original one $\omega_0=\sum d x_i\wedge d \xi_i+d y_i \wedge d \eta_i$ on $\R^{2d}$ and define $\omega_t=(1-t)\omega_0+t\omega_1$.
	\medbreak
	By Poincaré's lemma, there exists a 1-form $\alpha$ on a neighborhood $\mathcal{U}$ of $\rho$ such that 
	\begin{equation*} 
		\dot{\omega_t}=\omega_1-\omega_0=\d \alpha,\quad \text{ on } \mathcal{U}.
	\end{equation*}
	As suggested by Moser's trick, we look for a vector field $X_t$ such that 
	\begin{equation*} 
		\iota_{X_t} \omega_t+\alpha=0.
	\end{equation*}
	From the choice of our coordinates $\kappa$, we can see that $\alpha=0$ on $K^\delta$, $W^u_{\rho}$, $W^s_{\rho}$. Hence by non-degeneracy of $\omega_t$,
	\begin{equation*}
		X_t \text{ vanishes on } K^\delta, W^u_{\rho}, W^s_{\rho}.
	\end{equation*}
	We may shrink the neighborhood of $\rho$ so that the flow is defined for all $0\leq t \leq 1$. Now, we set $\phi$ the time-1 map of the flow of $X_t$ which is a diffeomorphism. We compose this diffeomorphism with the chart and check that the properties \ref{coord K}, \ref{coord W^u}, \ref{coord W^s}.

	\section{\texorpdfstring{Study of the classical dynamics associated with the second method when $\rho\notin K^{\delta}$}{Study of the classical dynamics associated with the second method when rho not in K delta}}\label{L_j}
	In this appendix, we would like to show how to obtain a control over quantities involving the classical dynamics that appear during the successive applications of the second method in section \ref{app_2nd}. This is an extension of the proposition 5.1 of \cite{NZ_pression_topo} called ``evolution of Lagrangian leaves''.
	\medbreak
	We will study the evolution of the manifolds $\mathcal{I}_j$ and the composition of parallel metaplectic operators. We can check that they behave well under composition in the case when $\rho\in K$: they respectively give the evolution of the unstable manifold, some parallel part of its linearization. However, the situation is not as clear in the case where $\rho$ only lies in a $\sqrt{h}$ neighborhood of $K^\delta$: this question is the object of this work.
	\medbreak
	\subsection{\texorpdfstring{Control of the functions describing $\mathcal{I}_j$}{Control of the functions describing Ij}}\label{phase}
	Here we would like to show estimates on derivatives of  functions describing the manifold $\mathcal{I}_j$. We will start by showing that they are uniformly bounded (which was a requirement to apply the method).
	\medbreak
	The first step to do so is to show expansiveness in time $1$ in the unstable direction.
	\medbreak
	Let us choose the time $t_0$ big enough so that the sub-matrices of the linearized dynamic at the center points $\tilde{\rho}$ and its iterates verifies (see definition \ref{coord}):
	\begin{equation*} 
		\d \Phi^{(j)}\left(\tilde{\rho}^j\right)=\begin{pNiceArray}{cccc}[last-row,last-col]
			A_j & 0 & E_j & 0 & x\text{ components}
			\\ 0 & B_j & 0 & 0 & y \text{ components}
			\\ F_j& 0 & C_j &0 & \xi \text{ components}
			\\ 0 &0 &0  &D_j & \eta \text{ components}
			\\ \partial_x & \partial_y & \partial_\xi & \partial_\eta &
		\end{pNiceArray},
	\end{equation*}
	with $\|B_j^{-1}\|,\|D_j\|< \nu <1$ and $\|A_j\|,\|C_j\|,\|D_j\|,\|E_j\|<\mu<\nu^{-1}$.
	\medbreak
	Let us write $\mathcal{I}_{n_s}=\left\{(q^{n_s}_x,y^{n_s},p^{n_s}_x,p^{n_s}_y), y^{n_s}\in D_{\epsilon_1}\right\}$ the manifold associated with the time $n_s$ at we change methods. We then show that $\mathcal{I}_j=\kappa^{(j)}\Phi^{(j-n_s)t_0}(\kappa^{(n_s)})^{-1}\left(\mathcal{I}_{n_s}\right)$ can be written in a similar fashion for any $j\in \llbracket n_s,n \rrbracket$ that might depend on $h$: $\mathcal{I}_j=\bigl\{(\overline{x^j}(y^j),y^j,\overline{\xi^j}(y^j), \overline{\eta^j}(y^j)),$\newline $ y^j\in D_{\epsilon_1}\bigr\}$ with derivatives of functions parameterizing it uniformly bounded.
	\medbreak
	Assume that $\mathcal{I}_j=\left\{(\overline{x^j}(y^j),y^j,\overline{\xi^j}(y^j), \overline{\eta^j}(y^j)), y^j\in D_{\epsilon_1} \right\}$ and the derivatives are bounded.
	We can write its image under $\Phi^{(j)}$ as $(x^{j+1},y^{j+1},\xi^{j+1},\eta^{j+1})=\Phi^{(j)}(\overline{x^j}(y^j),y^j,\overline{\xi^j}(y^j), \overline{\eta^j}(y^j))$ verifying
	\begin{align}
		x^{j+1}(y^j)&=\quad\ \ A_j \overline{x^j}(y^j)+ E_j\overline{\xi^j}(y^j)&+\tilde{\alpha_j}\left(\overline{x^j}(y^j),y^j,\overline{\xi^j}(y^j), \overline{\eta^j}(y^j)\right),
		\\ y^{j+1}(y^j)&=\qquad\qquad\ \ B_j y^j&+\tilde{\beta_j}\left(\overline{x^j}(y^j),y^j,\overline{\xi^j}(y^j), \overline{\eta^j}(y^j)\right) \label{syst_lin},
		\\ \xi^{j+1}(y^j)&=\quad\ \  F_j \overline{x^j}(y^j)+ C_j\overline{\xi^j}(y^j)&+\tilde{\gamma_j}\left(\overline{x^j}(y^j),y^j,\overline{\xi^j}(y^j), \overline{\eta^j}(y^j)\right),
		\\\eta_{j+1}(y^j)&=\qquad\quad\ \ D_j \overline{\eta^j}(y^j)&+\tilde{\delta_j}\left(\overline{x^j}(y^j),y^j,\overline{\xi^j}(y^j),\overline{\eta^j}(y^j)\right),
	\end{align}
	where $\tilde{\alpha_j},\tilde{\beta_j},\tilde{\gamma_j},\tilde{\delta_j}$ are smooth functions such that 
	\begin{equation*} 
		\tilde{\alpha_j}(0,0,0,0)=0 \text{ and } \|\tilde{\alpha_j}\|_{C^1(V_{\Phi^{jt_0}(\rho)})}\leq C \epsilon^\gamma,
	\end{equation*}
	and the same is true for $\tilde{\beta_j},\tilde{\gamma_j},\tilde{\delta_j}$.
	
	As a consequence, taking the $y$ coordinates and derivating with respect to $y^{j}$ gives
	\begin{align*}
		\frac{\partial y^{j+1}}{\partial y^j}&= B_j+\frac{\partial \tilde{\beta_j}}{\partial x^j}\frac{\partial \overline{x^j} }{\partial y^j}+ \frac{\partial \tilde{\beta_j}}{\partial y^j}+\frac{\partial \tilde{\beta_j}}{\partial \xi^j}\frac{\partial \overline{\xi^j} }{\partial y^j}+\frac{\partial \tilde{\beta_j}}{\partial\eta^j}\frac{\partial \overline{\eta^j }}{\partial y^j}
		\\&=\left(B_j+ O(\epsilon^\gamma (1+\gamma_1))\right),
	\end{align*}
	where $\gamma_1>0$ is a constant controlling the derivatives of the functions $\overline{x^j},\overline{\xi^j},\overline{\eta^j}$.
	\newline
	Recall that $\|B_j^{-1}\|< \nu <1$ so that if $\epsilon$ is small enough, the matrix above is still expanding. Hence, we can use the $y^{j+1}$ coordinate to parameterize $\mathcal{I}_{j+1}=\d \Phi^{(j)}(\mathcal{I}_j)$ with the implicit function theorem: the app $y^j\mapsto y^{j+1}$ is invertible and its inverse verifies:
	\begin{equation*} 
		\left\|\frac{\partial y^j}{\partial y^{j+1}} \right\|\leq \nu.
	\end{equation*}
	Moreover, we can differentiate the functions $\overline{x^{j+1}}(y^{j+1}),\overline{\xi^{j+1}}(y^{j+1}),\overline{\eta^{j+1}}(y^{j+1})$ and obtain (we can treat similarly $\overline{\xi^{j+1}}$ and $\overline{\eta^{j+1}}$):
	\begin{equation*} 
		\frac{\partial \overline{x^{j+1}}}{\partial y^{j+1}}= \left(\bigl(A_j+\partial_x\tilde{\alpha}\bigr) \frac{\partial \overline{x^j}}{\partial y^j}+\partial_y\tilde{\alpha}+\bigl(E_j+\partial_\xi \tilde{\alpha}\bigr) \frac{\partial \overline{\xi^j}}{\partial y^j}+ \partial_\eta \tilde{\alpha}\frac{\partial \overline{\eta^j}}{\partial y^j} \right) \frac{\partial y^j}{\partial y^{j+1}},
	\end{equation*}
	and for $\epsilon$ small enough we obtain 
	\begin{equation*} 
		\left\| \frac{\partial \overline{x^{j+1}}}{\partial y^{j+1}}\right\| \leq \gamma_1,
	\end{equation*}
	thanks to $\|A_j\|\leq \mu<\nu^{-1}$ and $\left\|\frac{\partial y^j}{\partial y^{j+1}} \right\|\leq \nu$.
	\medbreak
	For the study of higher derivatives of  $\overline{x^{j+1}},\overline{\xi^{j+1}},\overline{\eta^{j+1}}$, we work by induction on the degree $l$ of differentiation and we differentiate $l$ times the equation 
	\begin{equation*} 
		\Phi^{(j)}\left( \overline{x^{j}}(y^j),y^j,\overline{\xi^{j}}(y^j),\overline{\eta^j}(y^j)\right)=\left(\overline{x^{j+1}}(y^{j+1}),y^{j+1},\overline{\xi^{j+1}}(y^{j+1}),\overline{\eta^{j+1}}(y^{j+1})\right).
	\end{equation*}
	It gives for instance for $\overline{x^j}$: 
	\begin{equation*} 
		\frac{\partial^l \overline{x^{j+1}}}{(\partial y^{j+1})^l}=    \left(\bigl(A_j+\partial_x \tilde{\alpha_j}\bigr) \frac{\partial^l \overline{x^j}}{(\partial y^j)^l}+\partial_y\tilde{\alpha_j}+\bigl(E_j+\partial_\xi \tilde{\alpha}\bigr) \frac{\partial^l \overline{\xi^j}}{(\partial y^j)^l}+\partial_\eta \tilde{\alpha_j} \frac{\partial^l \overline{\eta^j}}{(\partial y^j)^l}   \right)  \left(\frac{\partial y^j}{\partial y^{j+1}}\right)^l +P_{l,k},
	\end{equation*}
	with $P_{l,j}$ a polynomial, in lower order derivatives of the functions, with coefficients uniformly bounded in $j$. 
	As we only consider values in a $\epsilon$ disk, there exists a constant $C_l(\gamma_{l-1})>0$ independent of $j$ such that:
	\begin{equation*} 
		\left\|\frac{\partial^l \overline{x^{j+1}}}{(\partial y^{j+1})^l}\right\| \leq \nu^{l-1} \nu\mu \max\left(\left\|\frac{\partial^l \overline{x^j}}{(\partial y^j)^l}\right\|, \left\|\frac{\partial^l \overline{\xi^j}}{(\partial y^j)^l}\right\|, \left\|\frac{\partial^l \overline{\eta^j}}{(\partial y^j)^l}\right\|\right) +C_l(\gamma_{l-1}).
	\end{equation*}
	Hence, if $\|\overline{x^j}\|_{C^l}, \|\overline{\xi^j}\|_{C^l}, \|\overline{\eta^j}\|_{C^l}\leq \gamma_l$ for $\gamma_l$ large enough, we have $\left\|\frac{\partial^l \overline{x^{j+1}}}{(\partial y^{j+1})^l}\right\|\leq \nu \mu \gamma_l<\gamma_l$. 
	\medbreak
	Finally, we would to show the more precise estimates described in theorem \ref{th_trapped_set}. For this, we will use the preliminary section \ref{proof_dyn_class}, controlling the behavior of points under the classical dynamics.
	\newline
	Indeed, we know that the points of $\mathcal{I}_{n}$ relevant in theorem \ref{th_trapped_set} enter in the framework of lemma \ref{proof_dyn_class}. As a consequence, if $y^n\in \R^{d_\perp}$ such that $u_\gamma(y)$ is not $O(h^\infty)$ then 
	\begin{equation*} 
		|\overline{x^n}(y^n)|, |\overline{\xi^n}(y^n)|,|\overline{\eta^n}(y^n) |\leq C \sigma_c^{(t)} h^{\tau}h^{0-\cdot}.
	\end{equation*}
	Now that we also now that all the derivatives of $\overline{x^n},\overline{\xi^n}$ and $\overline{\eta^n}$ are bounded, we can show that their derivatives must satisfy the similar estimates by contradiction: assume that for instance 
	\begin{equation*} 
		\|\overline{x^n}\|_\infty\leq Ch^{\alpha} \text{ and } \|\partial_i\overline{x^n}_j\|_\infty\geq Ch^{\alpha-\beta}, \alpha,\beta>0. 
	\end{equation*}
	Then, using the fact that if $f:\R\to\R$ is smooth and $f,f''$ are bounded then $\|f'\|_{\infty} \leq \sqrt{2 \|f\|_{\infty} \|f''\|_{\infty}}$, we get (for a different $C>0$)
	\begin{equation*} 
		\| \partial^2_i\overline{x^n}_j\|_{\infty} \leq C h^{\alpha-2\beta}.
	\end{equation*}
	Iterating this results until $\alpha-k\beta$ becomes negatives gives a contradiction with $\| \partial^k_i\overline{x^n}_j\|_{\infty}\leq C$.

	\subsection{\texorpdfstring{Study the derivatives of $y^i$ with respect to $y^{i+1}$,link with the Lyapunov exponent}{Study the derivatives of yi with respect to y i+1,link with the Lyapunov exponent}}
	
	To do this we come back to the equation (\ref{syst_lin}) and study its derivative:
	\begin{align*}
		\frac{\partial y^{j+1}}{\partial y^j}&= B_j+\frac{\partial \tilde{\beta_j}}{\partial x^j}\frac{\partial \overline{x^j} }{\partial y^j}+ \frac{\partial \tilde{\beta_j}}{\partial y^j}+\frac{\partial \tilde{\beta_j}}{\partial \xi^j}\frac{\partial \overline{\xi^j} }{\partial y^j}+\frac{\partial \tilde{\beta_j}}{\partial\eta^j}\frac{\partial \overline{\eta^j }}{\partial y^j}
		\\&=\left(B_j+ O(\|\overline{x^j}(y^j),y^j,\overline{\xi^j}(y^j),\overline{\eta^j}\|)\right),
	\end{align*}
	where we have used that $\overline{x^j},\overline{\xi^j},\overline{\eta^j}$ have bounded derivatives, $d\tilde{\beta_j}(0,0,0,0)=0$ and the mean value theorem.
	\medbreak
	We now compute:
	\begin{align*}
		\det\left(\frac{\partial y^n}{\partial y^{n_s}}\right)&= \det \prod_{k=n_s}^{n-1} \left(B_k+ O(\|\overline{x^j}(y^j),y^j,\overline{\xi^j}(y^j),\overline{\eta^j}\|)\right)
		\\&=\det\left(\prod_{k=n_s}^{n-1} B_k\right) \prod_{k=n_s}^{n-1} \det(I+O(\|\overline{x^j}(y^j),y^j,\overline{\xi^j}(y^j),\overline{\eta^j}\|)).
	\end{align*}
	To conclude, we notice that $\sum_{k=n_s}^{n-1}\left\|\overline{x^j}(y^j),y^j,\overline{\xi^j}(y^j),\overline{\eta^j}\right\|$ is bounded independently of $h$: it is the sum of coordinates of the iterates of a point $\rho$ under $\Phi^{t_0}$ which enters in the framework of lemma \ref{proof_dyn_class}, this lemma gives estimates that are sufficient to bound the sum by $O(\epsilon)$.
	\medbreak
	Hence, we obtain:
	\begin{equation*} 
		\det\left( \frac{\partial y^n}{\partial y^n_s}\right)= \det \left(\prod_{k=n_s}^{n-1} B_k\right) (1+O(\epsilon)),
	\end{equation*}
	which gives
	\begin{Lemma}\label{det y}
		\begin{equation*}
			\left|\det\left( \frac{\partial y^{n_s}}{\partial y^n}\right)\right|\leq C(J_u^{((n-n_s)t_0)}(\tilde{\rho}))^{-1}.
		\end{equation*}
	\end{Lemma}
	
	Now that we know that this flow is contracting, we can use similar technique as section \ref{phase} to deduce the behavior of its higher derivatives.
	
	\begin{Lemma}\label{der_contract}
		Let $l\geq 1$ then 
		\begin{equation*} 
			\frac{\partial^l y^{n_s}}{(\partial y^n)^l}= O\left((n-n_s)^{l-1} e^{-(n-n_s)t_0(\nu^{\text{min}}_{\text{hyp}}+\epsilon)}\right).
		\end{equation*}
	\end{Lemma}
	
	\begin{proof}
		The case $l=1$ can be dealt with in a similar fashion to the proof of lemma \ref{compare_linear}.
		\medbreak
		For the case $l=2$, we start from the formula:
		\begin{equation}\label{der_prod}
			\frac{\partial y^{n_s}}{\partial y^n}=\prod_{i=n_s}^{n-1} \frac{\partial y^i}{\partial y^{i+1}}, 
		\end{equation}
		and we differentiate with respect to $y^n$ which yields:
		\begin{equation*} 
			\frac{\partial^2 y^{n_s}}{(\partial y^n)^2}=\sum_{j=n_s}^{n-1} \prod_{i\neq j} \left(\frac{\partial y^i}{\partial y^{i+1}}\right) \frac{\partial^2 y^j}{(\partial y^{j+1})^2} \frac{\partial y^j}{\partial y^n}.
		\end{equation*}
		Then, using the fact that $|\frac{\partial y^j}{\partial y^n}|\leq 1 $, we get:
		\begin{equation*} 
			\frac{\partial^2 y^{n_s}}{(\partial y^n)^2} =O\left(n e^{-(n-n_s)t_0(\nu^{\text{min}}_{\text{hyp}}+\epsilon)} \right).
		\end{equation*}
		To obtain the result for a general $l$, we simply differentiate $l-1$ times the equation (\ref{der_prod}), obtaining $n^{l-1}$ terms that can all be estimated by $O(\frac{\partial y^{n_s}}{\partial y^n})$, which gives the result.
	\end{proof}

	\subsection{Verification of the extra assumption (\ref{GammadGamma})} \label{proof_GammadGamma}
	\medbreak
	In this paragraph, we will see how to show that 
	\begin{equation*} 
		\left\| \Im \Gamma_{\scalerel*{\parallel}{\perp}}^{-1/2} \frac{\partial \Gamma_{\scalerel*{\parallel}{\perp}}}{\partial y_\alpha}\Im \Gamma_{\scalerel*{\parallel}{\perp}}^{-1/2}\right\|\leq C |\log h| h^{-2\nu},
	\end{equation*}
	with the $\Gamma_{\scalerel*{\parallel}{\perp}}$ is associated with $d^c F^{(i)}(y)$, we claim that
	the  proof to show the  general estimates of (\ref{GammadGamma}) is similar.
	\medbreak
	We write that
	\begin{align*}
		\left\|\Im\Gamma_{\scalerel*{\parallel}{\perp}}^{-1/2}\partial\Gamma_{\scalerel*{\parallel}{\perp}}  \Im\Gamma_{\scalerel*{\parallel}{\perp}}^{-1/2}\right\|^2&=\text{Tr } (\Im\Gamma_{\scalerel*{\parallel}{\perp}}^{-1}\partial\Gamma_{\scalerel*{\parallel}{\perp}}\Im\Gamma_{\scalerel*{\parallel}{\perp}}^{-1}\partial\Gamma_{\scalerel*{\parallel}{\perp}})
		\\&=\text{Tr} ( M \overline{M}^T (\partial N M^{-1} +N \partial (M^{-1}))M \overline{M}^T (\partial N M^{-1} + N\partial (M^{-1}))),
	\end{align*}
	where $M=A+iB$, $N=C+iD$ for $\begin{pmatrix} A & B \\ C & D \end{pmatrix}$ the symplectic matrix associated with $\Gamma_{\scalerel*{\parallel}{\perp}}$.
	\medbreak
	Expanding this expression gives four terms:
	\begin{itemize}
		\item A first term is given by
		\begin{align*}
			|\text{Tr} ( M \overline{M}^T \partial N M^{-1}M \overline{M}^T \partial N M^{-1} )|&=|\text{Tr } (\overline{M}^T\partial N \overline{M}^T \partial N)|
			\\&\leq \|M\|^2 \|\partial N\|^2
			\\&\leq C h^{-4\nu} |\log h|^2.
		\end{align*}
		Hence its absolute value is bounded by $C h^{-4\nu} |\log h|^2$.
		\item The cross-product gives
		\begin{align*}
			\left|\text{Tr} ( M \overline{M}^T N \partial (M^{-1})M \overline{M}^T \partial N M^{-1} )\right|&\leq\left|\text{Tr }( \overline{N}^TM \partial (M^{-1})M \overline{M}^T \partial N  )\right|
			\\&\quad+\left|-2i\text{Tr } (\partial (M^{-1})M \overline{M}^T \partial N M^{-1} )\right|
			\\&\leq\left|-\text{Tr }( \overline{N}^T \partial M \overline{M}^T \partial N )\right|+\left|2i \text{Tr }( M^{-1} \partial M \overline{M}^T \partial N )\right|
			\\&\leq C (\|N\|\|M\|\|\partial M\| \|\partial N\|+\|M\|\|M^{-1}\|\|\partial M\| \|\partial N\|)
			\\&\leq C h^{-4\nu} |\log h|^2,
		\end{align*}
		where we used that $\overline{M}^T N=\overline{N}^TM -2i I_{d_{\scalerel*{\parallel}{\perp}}}$. This gives the same bound $C h^{-4\nu} |\log h|^2$.
		\item The other cross-product gives the same term.
		\item Finally we also have
		\begin{equation*}
		\begin{multlined}
			|\text{Tr }(M \overline{M}^TN \partial M^{-1}M \overline{M}^T N \partial (M^{-1})) |=
			\\|	\text{Tr }(M \overline{N}^T M \partial (M^{-1})M \overline{M}^T N \partial (M^{-1}))  -2i \text{Tr }(M \partial (M^{-1})M \overline{M}^T N \partial(M^{-1}))|
			\\=|-\text{Tr } (M \overline{N}^T \partial M \overline{N}^T M \partial (M^{-1})) +2i \text{Tr } (M \overline{N}^T \partial M \partial (M^{-1})) +2i \text{Tr } (\partial M \overline{M}^T N \partial (M^{-1}))|
			\\\shoveleft{= |\text{Tr } ( \overline{N}^T \partial M \overline{M} \partial M) -4i \text{Tr } ( \overline{N}^T \partial M M \partial M) +8 \text{Tr } (\partial M M^{-1} \partial M M^{-1})|}
			\\\leq \shoveleft{\phantom{}}  C( \|N\| \|M\| \|\partial M\|^{2}+\|M^{-1}\|^{2} \|\partial M\|^{2}),\shoveright{\phantom{}}
		\end{multlined}
		\end{equation*}
		and the same estimation by $C h^{-4\nu} |\log h|^2$.
	\end{itemize}
	
	\printbibliography
	
\end{document}